\documentclass[12pt]{amsart}
\usepackage{amssymb}
\usepackage{enumitem}
\usepackage{graphicx}
\usepackage[cmtip,all]{xy}
\usepackage{bbm}
\usepackage{comment}
\usepackage{mathrsfs}
\usepackage{mathtools}
\usepackage{xcolor}
\usepackage[dvipdfmx]{hyperref}
\usepackage{autobreak}
\usepackage{etoolbox}

\hypersetup{
 colorlinks,
 linkcolor={blue!50!black},
 citecolor={blue!50!black},
 urlcolor={blue!80!black}
}

\everymath{\displaystyle}

\calclayout

\def\DefineSymbol#1#2{\newcommand{#1}{{\mathrm {#2}}}}
\def\DefineCategory#1#2{\newcommand{#1}{{\mathrm {#2}}}}

\theoremstyle{plain}
	\newtheorem{theorem}{Theorem}[section]
	\newtheorem{lemma}[theorem]{Lemma}
	\newtheorem{proposition}[theorem]{Proposition}
	\newtheorem{corollary}[theorem]{Corollary}
	\newtheorem{conjecture}[theorem]{Conjecture}
	
\theoremstyle{definition}
	\newtheorem{definition}[theorem]{Definition}
	\newtheorem{lemma and definition}[theorem]{Lemma and Definition}
	\newtheorem{example}[theorem]{Example}

	\newtheorem{construction}[theorem]{Construction}
	\newtheorem{notation}[theorem]{Notation}
	
\theoremstyle{remark}
	\newtheorem{remark}[theorem]{Remark}
	\numberwithin{equation}{section}

\DefineSymbol{\pr}{pr}
\DefineSymbol{\ab}{ab}
\DefineSymbol{\can}{can}
\DefineSymbol{\perf}{perf}
\DefineSymbol{\id}{id}
\DefineSymbol{\const}{const}
\DefineSymbol{\op}{op}
\DefineSymbol{\diag}{diag}
\DefineSymbol{\proet}{pro\acute{e}t}
\DefineSymbol{\cond}{cond}
\DefineSymbol{\cont}{cont}
\DefineSymbol{\conti}{conti}
\DefineSymbol{\Cond}{Cond}
\DefineSymbol{\dCond}{dCond}
\DefineSymbol{\disc}{disc}
\DefineSymbol{\Tot}{Tot}
\DefineSymbol{\triv}{triv}
\DefineSymbol{\Bun}{Bun}
\DefineSymbol{\adic}{adic}
\DefineSymbol{\rig}{rig}
\DefineSymbol{\nc}{nc}
\DefineSymbol{\nuc}{nuc}
\DefineSymbol{\cpt}{cpt}
\DefineSymbol{\an}{an}
\DefineSymbol{\alg}{alg}
\DefineSymbol{\la}{la}
\DefineSymbol{\rla}{rla}
\DefineSymbol{\dep}{dep}
\DefineSymbol{\red}{red}
\DefineSymbol{\ex}{ex}
\DefineSymbol{\dR}{dR}
\DefineSymbol{\sdR}{sdR}
\DefineSymbol{\HT}{HT}
\DefineSymbol{\st}{st}
\DefineSymbol{\pst}{pst}
\DefineSymbol{\Sen}{Sen}
\DefineSymbol{\unr}{unr}
\DefineSymbol{\wt}{wt}

\DeclareMathOperator{\Hom}{Hom}

\DeclareMathOperator{\Fun}{Fun}

\DeclareMathOperator{\Ker}{Ker}
\DeclareMathOperator{\Coker}{Coker}

\DeclareMathOperator{\Gal}{Gal}
\DeclareMathOperator{\Lie}{Lie}
\DeclareMathOperator{\Frac}{Frac}

\DeclareMathOperator{\cofib}{cofib}

\DeclareMathOperator{\Fil}{Fil}
\DeclareMathOperator{\gr}{gr}
\DeclareMathOperator{\Fit}{Fit}

\DeclareMathOperator{\rk}{rk}

\DeclareMathOperator{\Sym}{Sym}
\DeclareMathOperator{\Spec}{Spec}
\DeclareMathOperator{\AnSpec}{AnSpec}
\DeclareMathOperator{\Spa}{Spa}

\DeclareMathOperator{\Spm}{Spm}

\DeclareMathOperator{\intHom}{\underline{Hom}}

\newcommand{\lra}{\longrightarrow}

\newcommand{\heart}{\heartsuit}

\newcommand{\hotimes}{\mathbin{\hat{\otimes}}}

\DefineCategory{\Set}{Set}
\DefineCategory{\Ab}{Ab}
\DefineCategory{\Ring}{Ring}
\DefineCategory{\Mod}{Mod}
\DefineCategory{\LMod}{LMod}
\DefineCategory{\Rep}{Rep}
\DefineCategory{\coMod}{coMod}
\DefineCategory{\Alg}{Alg}
\DefineCategory{\Ch}{Ch}
\DefineCategory{\Mon}{Mon}
\DefineCategory{\CMon}{CMon}
\DefineCategory{\PCoh}{PCoh}
\DefineCategory{\Perf}{Perf}
\DefineCategory{\Vect}{Vect}
\DefineCategory{\FP}{FP}
\DefineCategory{\Sch}{Sch}
\DefineCategory{\cAff}{cAff}
\DefineCategory{\Aff}{Aff}
\DefineCategory{\AlgAff}{AlgAff}
\DefineCategory{\AlgAnsp}{AlgAnsp}
\DefineCategory{\AffRing}{AffRing}
\DefineCategory{\AnRing}{AnRing}
\DefineCategory{\AnAdic}{AnAdic}
\DefineCategory{\AffAnAdic}{AffAnAdic}
\DefineCategory{\Ani}{Ani}
\DefineCategory{\CAlg}{CAlg}
\DefineCategory{\Corr}{Corr}
\DefineCategory{\Comm}{Comm}

\newcommand{\Cbb}{\mathbb{C}}

\newcommand{\Gbb}{\mathbb{G}}

\newcommand{\Lbb}{\mathbb{L}}

\newcommand{\Nbb}{\mathbb{N}}

\newcommand{\Qbb}{\mathbb{Q}}

\newcommand{\Zbb}{\mathbb{Z}}

\newcommand{\Acal}{\mathcal{A}}
\newcommand{\Bcal}{\mathcal{B}}

\newcommand{\Dcal}{\mathcal{D}}

\newcommand{\Lcal}{\mathcal{L}}

\newcommand{\Ocal}{\mathcal{O}}

\newcommand{\Rcal}{\mathcal{R}}

\newcommand{\Wcal}{\mathcal{W}}

\newcommand{\Afrak}{\mathfrak{A}}

\newcommand{\Xfrak}{\mathfrak{X}}

\newcommand{\mfrak}{\mathfrak{m}}
\newcommand{\nfrak}{\mathfrak{n}}

\renewcommand{\tilde}{\widetilde}
\renewcommand{\hat}{\widehat}
\renewcommand{\bar}{\overline}

\begin{document}

\title[$p$-adic monodromy theorem]{The $p$-adic monodromy theorem over algebraic-affinoid algebras}
\author{Yutaro Mikami}
\date{\today}
\address{RIKEN, Center for Advanced Intelligence Project AIP, Mathematical Science Team, Wako, Saitama 351-0918, Japan}
\email{yutaro.mikami@riken.jp}
\begin{abstract}
In the previous paper \cite{Mikami24}, motivated by the categorical $p$-adic local Langlands correspondence, the author studied families of $G_K$-equivariant vector bundles over the Fargues-Fontaine curve parametrized by algebraic-affinoid $\Qbb_{p,\square}$-algebras (e.g., $\Qbb_p[T]$).
In this paper, we study the $p$-adic Hodge theoretic properties of such families.
More precisely, we define the notions of Hodge-Tate, de Rham, and semistable representations for such families, and then prove the $p$-adic monodromy theorem (``de Rham'' implies ``potentially semistable'') in this setting.
This is a generalization of the work of Berger-Colmez (\cite{BC08}).
As an application, we prove the classification of families of $G_K$-equivariant line bundles.
While a similar classification was previously obtained in \cite{Mikami24} under a certain freeness condition by relying on the results of Kedlaya--Pottharst--Xiao \cite{KPX14}, the approach in the present paper removes this freeness condition and entirely bypasses their results.
\end{abstract}
\maketitle

\tableofcontents

\addtocontents{toc}{\protect\setcounter{tocdepth}{1}}
\section*{Introduction}

\subsection{Background}
Let $K$ be a finite extension of $\Qbb_p$.
The \textit{(locally analytic) $p$-adic local Langlands correspondence} (for $GL_n(K)$) is a conjectural relation between locally analytic representations of $GL_n(K)$ and $G_K$-equivariant vector bundles over the Fargues--Fontaine curve $X_{\Cbb_p}$ associated to $\Cbb_p^{\flat}$ (or equivalently, $(\varphi,\Gamma_K)$-modules over the Robba ring $\Rcal_K$ (\cite{FF18})).
When $GL_n(K)=GL_2(\Qbb_p)$, the correspondence was constructed by Colmez in \cite{Col10, Col16}.
However, in contrast to the smooth local Langlands correspondence (i.e., $l$-adic, $l\neq p$), even formulating the conjecture is difficult outside the cases of $GL_2(\Qbb_p)$ (and $GL_1$).

Recently, in \cite{FS24}, Fargues--Scholze formulated the geometrization of the smooth local Langlands correspondence using the Fargues--Fontaine curve, and stated the categorical local Langlands correspondence (conjecture).
As a $p$-adic analog of this, Emerton--Gee--Hellmann proposed the categorical $p$-adic local Langlands correspondence in \cite{EGH23}.
This is stated as follows:

\begin{conjecture}[{\cite[Conjecture 6.2.4]{EGH23}}]
There exists an exact functor
$$\Afrak_{GL_n(K)}^{\rig} \colon \Dcal(\ast/GL_n(K)^{\la}) \to \Dcal(\Xfrak_{n,K})$$
satisfying many good properties (e.g., compatibilities with other categorical (local) Langlands correspondences), where
\begin{itemize}
	\item $\Dcal(\ast/GL_n(K)^{\la})$ is the stable $\infty$-category of locally analytic representations of $GL_n(K)$ (\cite{RJRC23}),
	\item $\Xfrak_{n,K}$ is the rigid analytic moduli stack of $G_K$-equivariant vector bundles over $X_{\Cbb_p}$ (\cite[5.1]{EGH23}),
	\item $\Dcal(\Xfrak_{n,K})$ is the stable $\infty$-category of solid quasi-coherent sheaves on $\Xfrak_{n,K}$.
\end{itemize}
\end{conjecture}

Moreover by using $\Bun_{GL_n(K)}^{\la}$, the $p$-adic version of $\Bun_{GL_n(K)}$, which was defined by Ansch\"{u}tz--Le Bras--Rodr\'{\i}guez Camargo--Scholze, we can also state the following conjecture.

\begin{conjecture}\label{CpLLC}
There exists an exact functor
$$\Afrak_{GL_n(K)}^{\rig} \colon \Dcal(\Bun_{GL_n(K)}^{\la}) \to \Dcal(\Xfrak_{n,K})$$
satisfying many good properties.
\end{conjecture}

\begin{remark}
More precisely, it seems necessary to restrict to the subcategories consisting of objects satisfying suitable finiteness conditions, but we will not pursue this in this paper.
\end{remark}

However, it was pointed out in \cite[Remark 6.2.9 (c)]{EGH23} that the functor $\Afrak_{GL_n(K)}^{\rig}$ is not expected to be fully faithful.
When $n=1$, this problem was resolved by Rodrigues Jacinto--Rodr\'{\i}guez Camargo in \cite{RJRC23} by modifying $\Xfrak_{1,K}$.
More precisely, by \cite{KPX14}, there is an explicit description (over the maximal unramified extension field $K_0$ of $\Qbb_p$ in $K$)
$$\Xfrak_{1,K}\cong (\Wcal\times \Gbb_m^{\an})/\Gbb_m^{\an} $$
with trivial $\Gbb_m^{\an}$-action, where $\Wcal$ is the rigid analytic moduli space of continuous characters of $\Ocal_K^{\times}$, and where $\Gbb_m^{\an}$ is the rigid analytic multiplicative group.
They defined the modification $\Xfrak_{1,K}^{\mathrm{mod}}$ of $\Xfrak_{1,K},$ given by 
$$\Xfrak_{1,K}^{\mathrm{mod}}=(\Wcal\times \Gbb_m^{\alg})/\Gbb_m^{\alg},$$
where $\Gbb_m^{\alg}=\AnSpec (K_0[T^{\pm 1}],\Zbb_p)_{\square}$ is an analytic space in the sense of Clausen and Scholze \cite{AG}, and where we regard $\Wcal$ as an analytic space.
\begin{theorem}[{\cite[Theorem 4.4.4]{RJRC23}}]
     There is a natural equivalence of stable $\infty$-categories
     $$\Dcal(\Bun_{K^{\times}}^{\la})\overset{\sim}{\to}\Dcal_{\square}(\Xfrak_{1,K}^{\mathrm{mod}}).$$
\end{theorem}

This result suggests that a test category containing both rigid analytic varieties and ``algebraic varieties'' (e.g., $\mathrm{AnSpec}(\mathbb{Q}_p[T^{\pm 1}], \mathbb{Z})_{\square}$) is necessary to define the modification of the moduli stack $\mathfrak{X}_{n,K}^{\mathrm{mod}}$.
Motivated by this observation, in \cite{Mikami24}, the author defined the test category, and proposed a candidate of $\Xfrak_{n,K}^{\mathrm{mod}}$ as a moduli stack over this test category.

\begin{definition}
\begin{enumerate}
	\item Let $E$ be a finite extension of $\Qbb_p$.
	An analytic $E_{\square}$-algebra $\Acal=(A,A^+)_{\square}$ is called an \textit{algebraic-affinoid analytic $E_{\square}$-algebra} if there is an affinoid $E$-algebra $R$ and a morphism $R\to A$ such that $A$ is relatively discrete and finitely generated $R$-algebra.
	Let $\AlgAff_{E}$ denote the category of algebraic-affinoid analytic $E_{\square}$-algebras.
	\item Let $\Xfrak_{n,K}^{\mathrm{mod}}$ be the sheafification (with respect to a suitable topology) of the presheaf of groupoids on $\AlgAff_{K_0}^{\op}$ that assigns to each $\Acal=(A,A^+)_{\square}\in \AlgAff_{K_0}$ the groupoid of $G_K$-equivariant ``vector bundles'' of rank $n$ over $X_{\Cbb_p}\times \AnSpec\Acal$.
\end{enumerate}
\end{definition}

\begin{remark}
	An algebraic-affinoid analytic $L_{\square}$-algebra is not necessarily Fredholm (cf. \cite[Appendix A]{Mikami24}), and because of this, there is a subtle issue on the descent of finite projective modules.
	To avoid this problem, we take the sheafification.
\end{remark}

\begin{theorem}[{\cite[Theorem 5.20]{Mikami24}}]\label{thm:classification intro}
	Under a certain freeness assumption, the stack $\Xfrak_{1,K}^{\mathrm{mod}}$ defined via the moduli interpretation coincides with the one defined by Rodrigues Jacinto--Rodr\'{\i}guez Camargo.
\end{theorem}

The stack $\Xfrak_{n,K}$ has a substack consisting of de Rham $G_K$-equivariant vector bundles over $X_{\Cbb_p}$.
Following the work of Berger--Colmez (\cite{BC08}), it is explained in \cite[5.2]{EGH23} that de Rham families of $G_K$-equivariant vector bundles over $X_{\Cbb_p}$ can be described as families of filtered $(\varphi, N, G_K)$-modules, which is called the \textit{$p$-adic monodromy theorem} (for arithmetic families).
Since $(\varphi,N, G_K)$-modules can also be described in terms of Weil--Deligne representations (\cite{Fontaine94l}), this substack will play an important role in the comparison between categorical locally analytic $p$-adic local Langlands correspondence and the smooth categorical local Langlands correspondence.
The aim of this paper is to extend the above observation to families parametrized by algebraic-affinoid analytic $\Qbb_{p,\square}$-algebras.

\subsection{Statements of the main results}
Let us now describe what is carried in this paper.
Let $\Acal=(A,A^+)_{\square}$ be an algebraic-affinoid $\Qbb_{p,\square}$-algebra.

\subsubsection{Hodge-Tate representations and de Rham representations}
Let $B_{\HT}$ be the usual Fontaine's period ring.
We regard it as a solid $\Qbb_{p,\square}$-algebra.
We set $B_{\HT,A}=B_{\HT}\otimes A(=B_{\HT}\otimes \Acal)$.
As usual, there is a natural morphism $\AnSpec (B_{\HT},\Zbb)_{\square}\to X_{\Cbb_p}$.
Using this, from a $G_K$-equivariant vector bundle $V$ over $X_{\Cbb_p,\Acal}=X_{\Cbb_p}\times_{\AnSpec\Qbb_{p,\square}} \AnSpec \Acal$, we construct a finite projective $B_{\HT,A}$-module $V_{\HT}$ with a semilinear $G_K$-action.
We say that a $G_K$-equivariant vector bundle $V$ over $X_{\Cbb_p,\Acal}$ is \textit{Hodge-Tate} if the natural morphism
$V_{\HT}^{G_K}\otimes_A B_{\HT,A}\to V_{\HT}$
is an isomorphism.

In the context of de Rham representations, there are two candidates $B_{\dR,A}^+$: the tensor product $B_{\dR}^+\otimes A$, and its $t$-adic completion (where $t\in B_{\dR}^+$ is Fontaine's $t$).
Let $B_{\dR,A}^+$ denote the latter one, and $B_{\sdR,A}^+$ denote the former one\footnote{$\sdR$ stands for ``strongly de Rham''.}.
We note that when $A$ is a Banach $\Qbb_p$-algebra, $B_{\sdR,A}^+$ is already $t$-adically complete.
In general, however, this is not the case; for instance, if $A=\Qbb_p[T]$, then $B_{\sdR,A}^+=B_{\dR}^+[T]$ is not $t$-adically complete.
By inverting $t$, we obtain $B_{\dR,A}$ and $B_{\sdR,A}$.
In the same way as in the Hodge-Tate case, we define the notion of \textit{de Rham} (resp. \textit{strongly de Rham}) $G_K$-equivariant vector bundles over $X_{\Cbb_p,\Acal}$.
It is easy to see that ``strongly de Rham'' implies ``de Rham''.
We can also prove the converse (Theorem~\ref{thm:cdR implies dR}).
Therefore, these notions are in fact equivalent.
The definition using $B_{\dR,A}$ is convenient for studying the relationship with Hodge-Tate $G_K$-equivariant vector bundles or filtered modules, whereas the definition using $B_{\sdR,A}$ is convenient for proving the $p$-adic monodromy theorem.

\subsubsection{Semistable representations}
Let $K_0$ be the maximal unramified extension of $\Qbb_p$ in $K$.
In \cite{Ber02}, Berger defined the notion of semistable representations for $(\varphi,\Gamma_K)$-modules over the Robba ring (which is essentially a ``limit of $p$-adic period rings'').
However, directly adapting the Robba ring to our setting is problematic; taking such limits behaves poorly for algebraic-affinoid $\Qbb_{p,\square}$-algebras such as $\Qbb_p[T]$.
Therefore, it is necessary to give a different definition.
Let $V$ be a $G_K$-equivariant vector bundle over $X_{\Cbb_p,\Acal}$.
As usual, for each closed interval $[r,s]\subset (0,\infty)$, there is a natural morphism $Y_{\Cbb_p,\Acal}^{[r,s]}=Y_{\Cbb_p}^{[r,s]}\times_{\AnSpec \Qbb_{p,\square}} \AnSpec\Acal\to X_{\Cbb_p,\Acal}$.\footnote{Our normalization is chosen such that the distinguished point of $Y_{\Cbb_p}$ lies in $Y_{\Cbb_p}^{[1,1]}$.}
Let $V^{[r,s]}$ denote the pullback of $V$ along this morphism.
We write $\tilde{B}_{\Cbb_p}^{[r,s]}=\Ocal(Y_{\Cbb_p}^{[r,s]})$ and $\tilde{B}_{\Cbb_p,A}^{[r,s]}=\tilde{B}_{\Cbb_p}^{[r,s]}\otimes A$.
\begin{definition}
We say that $V$ is \textit{semistable} if the natural morphism
$$V^{[p^{-1},1]}[\log [p^{\flat}],1/t]^{G_K}\otimes_{K_0\otimes A} \tilde{B}_{\Cbb_p,A}^{[p^{-1},1]}[\log [p^{\flat}],1/t] \to V^{[p^{-1},1]}[\log [p^{\flat}],1/t]$$
is an isomorphism.
\end{definition}
Note that the choice of $[p^{-1},1]$ is not essential, and we may replace this with a closed interval $I$ such that $p^n,p^{n+1}\in I$ for some $n\in \Zbb$.
By gluing along the Frobenius $\varphi$, we can extend the fixed vectors over $[p^{-1},1]$
to those over any closed interval $[r,s]\subset (0,\infty)$.
Based on this observation, we prove that when $A=\Qbb_p$, the above definition coincides with Berger's definition (Theorem~\ref{thm:comparison with the classical definition}).
Moreover, by the standard argument, we prove that the category of semistable $G_K$-equivariant vector bundles over $X_{\Cbb_p,\Acal}$ is equivalent to the category of filtered $(\varphi,N)$-modules over $K_0\otimes A$ (Theorem~\ref{thm:dR filtered phi N Gal-module}).

\subsubsection{The $p$-adic monodromy theorem}
The main theorem of this paper is the following:

\begin{theorem}[{Theorem~\ref{thm:p-adic monodromy}}]
	Let $V$ be a $G_K$-equivariant vector bundle over $X_{\Cbb_p,\Acal}$.
	Then $V$ is de Rham if and only if it is potentially semistable (i.e., semistable as a $G_L$-equivariant vector bundle for some finite extension $L/K$).
\end{theorem}

It is easy to see that ``potentially semistable'' implies ``de Rham''.
Let us briefly outline the proof of the converse.
First, we reduce to the case $A=R_f=R[1/f]$, where $R$ is an affinoid $\Qbb_p$-algebra and $f\in R$ is a non-zero divisor.
When $A$ is reduced, then the same argument as in \cite{BC08} works well.
More precisely, there is a closed embedding $R\to \prod_{i=1}^n E_i,$ where each $E_i/\Qbb_p$ is a complete discretely valued field with algebraically closed residue field. 
It also induces an injection $A=R_f\to \prod_{i=1}^n E_i$.
Then from the $p$-adic monodromy theorem for $E_i$ proved by \cite{Andre02,Mebkhout02,Ked22book}, we can deduce the theorem for $A=R_f$.
For the general non-reduced case, we proceed by d\'{e}vissage.

\subsubsection{A classification of $G_K$-equivariant line bundles}
As an application of the $p$-adic monodromy theorem, we prove the classification of $G_K$-equivariant line bundles over $X_{\Cbb_p,\Acal}$.
More precisely, we prove the following.
\begin{theorem}[{Theorem~\ref{thm:classification of line bundles}}]\label{thm:classification of line bundles intro}
	Let $V$ be a $G_K$-equivariant line bundle over $X_{\Cbb_p,\Acal}$.
	Then there exist a unique continuous character $\delta\colon K^{\times}\to A^{\times}$ and a unique finite projective $\Acal$-module $M$ of rank $1$ such that $V\cong \Ocal_{X_{\Cbb_p,\Acal}}(\delta)\otimes_{\Acal}M.$
	For the definition of $\Ocal_{X_{\Cbb_p,\Acal}}(\delta)$, see Definition~\ref{def:character type}.
\end{theorem}

This classification was previously obtained by the author in \cite{Mikami24} (Theorem~\ref{thm:classification intro}) under a certain freeness condition. 
The approach in the present paper does not require this condition. 
Moreover, while the previous classification relied on the results for the affinoid cases proved by Kedlaya--Pottharst--Xiao (\cite{KPX14}), the new approach bypasses them entirely. 
In particular, it also yields an alternative proof of the classification by Kedlaya--Pottharst--Xiao.

Let us explain the outline of the proof of Theorem~\ref{thm:classification of line bundles intro}.
First, we construct a character $\delta_0 \colon K^{\times}\to A^{\times}$ such that $V\otimes \Ocal_{X_{\Cbb_p,\Acal}}(\delta_0^{-1})$ is Hodge-Tate of Hodge-Tate weight $0$. 
By replacing $V$ with $V\otimes \Ocal_{X_{\Cbb_p,\Acal}}(\delta_0^{-1})$, we may assume that $V$ is Hodge-Tate of Hodge-Tate weight $0$. 
Then $V$ is also de Rham, and thus potentially semistable by the $p$-adic monodromy theorem.
Therefore, $V$ can be described using a $(\varphi,G_K)$-module\footnote{Since the Hodge-Tate weight of $V$ is zero, we may ignore the filtration. Moreover, in the case of rank $1$, the monodromy operator $N$ always vanishes.} of rank $1$.
Using this description, we obtain Theorem~\ref{thm:classification of line bundles intro}.

\begin{remark}
	For an affinoid $\Qbb_p$-algebra $A$, the classification of $G_K$-equivariant line bundles over $X_{\Cbb_p,\Acal}$ using the $p$-adic monodromy theorem is also addressed in the forthcoming work \cite{Hellmann-Heuer26} of Hellmann--Heuer.
\end{remark}

\subsection{Outline of the paper}
This paper is organized as follows.
In the former part of Section 1, we define the notion of Hodge-Tate representations and prove basic properties.
In the latter part of Section 1, we prove that coefficients of Sen polynomials are ``bounded''.

In the former part of Section 2, we define the notion of de Rham representations and compare them and filtered modules.
In the latter part of Section 2, we define the notion of strongly de Rham representations and prove that ``strongly de Rham'' is equivalent to ``de Rham''.

In the first part of Section 3, we define the notion of semistable representations and prove basic properties.
In the second part of Section 3, we prove the $p$-adic monodromy theorem.
In the final part of Section 3, we construct an equivalence between the category of semistable representations and filtered $(\varphi,N)$-modules.

In Section 4, we prove the classification of $G_K$-equivariant line bundles.
\subsection{Convention}

\begin{itemize}
\item
All rings, including condensed ones, are assumed unital and commutative.
\item 
In contrast to \cite{And21, Mann22}, we use the term \textit{ring} to refer to an \textit{ordinary ring} (not an animated ring). 
Moreover, we use the symbol $-\otimes-$ to refer to a \textit{non-derived tensor product} and use the symbol $-\otimes^{\Lbb}-$ to refer to a \textit{derived tensor product}.
Similarly, we adopt analogous notation for Hom and limit.
\item
We use the terms \textit{f-adic ring} and \textit{affinoid pair} rather than \textit{Huber ring} and \textit{Huber pair}.
\item For an f-adic ring $A$, we denote the ring of power-bounded elements of $A$ by $A^{\circ}$.
\item For a complete non-archimedean field $K$, we use the term \textit{affinoid $K$-algebra} to refer to a \textit{topological $K$-algebra topologically of finite type over $K$}.
\item Throughout this paper, all radii $r$ and $s$ are assumed to be rational numbers.
\item Unless otherwise stated, all actions are assumed to be continuous (or equivalently, actions as condensed objects).
\end{itemize}

\subsection{Convention and notation about condensed mathematics}
In this paper, we use condensed mathematics.
We summarize the notations and conventions related to condensed mathematics.

\begin{itemize}
\item Throughout this paper, we fix an uncountable solid cutoff cardinal $\kappa$ as in \cite[Definition 2.9.11]{Mann22} and work with $\kappa$-condensed objects.
	Our results do not depend on the choice of $\kappa$.
	If the reader prefers to work with light condensed objects, one may simply replace ``condensed'' with ``light condensed'' throughout the paper without affecting any of the arguments.
\item We often identify a compactly generated topological set, ring, group, etc. $X$ whose points are closed (i.e., $X$ is $T1$) with a condensed set, ring, group, etc. $\underline{X}$ associated to $X$. 
    It is justified by \cite[Proposition 1.7]{CM}.
    If there is no room for confusion, we simply write $X$ for $\underline{X}$.
	For example, there is a fully faithful functor from the category of Banach $\Qbb_p$-modules to the category of condensed $\Qbb_p$-modules (\cite[Lemma 3.24]{RJRC22}), and we regard a (classical) Banach $\Qbb_p$-module $V$ as a condensed $\Qbb_p$-module via this functor.
	Moreover, a condensed $\Qbb_p$-module $V$ is said to be \textit{Banach} if $V$ is isomorphic to $\underline{V_0}$ for some Banach $\Qbb_p$-module $V_0$.
\item For a condensed set $X$, we simply write $x\in X$ to mean $x\in X(\ast)$.
\item 
    In contrast to \cite{Mann22}, we use the term \textit{ring} to refer to an \textit{ordinary ring} (not a condensed animated ring). 
    Sometimes we use the term \textit{discrete ring} (resp.\ \textit{discrete animated ring}) to refer to an ordinary ring (resp.\ animated ring) in order to emphasize that it is not a condensed one. We also use the term \textit{static ring} (resp.\ \textit{static analytic ring})  to refer to an ordinary ring (resp.\ analytic ring) in order to emphasize that it is not an animated one.
    \item We use the terms ``analytic animated ring'' and ``uncompleted analytic animated ring'' according to \cite{Mann22}.
    \item For an uncompleted analytic animated ring $\Acal$, we denote the underlying condensed animated ring of $\Acal$ by $\underline{\Acal}$.
    \item For an uncompleted analytic animated ring $\Acal$, an object $M\in \Dcal(\underline{\Acal})$ is said to be \textit{$\Acal$-complete} if it lies in $\Dcal(\Acal)$.    
    \item Let $\Zbb_{\square}$, $\Zbb_{p,\square}$ denote the analytic rings defined in \cite[Example 7.3]{CM}. 
    Note that the analytic ring structure of $\Zbb_{p,\square}$ is induced from $\Zbb_{\square}$.
	Moreover, for a usual ring $A$, let $A_{\square}=(A,A)_{\square}$ denote the analytic ring defined in \cite[Definition 2.9.1]{Mann22}.
    \item For a condensed animated ring $A$ and for a morphism of usual rings $B\to \pi_0A(\ast)$, let $(A,B)_{\square}$ denote the condensed animated ring $A$ with the induced analytic ring structure from $(B,B)_{\square}$.
    For details, see \cite[Definition 7.1.1]{RC26}.
	For a finite extension $E/\Qbb_p$, we simply write $E_{\square}=(E,\Zbb)_{\square}$.
	\item For a solid $\Qbb_{p,\square}$-algebra $A$, we abbreviate $-\otimes_{(A,\Zbb)_{\square}}-$ by $-\otimes_A-$.
	When $A=\Qbb_p$, we will simply write $-\otimes-$.
	\item For a profinite set $S$ and an object $M\in \Dcal(\Zbb_{\square})$, we write $C(S,M)=R\intHom_{\Zbb}(\Zbb_{\square}[S],M)$. 
    If $M$ is static then $C(S,M)$ is static, since $\Zbb_{\square}[S]$ is a projective $\Zbb_{\square}$-module.
    If $M$ is a $\Zbb_{\square}$-module associated to a compactly generated Hausdorff $\Zbb$-module $N$ (that is, $\underline{N}=M$), then $C(S,M)$ is a $\Zbb_{\square}$-module associated to the module $C(S,N)$ of continuous functions $S\to N$ endowed with the compact-open topology.
\end{itemize}

\subsection{Notation}
\begin{itemize}
\item We fix an algebraic closure $\bar{\Qbb}_p$ and its completion $\Cbb_p$. Unless otherwise stated, all fields considered are regarded as subfields of $\Cbb_p$.
\item Let $K$ be a finite extension field of $\Qbb_p$, and $K_0$ be the maximal unramified extension of $\Qbb_p$ in $K$.
We write $K_n=K(\zeta_{p^n})$ and $K_{\infty}=K(\zeta_{p^{\infty}}).$
Let $K^{\unr}$ denote the maximal unramified extension of $K$, and $\breve{K}$ denote the completion of $K^{\unr}$.
\item We write $G_K=\Gal(\overline{K}/K)$, $H_K=\Gal(\overline{K}/K_{\infty})$, and $\Gamma_K=\Gal(K_{\infty}/K)$.
\item Let $\chi \colon G_K\to \Gamma_K\to \Qbb_p^{\times}$ denote the $p$-adic cyclotomic character.
\item Let $\Acal=(A,A^+)_{\square}$ be an algebraic-affinoid analytic $\Qbb_{p,\square}$-algebra (\cite[Definition 2.4]{Mikami24}) or an analytic $\Qbb_{p,\square}$-algebra associated to a sheafy analytic affinoid pair $(A,A^+)$ over $(\Qbb_p,\Zbb_p)$.
\item Let $Y_{\Cbb_p}=\Spa (W(\Ocal_{\Cbb_p}^{\flat}))\setminus\{p[p^{\flat}]=0\}$, and let $X_{\Cbb_p}=Y_{\Cbb_p}/\varphi^{\Zbb}$ be the Fargues--Fontaine curve\footnote{In \cite{Mikami24}, $X_{\Cbb_p}$ is written as $X_{\bar{K}}$. Since we often vary $K$ in this paper, we adopt this notation.} associated to $\Cbb_p^{\flat}$.
For a closed interval $[r,s]\subset (0,\infty)$ ($r,s \in \Qbb$), we define $Y_{\Cbb_p}^{[r,s]}$ as in  \cite{Mikami24}.
We write 
\begin{align*}
	X_{\Cbb_p,\Acal}&=X_{\Cbb_p}\times_{\AnSpec \Qbb_{p,\square}} \AnSpec \Acal,\\
	Y_{\Cbb_p,\Acal}^{[r,s]}&=Y_{\Cbb_p}^{[r,s]}\times_{\AnSpec \Qbb_{p,\square}} \AnSpec \Acal.
\end{align*}
\item We write $\tilde{B}_{\Cbb_p}^{[r,s]}$ (resp. $\tilde{B}_{\Cbb_p}^{[r,s]+}$) for the ring of analytic functions $\Ocal_{Y_{\Cbb_p}}(Y_{\Cbb_p}^{[r,s]})$ (resp. $\Ocal^+_{Y_{\Cbb_p}}(Y_{\Cbb_p}^{[r,s]})$). 
We write $\tilde{B}^{[r,s]}_{\Cbb_p,A}=\tilde{B}^{[r,s]}_{\Cbb_p}\otimes \Acal$. Note that this does not depend on the choice of $A^+$ since $\tilde{B}_{\Cbb_p}^{[r,s]}$ is a nuclear $\Qbb_{p,\square}$-module (cf. \cite[Lemma 3.2]{Mikami24}).
\item We write $\tilde{B}_{K_{\infty}}^{[r,s]}=(\tilde{B}^{[r,s]}_{\Cbb_p})^{H_K}$ and $B_{K,\infty}^{[r,s]}=(\tilde{B}_{K_{\infty}}^{[r,s]})^{\Gamma_K\mathchar`-\la}.$
We define $B_{K,n}^{[r,s]}$ as in \cite{Mikami24}, which is an affinoid $\Qbb_p$-algebra.
We define $\tilde{B}_{K_{\infty},A}^{[r,s]}$, $B_{K,\infty,A}^{[r,s]}$, and $B_{K,n,A}^{[r,s]}$ by applying $-\otimes \Acal$.
\item We fix a compatible system $\varepsilon=(1,\zeta_p,\zeta_{p^2},\ldots)\in \Cbb_p^{\flat}$ of $p^n$th roots of unity.
\item For a closed interval $[r,s]\subset (0,\infty)$, we write $t=\log [\varepsilon]\in \tilde{B}^{[r,s]}.$ 
\item For a $\Qbb_{p,\square}$-module $V$ with an action of $G_K$ (or $\Gamma_K$) and $n\in \Zbb$, let $V(n)$ denote the $n$th Tate twist.
\item We normalize the Hodge-Tate weight of $\Qbb_p(1)$ to be $1$.
\end{itemize}

\addtocontents{toc}{\protect\setcounter{tocdepth}{2}}
\subsection*{Acknowledgements}
The author is grateful to Yoichi Mieda for his support during the studies of the author.
This work was started during the author's stay at the University of M\"{u}nster under the Mathematics M\"{u}nster programme for Visiting Doctoral Researchers, and the author thanks Eugen Hellmann for his hospitality and support during the stay.
The author thanks Laurent Berger and Gal Porat for answering my questions.
This work was supported by JSPS KAKENHI Grant Number JP23KJ0693.

\section{Hodge-Tate representations}
\subsection{Definitions and basic properties of Hodge-Tate representations}

\begin{definition}
For $V\in \Dcal(X_{\Cbb_p,\Acal})$, let $V^{[r,s]}$ denote the pullback of $V$ under the morphism 
$$\AnSpec (\tilde{B}^{[r,s]}_{\Cbb_p,A},\tilde{B}^{[r,s]+}_{\Cbb_p}\otimes A^+)_{\square}\cong Y_{\Cbb_p,\Acal}^{[r,s]}\to X_{\Cbb_p,\Acal}.$$
We say $V\in \Dcal(X_{\Cbb_p,\Acal})$ is a \textit{vector bundle over $X_{\Cbb_p,\Acal}$}\footnote{This definition differs from \cite[Definition 3.30]{Mikami25}, and it is unclear whether these definitions are equivalent.} if and only if for any $[r,s]\subset (0,\infty)$, $V^{[r,s]}$ is a finite projective $\tilde{B}^{[r,s]}_A$-module.
Let $\Vect(X_{\Cbb_p,\Acal})$ denote the category of vector bundles over $X_{\Cbb_p,\Acal}$.
We note that $\Vect(X_{\Cbb_p,\Acal})$ is equivalent to the category of $\varphi$-modules over $\tilde{B}_{\bar{K},A}$ defined in \cite[Definition 3.4]{Mikami24}. 
Similarly, we define the category $\Vect(X_{\Cbb_p,\Acal}/G_K)$ of $G_K$-equivariant vector bundles over $X_{\Cbb_p,\Acal}$, which is equivalent to the category of $(\varphi,G_K)$-modules over $\tilde{B}_{\bar{K},A}$ defined in \cite[Definition 3.4]{Mikami24}. 
\end{definition}

\begin{remark}
	The categories $\Vect(X_{\Cbb_p,\Acal})$ and $\Vect(X_{\Cbb_p,\Acal}/G_K)$ do not depend on the choice of $A^+$.
\end{remark}

\begin{definition}
We define a $\Cbb_{p,\square}$-algebra $B_{\HT}$ with a semilinear $G_K$-action as
$B_{\HT}\coloneqq \Cbb_p[t,t^{-1}]$ where $t$ is an indeterminate such that $gt=\chi(g)t$ for $g\in G_K$.
We define $B_{\HT,A}=B_{\HT}\otimes\Acal$.
Since $B_{\HT}$ is a nuclear $\Qbb_{p,\square}$-module, the definition of $B_{\HT,A}$ does not depend on the choice of $A^+$ (cf. the proof of \cite[Lemma 3.2]{Mikami24}).
\end{definition}

\begin{lemma}\label{lem:HT fixed vector}
	We have $B_{\HT,A}^{G_K}=K\otimes A$.
\end{lemma}
\begin{proof}
	We have $B_{\HT}^{G_K}=K$, which is well-known.
	In other words, there is a left exact sequence
	$$0\to K\to B_{\HT} \to C(G_K,B_{\HT}),$$
	where $B_{\HT} \to C(G_K,B_{\HT})$ is given by $x\mapsto (g\mapsto gx)$.
	Since $A$ is nuclear and flat over $\Qbb_{p,\square}$ by \cite[Remark 2.3, Example 1.36 (3)]{Mikami24}, we get a left exact sequence
	$$0\to K\otimes A\to B_{\HT,A} \to C(G_K,B_{\HT,A}),$$
	where we note $C(G_K,B_{\HT,A})\cong C(G_K,B_{\HT})\otimes A$ by \cite[Proposition 5.35]{And21}.
\end{proof}

\begin{lemma}
	The ring $B_{\HT,A}$ is faithfully flat over $K\otimes A$.
\end{lemma}
\begin{proof}
	We may assume $A=\Qbb_p$.
	Then the claim easily follows from \cite[Lemma 3.21]{RJRC22}.
\end{proof}

\begin{definition}
\begin{enumerate}
\item Let $V$ be a finite projective $\Cbb_p\otimes A$-module with a semilinear $G_K$-action.
We define a finite projective $B_{\HT,A}$-module $V_{\HT}$ with a semilinear $G_K$-action as 
$$V_{\HT}\coloneqq V \otimes_{\Cbb_p\otimes  A} B_{\HT,A}=\bigoplus_{n\in \Zbb}t^nV=\bigoplus_{n\in \Zbb}V(n),$$
where $-(n)$ means the $n$th Tate twist.
We also define $D_{\HT}^K(V)\coloneqq V^{G_K}$, which admits a grading $\gr^{\bullet}D_{\HT}^K(V)$ defined via $\gr^{n}D_{\HT}^K(V)\coloneqq(t^nV)^{G_K}$.
When $K$ is clear from the context, we omit $K$ from the notation.
We say that $V$ is \textit{Hodge-Tate} if the natural morphism
$$D_{\HT}(V)\otimes_{K\otimes A} B_{\HT,A}\to V_{\HT}$$
is an isomorphism.
Let $\FP_{\Cbb_p\otimes A}(G_K)$ denote the category of finite projective $\Cbb_p\otimes A$-modules with a semilinear $G_K$-action, and $\FP_{\Cbb_p\otimes A}^{\HT}(G_K)$ denote the subcategory of $\FP_{\Cbb_p\otimes A}(G_K)$ consisting of Hodge-Tate representations.
\item Let $V$ be a $G_K$-equivariant vector bundle over $X_{\Cbb_p,\Acal}$.
We define $V_{\infty}$ as the reduction of $V$ at the distinguished point $x_{\infty}$, that is, $V_{\infty}=V^{[1,1]}/t$, which is a finite projective $\Cbb_p\otimes  A$-module with a semilinear $G_K$-action.
We say that $V$ is \textit{Hodge-Tate} if $V_{\infty}$ is Hodge-Tate.
In this case, we simply write $V_{\HT}$ (resp. $D_{\HT}(V)$, $\gr^{n}D_{\HT}(V)$) for $(V_{\infty})_{\HT}$ (resp. $D_{\HT}(V_{\infty})$, $\gr^{n}D_{\HT}(V_{\infty})$).
Let $\Vect^{\HT}(X_{\Cbb_p,\Acal}/G_K)$ denote the category of Hodge-Tate $G_K$-equivariant vector bundles over $X_{\Cbb_p,\Acal}$.
\end{enumerate}
\end{definition}

\begin{remark}\label{rem:A^+ independent}
	The definition of Hodge-Tate is independent of the choice of $A^+$.
\end{remark}

\begin{lemma}\label{lem:DHT fin proj}
For $V\in\FP_{\Cbb_p\otimes A}^{\HT}(G_K)$, $D_{\HT}(V)$ and $\gr^{n}D_{\HT}(V)$ are finite projective $K\otimes A$-modules.
\end{lemma}
\begin{proof}
	Since $B_{\HT,A}$ is faithfully flat over $(K\otimes A,\Zbb)_{\square}$, $D_{\HT}(V)$ is a finite projective $K\otimes A$-module by \cite[Theorem 1.35]{Mikami24}.
	Since $\gr^{n}D_{\HT}(V)$ is a direct summand of $D_{\HT}(V)$, it is also a finite projective $K\otimes A$-module.
\end{proof}

\begin{proposition}\label{prop:HT recover}
	For $V\in\FP_{\Cbb_p\otimes A}^{\HT}(G_K)$, there is a natural $G_K$-equivariant isomorphism
	$$V\cong \bigoplus_{n\in \Zbb} \gr^nD_{\HT}(V)\otimes_{K\otimes A} (\Cbb_p\otimes A)(-n).$$
\end{proposition}
\begin{proof}
	By taking the $0$th grading of the isomorphism of graded modules
	$$D_{\HT}(V)\otimes_{K\otimes A} B_{\HT,A}\to V_{\HT},$$
	we get the claim.
\end{proof}

\begin{definition}
	For $V\in\FP_{\Cbb_p\otimes A}^{\HT}(G_K)$, we define (a family of multisets of) Hodge-Tate weights of $V$ as usual which is a family of multisets indexed by $\pi_0(\Spec (K\otimes A)(\ast))$.
	We normalize the Hodge-Tate weight so that the Hodge-Tate weight of the $p$-adic cyclotomic character $\chi$ is equal to $1$. 
\end{definition}

\begin{example}\label{ex:Ht wt 0}
	For $V\in\FP_{\Cbb_p\otimes A}(G_K)$, $V$ is Hodge-Tate of Hodge-Tate weight $0$ if and only if the natural morphism
	$$V^{G_K}\otimes_{K\otimes A} (\Cbb_p\otimes A) \to V$$
	is an isomorphism.
\end{example}

\begin{lemma}\label{lem:HT dual direct sum tensor product}
	The subcategory $\FP_{\Cbb_p\otimes A}^{\HT}(G_K)\subset \FP_{\Cbb_p\otimes A}(G_K)$ is stable under taking duals, direct sums, and tensor products.
\end{lemma}
\begin{proof}
	For $V\in \FP_{\Cbb_p\otimes A}^{\HT}(G_K)$, there are isomorphisms
	\begin{align*}
		(V^*)_{\HT}&\cong (V_{\HT})^*\\
		&\cong \intHom_{B_{\HT,A}}(D_{\HT}(V)\otimes_{K\otimes A}B_{\HT,A},B_{\HT,A})\\
		&\cong \intHom_{K\otimes A}(D_{\HT}(V),B_{\HT,A})\\
		&\cong D_{\HT}(V)^*\otimes_{K\otimes A}B_{\HT,A},
	\end{align*}
	where $(-)^*$ denotes the dual.
	Since $D_{\HT}(V)^*$ is a finite projective $K\otimes A$-module, we have 
	\begin{align*}
	(D_{\HT}(V)^*\otimes_{K\otimes A}B_{\HT,A})^{G_K}\cong D_{\HT}(V)^*\otimes_{K\otimes A}B_{\HT,A}^{G_K}\cong D_{\HT}(V)^*.
	\end{align*}
	Therefore, $V^*$ is also Hodge-Tate.
	The other cases can be proved in a similar way.
\end{proof}

\begin{lemma}\label{lem:HT to graded tensor}
	For $V,W \in \FP_{\Cbb_p\otimes A}^{\HT}(G_K)$, there is a natural isomorphism of graded modules
	$$\gr^{\bullet}D_{\HT}(V) \otimes_{K\otimes A} \gr^{\bullet}D_{\HT}(W) \cong \gr^{\bullet}D_{\HT}(V\otimes_{B_{\HT,A}}W).$$
\end{lemma}
\begin{proof}
	It easily follows from the description in Proposition~\ref{prop:HT recover}
\end{proof}

\begin{lemma}\label{lem:potentially HT}
Let $L/K$ be a finite extension.
Then for $V\in\FP_{\Cbb_p\otimes A}(G_K)$, $V$ is Hodge-Tate as a $G_K$-representation if and only if it is Hodge-Tate as a $G_L$-representation.
\end{lemma}
\begin{proof}
	It easily follows from the Galois descent.
\end{proof}

\begin{lemma}\label{lem:HT basechange}
For $V\in\FP_{\Cbb_p\otimes A}^{\HT}(G_K)$ and a morphism $\Acal\to \Bcal=(B,B^+)_{\square}$, $V_B=V\otimes_{\Acal}\Bcal$ is Hodge-Tate (as an object of $\FP_{\Cbb_p\otimes B}(G_K)$).
Moreover, in this case, there is an isomorphism of graded $K\otimes B$-modules 
$$D_{\HT}(V_B)\cong D_{\HT}(V)\otimes_{K\otimes A} (K\otimes B).$$
\end{lemma}
\begin{proof}
	We write $M=(V_{\HT})^{G_K}$, which is a finite projective $K\otimes A$-module.
	We write $M_B=M\otimes_{A} B$, which is a finite projective $K\otimes B$-module.
	Then we have an isomorphism
	$$M_B\otimes_{K\otimes B} B_{\HT,A}\cong V_{\HT,B}.$$
	Therefore, $V_B$ is Hodge-Tate with $D_{\HT}(V_B)=D_{\HT}(V)\otimes_{K\otimes A} (K\otimes B).$
\end{proof}

\begin{proposition}\label{prop:HT analytic local}
Let 
$\{\AnSpec\Acal_i \to \AnSpec\Acal\}_{i=1}^n$
be an affinoid covering of $\AnSpec \Acal$. 
Then for $V\in\FP_{\Cbb_p\otimes A}(G_K)$, $V$ is Hodge-Tate if and only if $V_i=V\otimes_{\Acal}\Acal_i$ is Hodge-Tate for each $i$.
\end{proposition}
\begin{remark}
	If $\Acal=(A,A^+)_{\square}$ is an algebraic-affinoid analytic $\Qbb_{p,\square}$-algebra, the definition of an affinoid covering is given in \cite[Definition 2.26 (5)]{Mikami24}.
	If $(A,A^+)$ is a sheafy analytic affinoid pair over $(\Qbb_p,\Zbb_p)$, then an affinoid covering of $\AnSpec \Acal$ means a covering induced from an affinoid open covering $\{\Spa(A_i,A_i^+)\to \Spa(A,A^+)\}_{i=1}^n$ in Huber's sense.
	When $(A,A^+)$ is an affinoid pair over $(\Qbb_p,\Zbb_p)$ of weakly finite type, these notions coincide by \cite[Example 2.31]{Mikami24}.
\end{remark}

\begin{proof}
	The only if part follows from Lemma~\ref{lem:HT basechange}.
	Let us prove the if part.
	We write $M_i=V_{i,\HT}^{G_K}$, which is a finite projective $K\otimes A_i$-module.
	By the proof of Lemma~\ref{lem:HT basechange}, there are natural isomorphisms
	$M_i\otimes_{\Acal_i} \Acal_{ij} \cong M_j\otimes_{\Acal_j} \Acal_{ij}$, which satisfy the cocycle condition.
	Therefore, we obtain a dualizable object $M\in \Dcal(\Acal)$ and a morphism $M\to V_{\HT}$ by gluing $M_i\to V_{i,\HT}$.
	The morphism 
	$$M\otimes_{K\otimes A}^{\Lbb} B_{\HT,A}\to V_{\HT}$$ is an isomorphism since it becomes an isomorphism after applying $-\otimes_{\Acal} \Acal_i$.
	Since the morphism $(K\otimes A,\Zbb)_{\square} \to (B_{\HT,A},\Zbb)_{\square}$ is faithfully flat, $M$ is a finite projective $K\otimes A$-module by \cite[Theorem 1.35]{Mikami24}, which proves the claim.
\end{proof}

\begin{lemma}\label{lem:prod tensor inj}
	Let $I$ and $J$ be countable sets, and let $\{M_i\}_{i\in I}$ and $\{N_j\}_{j\in J}$ be families of $\Qbb_{p,\square}$-modules.
	Assume that each $M_i$ and $N_j$ can be written as a filtered colimit of Banach $\Qbb_p$-modules with injective transition morphisms. 
	Then the natural morphism
	$$\prod_{i\in I} M_i \otimes \prod_{j\in J}N_j\to \prod_{(i,j)\in I\times J}M_i\otimes N_j$$
	is injective.
\end{lemma}
\begin{proof}
	We write $M_i$ and $N_j$ as filtered colimits of Banach $\Qbb_p$-modules with injective transition morphisms 
	$$M_i=\varinjlim_{\lambda_i \in \Lambda_i} M_{i,\lambda_i},\quad N_i=\varinjlim_{\sigma_j \in \Sigma_j} N_{j,\sigma_j}.$$
	Then there are isomorphisms
	\begin{align*}
		\prod_{i\in I} M_i &\cong \varinjlim_{(\lambda_i)\in \prod \Lambda_i} \prod_{i\in I} M_{i,\lambda_i},\\
		\prod_{j\in J} N_j &\cong \varinjlim_{(\sigma_j)\in \prod \Sigma_j} \prod_{j\in J} N_{j,\sigma_j},\\
		\prod_{(i,j)\in I\times J}M_i\otimes N_j &\cong \varinjlim_{(\lambda_{i,j},\sigma_{i,j})_{i,j}\in \prod_{i,j} (\Lambda_i \times \Sigma_j)} \prod_{(i,j)\in I\times J}M_{i,\lambda_i} \otimes  N_{j,\sigma_j},
	\end{align*}
	where each transition morphism is injective.
	Therefore, we may assume that $M_i$ and $N_j$ are Banach $\Qbb_p$-modules.
	Then the claim follows from \cite[Lemma 3.28]{RJRC22}.
\end{proof}

\begin{lemma}\label{lem:HT inj}
	Assume $\Acal \in \AlgAff_{\Qbb_p}$.
	Then for $V\in\FP_{\Cbb_p\otimes A}(G_K)$, the natural morphism
	\begin{align*}
		&D_{\HT}(V)\otimes_{K\otimes A} B_{\HT,A}\to V_{\HT}
	\end{align*}
	is injective.
\end{lemma}
\begin{proof}
	Since $B_{\HT,A}$ is flat over $K\otimes A$, we get an injection
	$$V_{\HT}^{G_K}\otimes_{K\otimes A} B_{\HT,A}\to V_{\HT}\otimes_{K\otimes A} B_{\HT,A}.$$
	Let us prove that 
	$$V_{\HT}\otimes_{K\otimes A} B_{\HT,A} \to \prod_{\mfrak\in \Spm A, n\geq 1} V_{\HT,A/\mfrak^n}\otimes_{K\otimes A/\mfrak^n} B_{\HT,A/\mfrak^n}$$
	is injective.
	Since $V_{\HT}$ is a finite projective $B_{\HT,A}$-module, we may assume $V_{\HT}=B_{\HT,A}$.
	We need to show that
	$$(B_{\HT}\otimes_{K} B_{\HT})\otimes  A \to \prod_{\mfrak\in \Spm A, n\geq 1} (B_{\HT}\otimes_{K} B_{\HT})\otimes  A/\mfrak^n$$
	is injective.
	It easily follows from Lemma~\ref{lem:prod tensor inj}.
	Therefore, we get an injective morphism 
	$$V_{\HT}^{G_K}\otimes_{K\otimes A} B_{\HT,A}\to \prod_{\mfrak\in \Spm A, n\geq 1} V_{\HT, A/\mfrak^n}\otimes_{K\otimes A/\mfrak^n} B_{\HT,A/\mfrak^n}.$$
	This morphism factor through $\prod_{\mfrak\in \Spm A, n\geq 1} (V_{A/\mfrak^n,\HT})^{G_K}\otimes_{K\otimes A/\mfrak^n} B_{\HT,A/\mfrak^n}$, so we get an injection 
	$$V_{\HT}^{G_K}\otimes_{K\otimes A} B_{\HT,A}\to \prod_{\mfrak\in \Spm A, n\geq 1} (V_{A/\mfrak^n,\HT})^{G_K}\otimes_{K\otimes A/\mfrak^n} B_{\HT,A/\mfrak^n}.$$
	On the other hand, there is an injection
	$$V_{\HT}\to \prod_{\mfrak\in \Spm A, n\geq 1} V_{\HT,A/\mfrak^n}.$$
	Therefore, we may assume that $A$ is finite over $\Qbb_p$.
	Then the claim is well-known.
\end{proof}

\begin{corollary}\label{cor:HT sub quot}
Assume $\Acal \in \AlgAff_{\Qbb_p}$.
Let $V\in \FP^{\HT}_{\Cbb_p\otimes A}(G_K)$ and let $W\subset V$ be a $G_K$-stable submodule such that $W$ and $V/W$ are also finite projective.
Then $W$ and $V/W$ are also Hodge-Tate. 
\end{corollary}
\begin{proof}
	There is a diagram 
	$$
	\xymatrix{
	0\ar[r]& D_{\HT}(W)\otimes B_{\HT,A} \ar[r]\ar[d] & D_{\HT}(V)\otimes B_{\HT,A} \ar[r]\ar[d] & D_{\HT}(V/W)\otimes B_{\HT,A}\ar[d] &\\
	0 \ar[r] & W_{\HT}\ar[r] & V_{\HT}\ar[r] & (V/W)_{\HT} \ar[r] & 0,
	}$$
	where $-\otimes -$ abbreviates $-\otimes_{K\otimes A}-$, and where the upper sequence is left exact and the lower sequence is exact.
	The middle vertical morphism is an isomorphism, and by Lemma~\ref{lem:HT inj}, the right and left vertical morphism are injective.
	Therefore, by diagram chasing, we find that the right and left vertical morphism are also isomorphisms.
\end{proof}

\begin{remark}\label{rem:HT sub quot easy}
	We can show without the assumption $\Acal \in \AlgAff_{\Qbb_p}$ that if $W$ or $V/W$ is Hodge-Tate, then the other one is also Hodge-Tate.
	First, we assume that $W$ is Hodge-Tate.
	Then there is an isomorphism
	\begin{align*}
	V/W\cong &\Coker(D_{\HT}(W)\otimes_{K\otimes A} B_{\HT,A}\to D_{\HT}(V)\otimes_{K\otimes A} B_{\HT,A})\\
	\cong &(D_{\HT}(V)/D_{\HT}(W))\otimes_{K\otimes A} B_{\HT,A}.
	\end{align*}
	From this, we get $D_{\HT}(V/W)\cong D_{\HT}(V)/D_{\HT}(W)$ and $V/W$ is Hodge-Tate.
	Similarly, we can show that if $V/W$ is Hodge-Tate, then $W$ is also Hodge-Tate.
\end{remark}

Let us recall the following result on Galois descent. 
\begin{lemma}\label{lem:Galois descent}
	Let $\Rep_{\Cbb_{p}}(H_K)^{\nuc}$ denote the category of (static) nuclear $(\Cbb_p,\Zbb)_{\square}$-modules with a semilinear $H_K$-action, and let $\Mod(\hat{K}_{\infty})^{\nuc}$ denote the category of (static) nuclear $(\hat{K}_{\infty},\Zbb)_{\square}$-modules.
	Then the natural functors
	\begin{align*}
	&\Mod(\hat{K}_{\infty})^{\nuc}\to \Rep_{\Cbb_{p}}(H_K)^{\nuc};\; N\mapsto N\otimes_{\hat{K}_{\infty}} \Cbb_p,\\
	&\Rep_{\Cbb_{p}}(H_K)^{\nuc}\to \Mod(\hat{K}_{\infty})^{\nuc} ;\; M \to M^{H_K}
    \end{align*}
	are quasi-inverse to each other.
	Moreover, for $M\in \Rep_{\Cbb_{p}}(H_K)^{\nuc}$, we have
	$$R\Gamma(H_K,M)\cong M^{H_K},$$
	in other word, for any $i>0$, $H^i(H_K,M)=0$.
\end{lemma}
\begin{proof}
	It can be proved by the same argument as in \cite[Theorem 3.12]{Mikami24}.
\end{proof}

\begin{construction}\label{construction:Sen}
Let $V\in \FP_{\Cbb_p\otimes A}(G_K)$.
We define $V_{K_{\infty}}\coloneqq V^{H_K}$.
By Lemma~\ref{lem:Galois descent}, we have
\begin{align*}
V^{H_K}\otimes_{\hat{K}_{\infty} \otimes A} (\Cbb_p \otimes A)\cong V^{H_K}\otimes_{\hat{K}_{\infty}} \Cbb_p \cong V.
\end{align*}
By \cite[Theorem 1.35]{Mikami24}, $V_{K_{\infty}}=V^{H_K}$ is a finite projective $\hat{K}_{\infty} \otimes A$-module with a semilinear $\Gamma_K$-action.

We define $D_{\Sen}^{K_{\infty}}(V)\coloneqq V_{K_{\infty}}^{\Gamma_K\mathchar`-\la}$.
By the Tate-Sen method, 
$D_{\Sen}^{K_{\infty}}(V)$ is finite projective $K_{\infty} \otimes A$-module with a semilinear locally analytic $\Gamma_K$-action, and the natural morphism
$$D_{\Sen}^{K_{\infty}}(V) \otimes_{K_{\infty} \otimes A}(\hat{K}_{\infty}\otimes A)$$
is an isomorphism (cf. \cite[Section 3]{Mikami24}).

For an integer $n>0$, we define $D_{\Sen}^{K_n}(V)=D_{\Sen}^{K_{\infty}}(V)^{\Gamma_{K_n}\mathchar`-\an}.$
Then there exists $n(V)>0$ such that for any $n\geq n(V)$, the natural morphism 
$$D_{\Sen}^{K_n}(V)\otimes_{K_n\otimes A}(K_{\infty}\otimes A)\to D_{\Sen}^{K_{\infty}}(V)$$
is an isomorphism. 
In this case, $D_{\Sen}^{K_n}(V)$ a finite projective $K_n\otimes A$-module with a semilinear locally analytic $\Gamma_K$-action.
Let $n\geq n(V)$ or $n=\infty$.
Since the action of $\Gamma_K$ on $D_{\Sen}^{K_n}(V)$ is locally analytic, we get an action of $\Lie \Gamma_K$ on $D_{\Sen}^{K_n}(V)$.
The morphism $\Gamma_K\xrightarrow{\chi} \Zbb_p^{\times} \xrightarrow{\log} \Zbb_p$ defines a generator $\Theta \in \Lie \Gamma_K$.
Therefore, we get a $K_n\otimes A$-linear operator
$$\Theta_{\Sen}\colon D_{\Sen}^{K_n}(V) \to D_{\Sen}^{K_n}(V),$$
and we call this operator the \textit{Sen operator} of $V$.
\end{construction}

\begin{remark}
	Let $V$ be a $G_K$-equivariant vector bundle over $X_{\Cbb_p,\Acal}$.
	Then we simply write $D_{\Sen}^{K_n}(V)$ for $D_{\Sen}^{K_n}(V_{\infty})$.
	For $n$ sufficiently large (including $n=\infty$) the finite projective $\tilde{B}^{[1,1]}_{\Cbb_p,A}$-module $V^{[1,1]}$ with a locally analytic $G_K$-action descends to the finite projective $B_{K,n,A}^{[1,1]}$-module $V_n$ with the $\Gamma_K$-action by \cite{Mikami24}.
	We write $K_m=B_{K,n}^{[1,1]}/t$.
	Then we have $D_{\Sen}^{K_m}(V)=V_n/t$.
\end{remark}
\begin{remark}\label{rem:Sen operator compatible}
	For $(\infty\geq)m>n$, the natural morphism 
	$$D_{\Sen}^{K_n}(V) \otimes_{K_n\otimes A} (K_m\otimes A)\to D_{\Sen}^{K_m}(V)$$ 
	is an isomorphism.
	Moreover, the following diagram is commutative:
	$$
	\xymatrix{
		D_{\Sen}^{K_n}(V) \ar[d]\ar[r]^-{\Theta_{\Sen}} & D_{\Sen}^{K_n}(V)\ar[d] \\
		D_{\Sen}^{K_m}(V) \ar[r]^-{\Theta_{\Sen}} & D_{\Sen}^{K_m}(V).
	}
	$$
\end{remark}

\begin{example}\label{ex:Sen operator}
	Let $V=(\Cbb_p\otimes A)(i)$.
	Then there is an isomorphism $D_{\Sen}^{K_n}(V)\cong (K_n\otimes A)(i)$, and the Sen operator $\Theta_{\Sen}\colon D_{\Sen}^{K_n}(V)\to D_{\Sen}^{K_n}(V)$ is the multiplication by $i$.
\end{example}

\begin{lemma}\label{lem:HT Sen}
	Let $V\in \FP_{\Cbb_p\otimes A}(G_K)$, and let $n\geq n(V)$.
	Then the natural morphism $D_{\Sen}^{K_{\infty}}(V)[t,t^{-1}]^{\Gamma_K}\to V_{\HT}^{G_K}=D_{\HT}(V)$ is an isomorphism.
	Moreover, the following are equivalent:
	\begin{enumerate}
	\item $V$ is Hodge-Tate.
	\item The natural morphism
	$$D_{\Sen}^{K_{\infty}}(V)[t,t^{-1}]^{\Gamma_K}\otimes_{K\otimes A} (K_{\infty}\otimes A)[t,t^{-1}]\to D_{\Sen}^{K_{\infty}}(V)[t,t^{-1}]$$
	is an isomorphism.
	\end{enumerate}
\end{lemma}
\begin{proof}
	By Construction, for every $n\in \Zbb$, there is an isomorphism
	$$t^nD_{\Sen}^{K_{\infty}}(V)=((t^nV)^{H_K})^{\Gamma_K\mathchar`-\la}.$$
	The claim easily follows from this.
\end{proof}

\begin{corollary}\label{cor:DHT relatively discrete}
	Assume $\Acal \in \AlgAff_{\Qbb_p}$.
	Then for $V\in \FP_{\Cbb_p\otimes A}(G_K)$, $D_{\HT}(V)$ is a relatively discrete finitely generated $K\otimes A$-module.
\end{corollary}
\begin{proof}
	Since $D_{\Sen}^{K_{\infty}}(V)[t,t^{-1}]$ is relatively discrete over $K\otimes A$, $D_{\Sen}^{K_{\infty}}(V)[t,t^{-1}]^{\Gamma_K}\cong V_{\HT}^{G_K}$ is also relatively discrete over $K\otimes A$ by \cite[Lemma 1.24]{Mikami24}.
	The finiteness follows from Lemma~\ref{lem:HT inj} and faithfully flatness of the $K\otimes A$-module $B_{\HT,A}$.
\end{proof}

\begin{proposition}\label{prop:Sen operator HT}
	Let $V\in \FP_{\Cbb_p\otimes A}(G_K)$, and let $n\geq n(V)$.
	Then $V$ is Hodge-Tate if and only if the natural morphism
	$$\bigoplus_{i\in \Zbb}\Ker(\Theta_{\Sen}-i\colon D_{\Sen}^{K_n}(V)\to D_{\Sen}^{K_n}(V))\to D_{\Sen}^{K_n}(V)$$
	is an isomorphism.
\end{proposition}
\begin{proof}
	The only if part follows from Proposition~\ref{prop:HT recover} and Example~\ref{ex:Sen operator}. Let us prove the if part.
	By Remark~\ref{rem:Sen operator compatible}, we may assume $n<\infty$.
	By Lemma~\ref{lem:potentially HT}, it is enough to show that $V$ is Hodge-Tate as a $G_{K_m}$-equivariant vector bundle for some $\infty>m\geq n$.
	Since there is an isomorphism of morphisms
	$$(\Theta_{\Sen}-i\colon D_{\Sen}^{K_n}(V)\to D_{\Sen}^{K_n}(V)) \cong (\Theta_{\Sen}\colon D_{\Sen}^{K_n}(V(-i))\to D_{\Sen}^{K_n}(V(-i)))$$
	by Example~\ref{ex:Sen operator}, it suffices to show that if $\Theta_{\Sen}=0$ then $V$ is a Hodge-Tate $G_{K_m}$-equivariant vector bundle of Hodge-Tate weight $0$.
	Since $\Theta_{\Sen}=0$, the action of $\Gamma_{K_n}$ on $D_{\Sen}^{K_n}(V)$ is smooth.
	Since $D_{\Sen}^{K_n}(V)$ is a finite projective $K_n\otimes A$-module, there exists $m\geq n$ such that the action of $\Gamma_{K_m}$ on $D_{\Sen}^{K_n}(V)$ is trivial.
	By Remark~\ref{rem:Sen operator compatible}, the action of $\Gamma_{K_m}$ on $D_{\Sen}^{K_m}(V)$ is also trivial.
	Then the claim easily follows from Lemma~\ref{lem:HT Sen}.
\end{proof}

\begin{corollary}\label{cor:Sen op HT}
	Let $V\in \FP_{\Cbb_p\otimes A}(G_K)$, let $n\geq n(V)$, and let $a\geq b$ be integers.
	Then $V$ is Hodge-Tate with Hodge-Tate weights in $[a,b]$ if and only if the operator $\prod_{i=a}^b(\Theta_{\Sen}-i)$ on $D_{\Sen}^{K_n}(V)$ is $0$.
\end{corollary}

\begin{corollary}\label{cor:pointwise criterion of HT}
	Assume $\Acal \in \AlgAff_{\Qbb_p}$.
	Let $V\in \FP_{\Cbb_p\otimes A}(G_K)$, and let $a\leq b$ be integers.
	Then $V$ is Hodge-Tate with Hodge-Tate weights in $[a,b]$ if and only if for any $\mfrak \in \Spm A$ and $r\geq 1$, $V/\mfrak^r$ is Hodge-Tate with Hodge-Tate weights in $[a,b]$.
	If $A$ is reduced, it suffices to consider the case $r=1$.
\end{corollary}
\begin{proof}
	The natural morphism
	$$D_{\Sen}^{K_n}(V)\to \prod_{\mfrak,r}D_{\Sen}^{K_n}(V/\mfrak^r)$$
	is injective, we get the claim.
	If $A$ is reduced, then the natural morphism
	$$D_{\Sen}^{K_n}(V)\to \prod_{\mfrak}D_{\Sen}^{K_n}(V/\mfrak)$$
	is also injective.
	Therefore, it suffices to consider the case $r=1$.
\end{proof}

For later use, we prove the following lemma.
\begin{lemma}\label{lem:HT weights range}
	Assume $\Acal \in \AlgAff_{\Qbb_p}$.
	Then for $V \in \FP_{\Cbb_p\otimes A}(G_K)$, there exist integers $a\leq b$ such that for any integer $i\notin [a,b]$, $V(i)^{G_K}=0$.
\end{lemma}
\begin{proof}
	We take an integer $n\geq n(V)$. 
	First, we prove the following claim.

	\noindent\textbf{Claim:} For $i\in \Zbb$, if $\Theta_{\Sen}-i \colon D_{\Sen}^{K_n}(V)\to D_{\Sen}^{K_n}(V)$ is injective, then $V(i)^{G_K}=0$.

	\begin{proof}[Proof of the claim]
	In the same way as in the proof of Proposition~\ref{prop:Sen operator HT}, we may assume $i=0$. 
	By Remark~\ref{rem:Sen operator compatible}, $\Theta \colon D_{\Sen}^{K_{\infty}}(V)\to D_{\Sen}^{K_{\infty}}(V)$ is also injective.
	Assume that $V^{G_K}\neq 0$ and take $x\in V^{G_K}\setminus \{0\}$.
	Then $x\in D_{\Sen}^{K_{\infty}}(V)$, and the action of $\Gamma_K$ on $x$ is trivial.
	Therefore, $\Theta_{\Sen}(x)=0$, which is a contradiction.
	\end{proof}

	We write a (relatively discrete) $K_n\otimes A$-module $M_i=\Ker (\Theta_{\Sen}-i \colon D_{\Sen}^{K_n}(V)\to D_{\Sen}^{K_n}(V))$.
	Then the natural morphism $\bigoplus_{i\in \Zbb}M_i\to D_{\Sen}^{K_n}(V)$ is injective.
	Since $K_n\otimes A$ is noetherian and $D_{\Sen}^{K_n}(V)$ is (relatively discrete and) finitely generated over $K_n\otimes A$, for $i$ sufficiently large or small, we get $M_i=0$.
	By the claim, for $i$ sufficiently large or small, we obtain $V(i)^{G_K}=0$.
\end{proof}


\subsection{Sen polynomials}
\begin{lemma and definition}
Let $V\in \FP_{\Cbb_p\otimes A}(G_K)$, and let $\Theta_{\Sen}\colon D_{\Sen}^{K_n}(V)\to D_{\Sen}^{K_n}(V)$ be the Sen operator.
Assume that $V$ is of constant rank $r$.
Then the characteristic polynomial $P_V(T) \in (K_n\otimes A)[T]$ of $\Theta_{\Sen}$ lies in $(K\otimes A)[T]$ and it is independent of the choice of $n$.
We call it the \textit{Sen polynomial} of $V$.
\end{lemma and definition}
\begin{proof}
	The independence of $n$ follows from Remark~\ref{rem:Sen operator compatible}.
	Since the actions of $\Lie \Gamma_K$ and $\Gamma_K$ on $D_{\Sen}^{K_n}(V)$ commute, the coefficients of $P_V(T)$ are $\Gamma_K$-invariant.
	Thus, we have $P_V(T)\in (K\otimes A)[T].$
\end{proof}

Since the Sen polynomial comes from the action of $\Zbb_p=\Lie \Gamma_K$ on $D_{\Sen}^{K_n}(V)$, it should have an ``analytic nature''.
We make this observation precise.
Let $\Acal=(A,A^+)_{\square}\in \AlgAff_{\Qbb_p}.$ 
First, we recall some notions defined in \cite{Mikami24}.

\begin{definition}[{\cite[Definition 2.7]{Mikami24}}]
    Let $A$ be an algebraic-affinoid $\Qbb_{p,\square}$-algebra.
    We define $A^{\circ}(\ast)\subset A(\ast)$ as the subring of $f\in A(\ast)$ such that $A$ is $(\Zbb[T],\Zbb[T])_{\square}$-complete when viewed as a $\Zbb[T]$-module via $T\mapsto f$.
    We set $A^{b}(\ast)=A^{\circ}(\ast)[1/p]\subset A(\ast)$.
    Then we define the condensed subring $A^{b}\subset A$ to be the largest one whose underlying discrete subring is $A^{b}(\ast)$.
\end{definition}

\begin{definition}[{\cite[Definition 2.15]{Mikami24}}]
    Let $A$ be an algebraic affinoid $\Qbb_{p,\square}$-algebra.
    An \textit{affinoid $\Qbb_p$-algebra of definition of $A$} is a subring $A'\subset A$ such that $A'$ is an affinoid $\Qbb_p$-algebra and $A$ is relatively discrete and finitely generated over $A'$.
\end{definition}

\begin{lemma}[{\cite[Lemma 2.11, Proposition 2.13]{Mikami24}}]\label{lemma:bounded reduced}
    Let $A$ be an algebraic affinoid $\Qbb_{p,\square}$-algebra.
    \begin{enumerate}
        \item Let $A'$ be an affinoid $\Qbb_p$-algebra of definition of $A$.
		Then $A'\subset A^b$ and $A^b$ is relatively discrete over $A'$.
		Moreover, $A^b$ is the integral closure of $A'$ in $A$.
		\item The nilradical $\sqrt{(0)}$ of $A$ is contained in $A^b$ and $A^b/\sqrt{(0)}=(A/\sqrt{(0)})^b.$
    \end{enumerate}
\end{lemma}

\begin{corollary}\label{cor:bounded injective}
	Let $A$ be an algebraic affinoid $\Qbb_{p,\square}$-algebra, and let $B$ be a relatively discrete and finitely generated $B$-algebra.
	Assume that $A\to B$ is injective.
	Then we have $A^b=A\cap B^b.$
\end{corollary}
\begin{proof}
	Let $A'$ be an affinoid $\Qbb_p$-algebra of definition of $A$.
	Then $B$ is also relatively discrete $A'$-algebra.
	Thus, $A'$ is an affinoid $\Qbb_p$-algebra of definition of $B$.
	Therefore, $A^b$ (resp. $B^b$) is the integral closure of $A'$ in $A$ (resp. in $B$).
	Since $A\to B$ is injective, the claim easily follows from this.
\end{proof}

\begin{corollary}\label{cor:bounded scalar extension}
	Let $A$ be an algebraic affinoid $\Qbb_{p,\square}$-algebra, and let $L/\Qbb_p$ be a finite extension.
	Then the natural morphism
	$$L\otimes A^b \to (L\otimes A)^b$$
	is an isomorphism.
\end{corollary}
\begin{proof}
	Let $A'$ be an affinoid $\Qbb_p$-algebra of definition of $A$.
	Then $L\otimes A'$ is an affinoid $\Qbb_p$-algebra of definition of $L\otimes A$, and the integral closure of $L\otimes A'$ in $L\otimes A$ is equal to $L\otimes A^b$.
	The claim follows from this.
\end{proof}

\begin{lemma}\label{lem:stable dense}
	Let $G$ be a profinite group, and $H\subset G$ be a dense subgroup.
	Let $A$ be an algebraic affinoid $\Qbb_{p,\square}$-algebra, and $M$ be a finitely generated $A$-module with a (continuous) $G$-action.
	If a submodule $N\subset M$ is stable under the action of $H$, then it is stable under the action of $G$.
\end{lemma}
\begin{proof}
	Since $A$ and $M$ are nuclear $\Qbb_{p,\square}$-modules, by \cite[Lemma 1.48]{Mikami24}, the action of $G$ on $M$ corresponds to the morphism $\rho\colon M\to C(G,M);\; m\mapsto (g\mapsto gm).$
	We need to show that the image of $N\subset M$ under $\rho$ is contained in $C(G,N)$.
	We set $L=M/N$.
	Since $\Qbb_{p,\square}[G]$ is a projective $\Qbb_{p,\square}$-module, there is an exact sequence
	$$0\to C(G,N) \to C(G,M)\to C(G,L)\to 0.$$
	Therefore, it is enough to show that the composition
	$$N\xrightarrow{\rho} C(G,M)\to C(G,L)$$
	is zero.
	By assumption, it becomes zero after composing with $C(G,L) \to \prod_{h\in H}L;\; f\mapsto (f(h))_h$, where $C(G,L)\to L;\; f\mapsto f(h)$ is the morphism induced from $\Qbb_p=\Qbb_{p,\square}[\{h\}]\to \Qbb_{p,\square}[G].$
	Therefore, it suffices to show that $C(G,L) \to \prod_{h\in H}L$ is injective.
	Since $L$ is a finitely generated $A$-module and $A$ is an algebraic affinoid $\Qbb_{p,\square}$-algebra, $L$ can be written as the filtered colimit $L=\varinjlim_{\lambda}L_{\lambda}$ of Banach $\Qbb_p$-submodules $L_{\lambda}$.
	Since $\Qbb_{p,\square}[G]$ is a compact $\Qbb_{p,\square}$-module, we have 
	$$C(G,L)=\varinjlim_{\lambda} C(G,L_\lambda).$$
	Moreover, the natural morphism $\prod_{h\in H}L_{\lambda}\to \prod_{h\in H}L$ is injective.
	Thus, we may assume that $L$ is a Banach $\Qbb_p$-module.
	In this case, $C(G,L)$ is a Banach $\Qbb_p$-module of continuous functions on $G$ with values in $L$. 
	Since $H$ is dense in $G$, $C(G,L) \to \prod_{h\in H}L$ is injective as a morphism of topological $\Qbb_p$-vector spaces.
	By noting that the condensification functor preserves injectivity, we get the claim.
\end{proof}

\begin{theorem}\label{thm:HTS wts are analytic}
	Let $A$ be an algebraic affinoid $\Qbb_{p,\square}$-algebra.
	Let $V\in \FP_{\Cbb_p\otimes A}(G_K)$ be a finite projective $\Cbb_p\otimes A$-module of constant rank with a semilinear $G_K$-action, and let $P_V(T)\in (K\otimes A)[T]$ be the Sen polynomial.
	Then we have $P_V(T)\in (K\otimes A^b)[T]$.
\end{theorem}
\begin{proof}
	We take $D_{\Sen}^{K_n}(V)$ with the semilinear $\Gamma_K$-action as in Construction~\ref{construction:Sen}.
	By Corollary~\ref{cor:bounded injective} and Corollary~\ref{cor:bounded scalar extension}, it is enough to show $P_V(T)\in (K_n\otimes A)^b[T].$
	To define the Sen operator, it is enough to consider the action of $\Gamma_{K_n}$.
	Therefore, the claim reduces to the following theorem.
\end{proof}

\begin{theorem}
	Let $\Gamma=\Zbb_p$ (to avoid confusion).
	Let $A$ be an algebraic-affinoid $\Qbb_{p,\square}$-algebra, and $V$ be a finite projective $A$-module of constant rank with a locally analytic $\Gamma$-action.
	Then the characteristic polynomial of the action of $1\in \Zbb_p=\Lie \Gamma$ on $V$ lies in $A^b[T].$
\end{theorem}
\begin{proof}
By Lemma~\ref{lemma:bounded reduced}, we may assume that $A$ is reduced.
There is an injective morphism $A\to \prod_{i=1}^n A[f_i^{-1}]$ such that the scalar extension of $V$ along this morphism is free.
Therefore, by Lemma~\ref{cor:bounded injective}, we may assume that $V$ is free.
We set $V=A^d$.
Let $\gamma=1\in \Gamma=\Zbb_p$, and let $\alpha\in GL_d(A)$ be a matrix corresponding to $\gamma\colon A^d\to A^d.$
Then there is a finite extension $B'$ of $\Frac A$ (the total ring of fractions of $A$) such that $\alpha$ is triangularizable over $B'$.
Let $B\subset B'$ be an $A$-subalgebra generated by the finite number of elements necessary to triangularize $\alpha$.
By Lemma~\ref{cor:bounded injective}, we may assume that $\alpha$ is triangularizable.
Then by Lemma~\ref{lem:stable dense}, there is a $\Gamma$-stable full flag of $V=A^d$.
Therefore, we may assume that $V=A$.
Then the action of $\Gamma$ on $A$ corresponds to a character $\Gamma\to A^{\times}.$
By \cite[Proposition 5.2]{Mikami24}, this character factors through $A'^{\times}\subset A^{\times}$ for some affinoid $\Qbb_p$-algebra $A'$ of definition of $A$.
Then we have $P_V(T)\in A'[T]\subset A^b[T].$
\end{proof}


\section{De Rham representations}
\subsection{Definitions and basic properties of de Rham representations}

\begin{definition}
We define $B_{\dR,A}^+\coloneqq \varprojlim_{n}\tilde{B}_{\Cbb_p,A}^{[1,1]}/t^n$ and $B_{\dR,A}\coloneqq B_{\dR,A}^+[1/t]$.
As usual, $B_{\dR,A}$ has a filtration $\Fil^{\bullet} B_{\dR,A}$ defined via $\Fil^n B_{\dR,A}=t^nB_{\dR,A}^+$.
\end{definition}

\begin{remark}\label{rem:de Rham complete}
	By \cite[Lemma 3.28]{RJRC22}, if $A$ is a Banach $\Qbb_p$-algebra, there is an isomorphism
	\begin{align*}
	B_{\dR,A}^+ &\cong \varprojlim_{n} (B_{\dR}^+/t^n\otimes A)\\
	&\cong B_{\dR}^+\otimes A.
	\end{align*}
\end{remark}

\begin{lemma}\label{lem:dR flat}
If $A$ is a Banach $\Qbb_p$-algebra, then $B_{\dR,A}$ is faithfully flat over $(K\otimes A,\Zbb)_{\square}$.
\end{lemma}
\begin{proof}
	By Remark~\ref{rem:de Rham complete}, we may assume $A=\Qbb_p$.
	Then the claim follows from \cite[Lemma 3.21]{RJRC22}.
\end{proof}

\begin{lemma}
	We have $B_{\dR,A}^{G_K}=K\otimes A$.
\end{lemma}
\begin{proof}
	We have the following isomorphism
	\begin{align*}
		B_{\dR,A}^{G_K}\cong\varinjlim_m \varprojlim_{m<n} (t^mB_{\dR,A}/t^nB_{\dR,A})^{G_K}.
	\end{align*}
	Therefore, it is enough to show 
	\begin{align*}
	(t^mB_{\dR,A}/t^nB_{\dR,A})^{G_K}=
	\begin{cases*}
		A & if $0\in [m,n]$,\\
		0 & if $0\notin [m,n].$
	\end{cases*}
	\end{align*}
	It can be proved in the same way as in Lemma~\ref{lem:HT fixed vector}.
\end{proof}

\begin{definition}
	\begin{enumerate}
\item Let $V^+$ be a finite projective $B_{\dR,A}^+$-module with a semilinear $G_K$-action.
We define $D_{\dR}^K(V^+)\coloneqq V^+[1/t]^{G_K}$, which admits a filtration $\Fil^{\bullet}D_{\dR}^K(V^+)$ defined via $\Fil^{n}D_{\dR}^K(V^+)\coloneqq(t^nV^+)^{G_K}$.
When $K$ is clear from the context, we omit $K$ from the notation.
We say that $V$ is \textit{de Rham} if the natural morphism
$$D_{\dR}(V^+)\otimes_{K\otimes A} B_{\dR,A}\to V^+[1/t]$$
is an isomorphism.
Let $\FP_{B_{\dR,A}^+}(G_K)$ denote the category of finite projective $B_{\dR,A}^+$-modules with a semilinear $G_K$-action, and $\FP_{B_{\dR,A}^+}^{\dR}(G_K)$ denote the subcategory of $\FP_{B_{\dR,A}^+}(G_K)$ consisting of de Rham representations.
\item Let $V$ be a $G_K$-equivariant vector bundle over $X_{\Cbb_p,\Acal}$.
We define $V^+_{\dR},V_{\dR}$ as
\begin{align*}
	&V^+_{\dR}\coloneqq V^{[1,1]}\otimes_{\tilde{B}^{[1,1]}_{\Cbb_p,A}}B_{\dR,A}^+,\\
	&V_{\dR}\coloneqq V^{[1,1]}\otimes_{\tilde{B}^{[1,1]}_{\Cbb_p,A}}B_{\dR,A}=V^+_{\dR}[1/t]
\end{align*}
which is a finite projective $B_{\dR,A}^+$-module (resp. $B_{\dR,A}$-module) with $G_K$-action.
We say that $V$ is \textit{de Rham} if $V_{\dR}$ is de Rham.
In this case, we simply write $D_{\dR}(V)$ (resp. $\Fil^{n}D_{\dR}(V)$) for $D_{\dR}(V_{\dR}^+)$ (resp. $\Fil^{n}D_{\dR}(V_{\dR}^+)$).
Let $\Vect^{\dR}(X_{\Cbb_p,\Acal}/G_K)$ denote the category of de Rham $G_K$-equivariant vector bundles over $X_{\Cbb_p,\Acal}$.
\end{enumerate}
\end{definition}

\begin{remark}
	For $V\in \Vect(X_{\Cbb_p,\Acal}/G_K)$ and a closed interval $1\in [r,s]\subset (0,\infty)$, there is a natural isomorphism
	$$V^{[r,s]}\otimes_{\tilde{B}^{[r,s]}_{\Cbb_p,A}}B_{\dR,A}^+\cong V^{[1,1]}\otimes_{\tilde{B}^{[1,1]}_{\Cbb_p,A}}B_{\dR,A}^+=V^+_{\dR}.$$
\end{remark}

\begin{lemma}
Assume that $A$ is a Banach $\Qbb_p$-algebra.
Let $V^+\in \FP_{B_{\dR,A}^+}^{\dR}(G_K)$.
Then $D_{\dR}(V^+)$ is a finite projective $K\otimes A$-modules.
In particular, for $n$ sufficiently large, $\Fil^n D_{\dR}(V^+)=D_{\dR}(V^+).$
\end{lemma}
\begin{proof}
	By using Lemma~\ref{lem:dR flat}, we can prove the lemma in the same way as in Lemma~\ref{lem:DHT fin proj}.
\end{proof}

In the case where $\Acal\in \AlgAff_{\Qbb_p}$, the proof of the above lemma is more technical because we do not know whether Lemma~\ref{lem:dR flat} holds in this case.

\begin{lemma}\label{lem:DdR relatively discrete}
	Assume $\Acal \in \AlgAff_{\Qbb_p}$.
	Then for $V^+\in \FP_{B_{\dR,A}^+}(G_K)$, $D_{\dR}(V^+)$ is a relatively discrete finitely generated $K\otimes A$-module.
\end{lemma}
\begin{proof}
	By Lemma~\ref{lem:HT weights range}, there are integers $a \leq b$ such that for any integer $i\notin [a,b]$, $(V^+/tV^+)(i)^{G_K}=0$.
	First, let us prove $(t^{-b}V^+)^{G_K}\cong D_{\dR}(V^+).$
	By applying $(-)^{G_K}$ to an exact sequence
	$$0\to t^{-b}V^+ \to t^{-b-1}V^+ \to t^{-b-1}V^+/t^{-b}V^+\to 0,$$
	we get an isomorphism $(t^{-b}V^+)^{G_K} \cong (t^{-b-1}V^+)^{G_K}.$
	By repeating this argument, we get an isomorphism $(t^{-b}V^+)^{G_K}\cong (t^{-c}V^+)^{G_K}$ for any $c\geq b$.
	Since $D_{\dR}(V^+)=\varinjlim_{c\geq b} (t^{-c}V^+)^{G_K}$, we get an isomorphism 
	$(t^{-b}V^+)^{G_K}\cong D_{\dR}(V^+).$
	
	Next, let us prove $(t^{-a+1}V^+)^{G_K}=0$.
	By applying $(-)^{G_K}$ to an exact sequence
	$$0\to t^{-a+2}V^+/t^{-a+3}V^+ \to t^{-a+1}V^+/t^{-a+3} \to t^{-a+1}V^+/t^{-a+2}V^+\to 0,$$
	we get $(t^{-a+1}V^+/t^{-a+3})^{G_K}=0$.
	By repeating this argument, we get $$(t^{-a+1}V^+/t^{-c})^{G_K}=0$$ for any $c<a$.
	Therefore, we get $$(t^{-a+1}V^+)^{G_K} \cong \varprojlim_{c<a}(t^{-a+1}V^+/t^{-c})^{G_K}=0.$$
	
	Finally, let us prove $(t^{-a-i}V^+)^{G_K}$ is relatively discrete for any $i\geq -1$.
	We proceed by induction on $i$.
	When $i=-1$, then $(t^{-a-i}V^+)^{G_K}=0$ is relatively discrete.
	In general, from an exact sequence 
	$$0\to t^{-a-i}V^+ \to t^{-a-i-1}V^+ \to t^{-a-i-1}V^+/t^{-a-i}V^+\to 0,$$
	we obtain the long exact sequence in cohomology
	\begin{align*}
			0\to &(t^{-a-i}V^+)^{G_K} \to (t^{-a-i-1}V^+)^{G_K} \to (t^{-a-i-1}V^+/t^{-a-i}V^+)^{G_K}\\
			\to &H^1(G_K,t^{-a-i}V^+).
	\end{align*}
	By Corollary~\ref{cor:DHT relatively discrete}, $(t^{-a-i-1}V^+/t^{-a-i}V^+)^{G_K}$ is relatively discrete.
	By \cite[Lemma 1.24]{Mikami24}, the kernel of $(t^{-a-i-1}V^+/t^{-a-i}V^+)^{G_K}\to H^1(G_K,t^{-a-i}V^+)$ is also relatively discrete.
	Since an extension of relatively discrete modules is also relatively discrete, we find that $(t^{-a-i-1}V^+)^{G_K}$ is also relatively discrete.
\end{proof}

\begin{lemma}\label{lem:BdR flatness noncondensed}
Assume $\Acal \in \AlgAff_{\Qbb_p}$.
Then $(K\otimes A)(\ast)\to B_{\dR,A}(\ast)$ is faithfully flat as a morphism of usual (i.e., non-condensed) rings.
\end{lemma}
\begin{proof}
	To simplify the notation, we omit $(\ast)$.
	First, we prove that $K\otimes A \to  B_{\dR,A}^+$ is flat.
	There is a factorization 
	$$K\otimes A \to (K\otimes A)[[t]] \to B_{\dR,A}^+.$$
	Since $K\otimes A$ is noetherian, $K\otimes A \to (K\otimes A)[[t]]$ is flat.
	Let us prove that $(K\otimes A)[[t]] \to B_{\dR,A}^+$ is also flat.
	Since $B_{\dR,A}^+$ is a $t$-adic completion of the noetherian ring $\tilde{B}_{\Cbb_p,A}^{[1,1]}$, $B_{\dR,A}^+$ is also noetherian. 
	Moreover $t$ is a non-zero divisor in $B_{\dR,A}^+$.
	Therefore, by \cite[Proposition 8.3.8]{FK18}, it suffices to show that $K\otimes A=(K\otimes A)[[t]]/t \to B_{\dR,A}^+/t=\Cbb_p\otimes_{\Qbb_{p,\square}}A$ is flat.
	Let us take an affinoid $\Qbb_p$-algebra $R$ and a morphism $R\to A$ such that $A$ is relatively discrete over $R$.
	Then $K\otimes A\to \Cbb_p\otimes_{\Qbb_{p,\square}}A$ is a base change (as a morphism of usual rings) of $K\otimes R\to \Cbb_p\otimes_{\Qbb_{p,\square}}R$, so we may assume that $A$ is already an affinoid $\Qbb_p$-algebra.
	Let us take a surjection $\Qbb_p\langle T_1,\ldots, T_n\rangle \to A$.
	Then $K\otimes A\to \Cbb_p\otimes_{\Qbb_{p,\square}}A$ is a base change (as a morphism of usual rings) of 
	\begin{align}\label{eq:flat}
		K\langle T_1,\ldots, T_n\rangle\to \Cbb_p\langle T_1,\ldots, T_n\rangle,
	\end{align} 
	so it suffices to show that \eqref{eq:flat} is flat.
	It reduces to showing that 
	$$\Ocal_K\langle T_1,\ldots, T_n\rangle\to \Ocal_{\Cbb_p}\langle T_1,\ldots, T_n\rangle$$
	is flat.
	By using \cite[Proposition 8.3.8]{FK18} again, it further reduces to showing that 
	$$\Ocal_K/\pi_K[T_1,\ldots, T_n]\to \Ocal_{\Cbb_p}/\pi_K[T_1,\ldots, T_n],$$
	where $\pi_K$ is a uniformizer of $K$, is flat, which is clear.

	From the above argument, we find that $K\otimes A \to  B_{\dR,A}$ is flat.
	We prove that the image of $\Spec B_{\dR,A} \to \Spec K\otimes A$ contains $\Spm K\otimes A$.
	For any maximal ideal $\mfrak \subset A$, we get a diagram
	\begin{align*}
		\xymatrix{
		K\otimes A \ar[r]\ar[d]&  B_{\dR,A}\ar[d]\\
		K\otimes A/\mfrak \ar[r] &  B_{\dR,A/\mfrak}.}
	\end{align*}
	Therefore, it is enough to show that under the assumption that $A$ is a finite extension field of $\Qbb_p$, $\Spec B_{\dR,A} \to \Spec K\otimes A$ is surjective.
	Since $K\otimes A\to B_{\dR,A}$ is a base change (as a morphism of usual rings) of $K\to B_{\dR}$, the claim is clear.
\end{proof}

\begin{proposition}\label{prop:dR fil exhaustive}
Let $V^+\in \FP_{B_{\dR,A}^+}^{\dR}(G_K)$.
Then $D_{\dR}(V^+)$ is a finite projective $K\otimes A$-modules.
In particular, for $n$ sufficiently large, $\Fil^n D_{\dR}(V^+)=D_{\dR}(V^+).$
\end{proposition}
\begin{proof}
	Let $\Cond_{-}(-)$ be the condensification functor defined in \cite[Definition 1.2]{Mikami24}.
	In the case where $A$ is a Banach $\Qbb_p$-algebra, the lemma is already proved, so we may assume $\Acal \in \AlgAff_{\Qbb_p}$.
	By Lemma~\ref{lem:DdR relatively discrete}, we get an isomorphism 
	$$D_{\dR}(V^+)\otimes_{K\otimes A} B_{\dR,A}\cong \Cond_{B_{\dR,A}}(D_{\dR}(V^+)(\ast)\otimes_{(K\otimes A)(\ast)} B_{\dR,A}(\ast)).$$
	Since $V^+[1/t]=D_{\dR}(V^+)\otimes_{K\otimes A} B_{\dR,A}$ is a finite projective $B_{\dR,A}$-module, the $B_{\dR,A}(\ast)$-module $D_{\dR}(V^+)(\ast)\otimes_{(K\otimes A)(\ast)} B_{\dR,A}(\ast)$ is also finite projective.
	By Lemma~\ref{lem:BdR flatness noncondensed}, $D_{\dR}(V^+)(\ast)$ is also a finite projective $(K\otimes A)(\ast)$-module.
	By Lemma~\ref{lem:DdR relatively discrete}, we have $D_{\dR}(V^+)\cong \Cond_{K\otimes A}(D_{\dR}(V^+)(\ast))$, and it is also a finite projective $K\otimes A$-module.
\end{proof}

Thanks to the above proposition, we can prove the following four claims in the same way as in the Hodge-Tate case.

\begin{lemma}\label{lem:dR dual direct sum tensor product}
The subcategory $\FP_{B_{\dR,A}^+}^{\dR}(G_K)\subset \FP_{B_{\dR,A}^+}(G_K)$ is stable under taking duals, direct sums, and tensor products.
\end{lemma}

\begin{lemma}\label{lem:dR basechange}
For $V^+\in \FP_{B_{\dR,A}^+}^{\dR}(G_K)$ and a morphism $\Acal\to \Bcal=(B,B^+)_{\square}$, $V^+_B=V^+\otimes_{B_{\dR,A}^+}B_{\dR,B}^+$ is de Rham.
Moreover, in this case, there is an isomorphism of $K\otimes B$-modules
$$D_{\dR}(V^+_B)\cong D_{\dR}(V^+)\otimes_{K\otimes A} (K\otimes B).$$
\end{lemma}

\begin{lemma}\label{lem:dR sub quot easy}
	Let $V\in \FP^{\dR}_{B_{\dR,A}^+}(G_K)$ and let $W\subset V$ be a $G_K$-stable submodule such that $W$ and $V/W$ are also finite projective.
	If $W$ or $V/W$ is de Rham, then the other one is also de Rham.
\end{lemma}

\begin{lemma}\label{lem:potentially dR}
Let $L/K$ be a finite extension.
Then for $V^+\in \FP_{B_{\dR,A}^+}^{\dR}(G_K)$, $V^+$ is de Rham as a $G_K$-representation if and only if it is de Rham as a $G_L$-representation.
\end{lemma}

We can show that ``de Rham'' implies ``Hodge-Tate'' as follows.

\begin{proposition}\label{prop:dR implies HT}
	We assume $\Acal \in \AlgAff_{\Qbb_p}$.
	Then for $V^+\in \FP_{B_{\dR,A}^+}^{\dR}(G_K)$, the $B_{\dR,A}^+$-module $V^+\otimes_{B_{\dR,A}^+} (\Cbb_p\otimes A)=V^+/t$ is Hodge-Tate.
\end{proposition}
\begin{proof}
	By Proposition~\ref{prop:dR fil exhaustive} and Lemma~\ref{lem:dR dual direct sum tensor product}, there exist integers $a\leq b$ such that 
	$$D_{\dR}(V^+)=\Fil^a D_{\dR}(V^+),\quad D_{\dR}((V^+)^*)=\Fil^{-b} D_{\dR}((V^+)^*),$$ where $(V^+)^*$ is the dual of $V^+.$
	Let $\mfrak$ be a maximal ideal of $A$, and let $r\geq 1$.
	Since there is a commutative diagram
	$$
	\xymatrix{
		\Fil^a D_{\dR}(V^+)\otimes_{A} A/\mfrak^r \ar[r]^-{\cong}\ar[d] & D_{\dR}(V^+)\otimes_{A} A/\mfrak^r\ar[d]^-{\cong}\\
		\Fil^a D_{\dR}(V^+/\mfrak^r) \ar[r] & D_{\dR}(V^+/\mfrak^r),\\
	}
	$$
	we obtain $$\Fil^a D_{\dR}(V^+/\mfrak^r)=D_{\dR}(V^+/\mfrak^r).$$
	Similarly, we have $$\Fil^{-b} D_{\dR}((V^+/\mfrak^r)^*)=D_{\dR}((V^+/\mfrak^r)^*).$$
	Since $A/\mfrak^r$ is a finite $\Qbb_p$-algebra, $V^+/\mfrak^r\otimes_{B_{\dR,A/\mfrak^r}^+}(\Cbb_p\otimes A/\mfrak^r)$ is Hodge-Tate with Hodge-Tate weights in $[a,b]$ by the classical theory.
	Then the claim follows from Corollary~\ref{cor:pointwise criterion of HT}.
\end{proof}

\begin{remark}
	We expect that this proposition holds without the assumption $\Acal \in \AlgAff_{\Qbb_p}$, but we do not know how to prove it.
\end{remark}

Next, we consider the correspondence between de Rham representations and filtered modules.
\begin{lemma}\label{lem:Tate twsit cohomology}
	Let $N$ be a nuclear (static) $K\otimes A$-module.
	Then for $n\neq 0$, we have
	\begin{align*}
	R\Gamma(G_K, N\otimes_{K\otimes A}(\Cbb_p\otimes A)(n))=
	\begin{cases*}
	N\oplus N[-1] & if $n=0,$\\
	0 & if $n\neq 0$.
	\end{cases*}
    \end{align*}
\end{lemma}
\begin{proof}
	We have 
	\begin{align*}
	R\Gamma(G_K, N\otimes_{K\otimes A}(\Cbb_p\otimes A)(n))&\cong R\Gamma(\Gamma_K,R\Gamma(H_K,N\otimes_{K}\Cbb_p(n)))\\
	&\cong R\Gamma(\Gamma_K,N\otimes_{K}\hat{K}_{\infty}(n))\\
	&\cong R\Gamma(\Gamma_K,\hat{K}_{\infty}(n))\otimes_{K}^{\Lbb} N,
	\end{align*}
	where the second isomorphism follows from Lemma~\ref{lem:Galois descent}, and the third isomorphism follows from \cite[Lemma 1.44]{Mikami24}.
	Since 
	\begin{align*}
	R\Gamma(\Gamma_K, \hat{K}_{\infty}(n))=
	\begin{cases*}
	K\oplus K[-1] & if $n=0,$\\
	0 & if $n\neq 0$,
	\end{cases*}
    \end{align*}
	we obtain the claim.
\end{proof}

\begin{proposition}\label{prop:filtered module BdR+}
Let $(M,\Fil^{\bullet}M)$ be a filtered $K\otimes A$-module satisfying
\begin{itemize}
	\item For $a$ sufficiently large and $b$ sufficiently small, $\Fil^a M=0$ and $\Fil^b M=M$.
	\item For every $n\in \Zbb$, $\gr^n M=\Fil^{n}M/\Fil^{n+1}M$ is a finite projective $K\otimes A$-module.
\end{itemize}
Then $\Fil^0 (M\otimes_{K\otimes A}^{\Lbb} B_{\dR,A})$, where $-\otimes_{K\otimes A}^{\Lbb}$ is the Day convolution in the $\infty$-category $\Fun(\Zbb^{\op},\Dcal(K\otimes A))$ (\cite[2.2.6]{HA}), is a finite projective $B_{\dR,A}^+$-module, and for every $n\in \Zbb$, we have $$\Fil^n(M\otimes_{K\otimes A}^{\Lbb} B_{\dR,A})=t^n\Fil^0(M\otimes_{K\otimes A}^{\Lbb} B_{\dR,A}).$$
\end{proposition}
\begin{proof}
	To simplify the notation, we write $F^n=\Fil^n(M\otimes_{K\otimes A}^{\Lbb} B_{\dR,A})$.
	First, let us prove that the filtered module $M\otimes_{K\otimes A}^{\Lbb} B_{\dR,A}$ is complete.
	By the usual induction, we may assume $0=\Fil^1 M\subset \Fil^0 M=M$ and $M$ is a finite projective $K\otimes A$-module.
	In this case, we need to show 
	$$R\varprojlim_{n} (M\otimes_{K\otimes A}^{\Lbb} t^nB_{\dR,A}^+)\cong 0.$$
	Since $M$ is a finite projective $K\otimes A$-module, it reduces to showing that
	$$R\varprojlim_{n} t^nB_{\dR,A}^+\cong 0,$$
	which follows from the $t$-adic completeness of $B_{\dR,A}^+$.

	Next, we have
	\begin{align*}
	\gr^n(M\otimes_{K\otimes A}^{\Lbb} B_{\dR,A})\cong &\bigoplus_{i+j=n} (\gr^i M)\otimes_{K\otimes A}^{\Lbb} \gr^j B_{\dR,A}\\
	\cong &\bigoplus_{i+j=n} (\gr^i M)\otimes_{K\otimes A}^{\Lbb} t^j(\Cbb_p\otimes A),
    \end{align*}
	which is static.
	Therefore, for $n\geq 0$, $\cofib(F^n\to F^0)$ is also static.
	By the completeness of $M\otimes_{K\otimes A}^{\Lbb} B_{\dR,A}$, we have
	\begin{align*}
	F^0\cong R\varprojlim_n \cofib(F^n\to F^0).
    \end{align*}
	Therefore, we get $F^0 \in \Dcal(K\otimes A)^{[0,1]}$.
	By the definition of the Day convolution, $F^0$ can be written as a colimit of $\Fil^i M\otimes_{K\otimes A}^{\Lbb} \Fil^{-i}B_{\dR,A}\in \Dcal(K\otimes A)^{\leq 0}$, so we get $F^0 \in \Dcal(K\otimes A)^{\leq 0}$.
	By combining them, we get 
	$$F^0 \in \Dcal(K\otimes A)^{\heart}.$$
	Moreover, since $\Fil^i M\otimes_{K\otimes A}^{\Lbb} \Fil^{-i}B_{\dR,A}$ is also relatively discrete over $B_{\dR,A}^+$, $F^0$ is relatively discrete over $B_{\dR,A}^+$.
	Since 
	\begin{align*}
	&(\cdots \to tB_{\dR,A}^+ \to B_{\dR,A}^+ \to t^{-1}B_{\dR,A}^+\to \cdots)\\
	\cong &(\cdots \xrightarrow{t} B_{\dR,A}^+ \xrightarrow{t} B_{\dR,A}^+ \xrightarrow{t} t^{-1}B_{\dR,A}^+\xrightarrow{t} \cdots),
    \end{align*}
	we have
	$$(\cdots \to F^1\to F^0 \to F^{-1}\to \cdots)\cong (\cdots \xrightarrow{t} F^0 \xrightarrow{t} F^0 \xrightarrow{t} F^0 \xrightarrow{t}\cdots).$$
	Since $\cofib (F^1\to F^0)\cong \cofib(t\colon F^0\to F^0)$ is static, $t$ is a non-zero divisor on $F^0$ and we can identify $F^n$ with $t^n F^0$.
	By the completeness of $M\otimes_{K\otimes A}^{\Lbb} B_{\dR,A}$, we find that $F^0$ is $t$-adically complete.
	Since $F^0/t\cong \bigoplus_{i} (\gr^i M)\otimes_{K\otimes A}^{\Lbb} t^{-i}(\Cbb_p\otimes A)$ is a finite projective $B_{\dR,A}^+/t=\Cbb_p \otimes A$-module, $F^0$ is also a finite projective $B_{\dR,A}^+$-module.
\end{proof}

\begin{definition}\label{def:Fil}
\begin{enumerate}
\item Let $\Fil_{K,A}$ denote the category of filtered $K\otimes A$-modules satisfying the conditions in Proposition~\ref{prop:filtered module BdR+}.
Since $\Fil_{K,A}$ is stable under the Day convolution $-\otimes_{K\otimes A}^{\Lbb}$, it becomes a symmetric monoidal category.
We denote this tensor product by $-\otimes_{K\otimes A}-$.
\item For $M\in \Fil_{K,A}$, we write $\Fil^n (M\otimes B_{\dR,A})=\Fil^n (M\otimes_{K\otimes A}^{\Lbb} B_{\dR,A})$, which is a finite projective $B_{\dR,A}^+$-module.
Moreover, from the construction, it has a natural semilinear $G_K$-action.
Therefore, we get a functor 
$$\Fil_{K,A} \to \FP_{B_{\dR,A}^+}(G_K);\; M\mapsto \Fil^0 (M\otimes B_{\dR,A}).$$
\end{enumerate}
\end{definition}
\begin{remark}
	By unwinding the definition, $\Fil^0 (M\otimes B_{\dR,A})$ is equal to the submodule 
	$$\sum_{i\in \Zbb} \Fil^i M \otimes_{K\otimes A} t^{-n}B_{\dR,A}^+\subset M\otimes_{K\otimes A}B_{\dR,A}.$$
	Moreover, we have $\Fil^0 (M\otimes B_{\dR,A})[1/t]=M\otimes_{K\otimes A}B_{\dR,A}.$
\end{remark}

\begin{proposition}\label{prop:filtered to cdR}
For $M\in \Fil_{K,A}$, $\Fil^0 (M\otimes B_{\dR,A}^+)$ is de Rham.
Moreover, there is a natural isomorphism 
$$M\cong D_{\dR}(\Fil^0 (M\otimes B_{\dR,A}))$$
of filtered $K\otimes A$-modules.
\end{proposition}
\begin{proof}
	From the isomorphism $\Fil^0 (M\otimes B_{\dR,A})[1/t]=M\otimes_{K\otimes A}B_{\dR,A}$, it follows that $\Fil^0 (M\otimes B_{\dR,A}^+)$ is de Rham, and we get a natural isomorphism
	$$M\cong D_{\dR}(\Fil^0 (M\otimes B_{\dR,A}))$$
	of $K\otimes A$-modules.
	Let us prove that the above is an isomorphism of filtered $K\otimes A$-modules.
	Since $\Fil^{n}M \otimes_{K\otimes A}  B_{\dR,A}^+ \subset \Fil^n (M\otimes B_{\dR,A})$, we get 
	\begin{align*}
		\Fil^{n}M&=(\Fil^{n}M \otimes_{K\otimes A}  B_{\dR,A}^+)^{G_K}\\
		&\subset \Fil^n (M\otimes B_{\dR,A})^{G_K}=\Fil^nD_{\dR}(\Fil^0 (M\otimes B_{\dR,A})).\end{align*}
	Therefore, $M\cong D_{\dR}(\Fil^0 (M\otimes B_{\dR,A}^+))$ defines a morphism 
	\begin{align}\label{eq:cdR filtered}
	M\to D_{\dR}(\Fil^0 (M\otimes B_{\dR,A}))
	\end{align}
	of filtered $K\otimes A$-modules.
	It suffices to show that 
	$$\gr^n M\to \gr^n D_{\dR}(\Fil^0 (M\otimes B_{\dR,A}))$$ is an isomorphism for every $n\in \Zbb$.
	We may assume $n=0$ (for simplicity of the notation).
	There is an exact sequence
	$$0\to \Fil^1(M\otimes B_{\dR,A}) \to \Fil^0(M\otimes B_{\dR,A}) \to \bigoplus_{i\in \Zbb} (\gr^i M)\otimes_{K\otimes A} t^{-i}(\Cbb_p\otimes A)\to 0.$$
	By taking $G_K$-invariant vectors, we obtain an exact sequence
	$$0\to \Fil^1D_{\dR}(\Fil^0 (M\otimes B_{\dR,A})) \to \Fil^0 D_{\dR}(\Fil^0 (M\otimes B_{\dR,A}))\to \gr^0M.$$
	Since $\Fil^0 M\subset \Fil^0 D_{\dR}(\Fil^0 (M\otimes B_{\dR,A}))$, the morphism 
	$$\Fil^0 D_{\dR}(\Fil^0 (M\otimes B_{\dR,A}))\to \gr^0M$$ is surjective.
	Therefore, we get an isomorphism $\gr^0D_{\dR}(\Fil^0 (M\otimes B_{\dR,A}))\cong \gr^0 M.$ 
	It is easy to show that the above isomorphism is induced from \eqref{eq:cdR filtered}.
\end{proof}

For later use, we prove the following lemma in a slightly more general setting. 
\begin{lemma}\label{lem:fixed vector modulo t^n}
Let $V^+$ be a (static) $t$-torsion free and $t$-adically complete $(B_{\dR,A}^+,\Zbb)_{\square}$-module with an action of $G_K$.
Assume that there exist integers $a\leq b$ and nuclear $K\otimes A$-modules $N_i$ for $a\leq i\leq b$ such that there is a $G_K$-equivariant decomposition
$$V^+/tV^+\cong \bigoplus_{i=a}^b N_i\otimes_{K\otimes A}(\Cbb_p\otimes A)(i).$$
Then $R\Gamma(G_K, t^nV^+)=0$ for every $n>-a$ and the natural morphisms
\begin{align*}
&R\Gamma(G_K, t^mV^+)\to R\Gamma(G_K, V^+[1/t]),\\
&R\Gamma(G_K, t^nV^+)\to R\Gamma(G_K, t^nV^+/t^{-a+1}V^+)
\end{align*}
are isomorphisms for any $m\leq -b$ and $n\leq -a$.
\end{lemma}
\begin{proof}
	By Lemma~\ref{lem:Tate twsit cohomology}, we have 
	$$R\Gamma(G_K,t^iV^+/t^{i+1}V^+)=0$$
	for $i>-a$ or $i<-b$.
	Therefore, by induction, the natural morphism
	$$R\Gamma(G_K, t^{n}V^+/t^iV^+)\to R\Gamma(G_K, t^{n}V^+/t^{-a+1}V^+)$$
	is an isomorphism for any $n\leq -a$ and $i>-a$.
	Since there is an isomorphism
	$$R\Gamma(G_K, t^{n}V^+)\cong R\varprojlim_{i} R\Gamma(G_K, t^{n}V^+/t^iV^+),$$
	the natural morphism
	$$R\Gamma(G_K, t^{n}V^+)\to R\Gamma(G_K, t^{n}V^+/t^{-a+1}V^+)$$
	is an isomorphism.
	Similarly, we can show $R\Gamma(G_K, t^nV^+)=0$ for every $n>-a$.
	By noting $R\Gamma(G_K, V^+[1/t])\cong \varinjlim_m R\Gamma(G_K, t^mV^+)$, we can show that 
	$$R\Gamma(G_K, t^mV^+)\to R\Gamma(G_K, V^+[1/t])$$ 
	is an isomorphism for any $m\leq -b$.
\end{proof}

\begin{example}
	If $V^+/t$ is Hodge-Tate with Hodge-Tate weights in $[a,b]$, then the condition in Lemma~\ref{lem:fixed vector modulo t^n} is satisfied.
\end{example}

\begin{proposition}\label{prop:cdR to filtered}
	Assume $\Acal \in \AlgAff_{\Qbb_p}$.
	Then for $V^+\in \FP_{B_{\dR,A}^+}^{\dR}(G_K)$, the following hold:
	\begin{enumerate}
		\item The filtered $K\otimes A$-module $D_{\dR}(V^+)$ lies in $\Fil_{K,A}$.
		\item The isomorphism
		$$D_{\dR}(V^+)\otimes_{K\otimes A}B_{\dR,A}\cong V$$
		induces an isomorphism
		$$\Fil^0(D_{\dR}(V^+)\otimes B_{\dR,A}) \xrightarrow{\sim} V^+.$$
		\item The exact sequence
	          $$0\to t^{n+1}V^+ \to t^nV^+\to t^n(V^+/t)\to 0$$
			  induces an exact sequence
			  $$0\to \Fil^{n+1}D_{\dR}(V^+)\to \Fil^{n}D_{\dR}(V^+) \to \gr^n D_{\HT}(V^+/t)\to 0.$$
			  In particular, there is a natural isomorphism $$\gr^{\bullet}D_{\dR}(V^+)\cong \gr^{\bullet}D_{\HT}(V^+/t).$$
	\end{enumerate}
\end{proposition}
\begin{proof}
	By Proposition~\ref{prop:dR implies HT}, $V^+/t$ is Hodge-Tate with Hodge-Tate weights in $[a,b]$ for some integers $a\leq b$. 
	We proceed by induction on $b-a$.
	By replacing $V^+$ with $V^+(-a)$, we may assume $a=0$.
	We write $\gr^{-n} D_{\HT}(V^+/t)=\bar{V}_n$.
	Then we have a decomposition
	$$V^+/t=\bigoplus_{i=0}^b\bar{V}_i \otimes_{K\otimes A}(\Cbb_p\otimes A)(i)=\bigoplus_{i=0}^b\bar{V}_i \otimes_{K}\Cbb_p(i).$$
	By Lemma~\ref{lem:fixed vector modulo t^n}, the natural morphism
	$$R\Gamma(G_K,V^+)\to R\Gamma(G_K,V^+/tV^+)$$
	is an isomorphism.
	From this, we get a morphism
	$$\bar{V}_0=(V^+/tV^+)^{G_K}\cong (V^+)^{G_K}\to V^+,$$
	and it induces 
	$$\bar{V}_0\otimes_{K\otimes A}B_{\dR,A}^+\to V^+.$$
	We write $V_0^+=\bar{V}_0\otimes_{K\otimes A}B_{\dR,A}^+$.
	Then $V_0^+$ is de Rham, and we have $D_{\dR}(V_0^+)=\bar{V}_0$ and $0=\Fil^1 D_{\dR}(V_0^+) \subset \Fil^0 D_{\dR}(V_0^+)=\bar{V}_0$. 
	We write $W^+=V^+/V^+_0$.
	By construction, $W^+$ is a finite projective $B_{\dR,A}^+$-module and there is an isomorphism
	$$W^+/t=\bigoplus_{i=1}^b\bar{V}_i \otimes_{K\otimes A}(\Cbb_p\otimes A)(i).$$
	By Lemma~\ref{lem:dR sub quot easy}, $W^+$ is also de Rham, and the Hodge-Tate weights of $W^+/t$ is in $[1,b]$.
	Let us prove that
	$$0\to D_{\dR}(V_0^+)\to D_{\dR}(V^+) \to D_{\dR}(W^+)\to 0$$
	is an exact sequence of filtered $K\otimes A$-modules, that is, for every $n\in \Zbb$, the sequence
	\begin{align}\label{eq:gr exact}
	0\to \Fil^n D_{\dR}(V_0^+)\to \Fil^nD_{\dR}(V^+) \to \Fil^nD_{\dR}(W^+)\to 0
	\end{align}
	is exact.
	By Lemma~\ref{lem:fixed vector modulo t^n}, all three terms in \eqref{eq:gr exact} are zero for $n\geq 1$, so \eqref{eq:gr exact} is exact when $n\geq 1$.
	Similarly, by Lemma~\ref{lem:fixed vector modulo t^n}, for $n\leq -b$, then \eqref{eq:gr exact} is isomorphic to the exact sequence 
	$$0\to D_{\dR}(V_0^+)\to D_{\dR}(V^+) \to D_{\dR}(W^+)\to 0$$
	of $K\otimes A$-modules.
	To prove the exactness of \eqref{eq:gr exact} for $-b< n\leq 0$, it suffices to show that
	$$\gr^n D_{\dR}(V^+)\to \gr^n D_{\dR}(W^+)$$
	is injective for $-b\leq n<0$.
	By taking the $G_K$-invariant vectors of the exact sequence 
	$$0\to t^n V^+\to t^{n+1}V^+ \to t^nV^+/t^{n+1}V^+\to 0,$$
	we get an injection 
	$$\gr^n D_{\dR}(V^+) \to (t^nV^+/t^{n+1}V^+)^{G_K}=\bar{V}_{-n}.$$
	Similarly, we get an injection
	$$\gr^n D_{\dR}(W^+) \to (t^nW^+/t^{n+1}W^+)^{G_K}=\bar{V}_{-n}.$$
	Therefore, $\gr^n D_{\dR}(V^+)\to \gr^n D_{\dR}(W^+)$ is injective.

	From the exact sequence $$0\to D_{\dR}(V_0^+)\to D_{\dR}(V^+) \to D_{\dR}(W^+)\to 0$$ 
	of filtered $K\otimes A$-modules and the induction hypothesis, we get (1) and (3).
	Moreover this exact sequence induces an exact sequence
	\begin{align*}
		0 \to &\Fil^0(D_{\dR}(V_0^+)\otimes B_{\dR,A})\to \Fil^0(D_{\dR}(V^+)\otimes B_{\dR,A})\\
		\to &\Fil^0(D_{\dR}(W^+)\otimes B_{\dR,A}) \to 0
	\end{align*}
	of $B_{\dR,A}^+$-modules, and there is a natural morphism from the above exact sequence to the exact sequence
	$$0\to V_0^+\to V^+\to W^+\to 0.$$
	By the induction hypothesis, the morphisms in the left terms and the right terms are isomorphisms.
	Therefore, the morphism in the middle term 
	$$\Fil^0(D_{\dR}(V^+)\otimes B_{\dR,A}) \to V^+$$
	is also an isomorphism.
\end{proof}

By the proof, we get the following corollary.

\begin{corollary}\label{cor:HT of HT wt 0 implies dR}
For $V^+\in \FP_{B_{\dR,A}^+}^{\dR}(G_K)$, if $V^+/t$ is Hodge-Tate of Hodge-Tate weight $0$, then $V$ is de Rham.
\end{corollary}
\begin{proof}
	We can define $V_0^+\in \FP^{\dR}_{B_{\dR,A}}(G_K)$ as in the proof of Proposition~\ref{prop:cdR to filtered}, then we have $V_0^+=V^+$, which proves the claim.
	We note that in this setting, the assumption $\Acal \in \AlgAff_{\Qbb_p}$ is not necessary, because this assumption is used only to deduce that $V^+/t$ is Hodge-Tate in the proof of Proposition~\ref{prop:cdR to filtered}. 
\end{proof}

\begin{theorem}\label{thm:cdR filtered equivalence}
	Assume $\Acal \in \AlgAff_{\Qbb_p}$.
	Then the functors
	\begin{align*}
		&\Fil_{K,A}\to \FP_{B_{\dR,A}^+}^{\dR}(G_K);\; M\mapsto \Fil^0 (M\otimes B_{\dR,A}), \\
		&\FP_{B_{\dR,A}^+}^{\dR}(G_K)\to \Fil_{K,A} ;\; V^+ \mapsto D_{\dR}(V^+)
	\end{align*}
	are symmetric monoidal and quasi-inverse to each other.
\end{theorem}
\begin{proof}
	Quasi-invertibility follows from Proposition~\ref{prop:filtered to cdR} and Proposition~\ref{prop:cdR to filtered}.
	For symmetric monoidality, it suffices to show it for the functor
	$$\FP_{B_{\dR,A}^+}^{\dR}(G_K)\to \Fil_{K,A} ;\; V^+ \mapsto D_{\dR}(V^+).$$
	For $V^+, W^+\in \FP_{B_{\dR,A}^+}^{\dR}(G_K)$, we have
	$$\Fil^n(D_{\dR}(V^+)\otimes_{K\otimes A}D_{\dR}(W^+))=\sum_{i+j=n} (\Fil^i D_{\dR}(V^+)\otimes_{K\otimes A} \Fil^j D_{\dR}(W^+)).$$
	Therefore, we get a morphism of filtered modules
	\begin{align}\label{dR tensor monoidal}
			D_{\dR}(V^+)\otimes_{K\otimes A}D_{\dR}(W^+) \to D_{\dR}(V^+\otimes_{B_{\dR,A}^+}W^+),
	\end{align}
	which is an isomorphism after forgetting filtrations.
	We want to show that \eqref{dR tensor monoidal} is an isomorphism of filtered modules.
	Since both filtrations are complete, it suffices to show that \eqref{dR tensor monoidal} induces an isomorphism of the associated graded modules.
	This follows from Proposition~\ref{prop:cdR to filtered} (3) and Lemma~\ref{lem:HT to graded tensor}.
\end{proof}

\begin{remark}
	We expect that it holds without the assumption $\Acal \in \AlgAff_{\Qbb_p}$, but we do not know how to prove it.
\end{remark}

\begin{lemma}\label{lem:dR basechange filtered}
For $V^+\in \FP_{B_{\dR,A}^+}^{\dR}(G_K)$ and a morphism $\Acal\to \Bcal=(B,B^+)_{\square}$, there is an isomorphism of filtered $K\otimes B$-modules
$$D_{\dR}(V^+\otimes_{B_{\dR,A}^+}B_{\dR,B}^+)\cong D_{\dR}(V^+)\otimes_{K\otimes A} (K\otimes B).$$
\end{lemma}
\begin{proof}
	By construction, there is a commutative diagram
	$$
	\xymatrix{
		\Fil_{K,A}\ar[r]^-{\cong}\ar[d] & \FP_{B_{\dR,A}^+}^{\dR}(G_K)\ar[d] \\
		\Fil_{K,B}\ar[r]^-{\cong} & \FP_{B_{\dR,B}^+}^{\dR}(G_K),
	}
	$$
	where the vertical morphisms are given by $-\otimes_{K\otimes A} (K\otimes B)$ and $-\otimes_{B_{\dR,A}^+}B_{\dR,B}^+$, respectively.
	The claim easily follows from this commutative diagram.
\end{proof}

\subsection{Strongly de Rham representations}

For the proof of the $p$-adic monodromy theorem, we introduce a variant of de Rham representations.

\begin{definition}
We define $B_{\sdR,A}^+\coloneqq B_{\dR}^+\otimes \Acal$ and $B_{\sdR,A}\coloneqq B_{\dR}\otimes \Acal= B_{\dR,A}^+[1/t]$\footnote{$\sdR$ stands for ``strongly de Rham''.}.
Since $B_{\dR}^+$ and $B_{\dR}$ are nuclear $\Qbb_{p,\square}$-modules by \cite[Proposition 3.29]{RJRC22}, these definitions are independent of the choice of $A^+$.
\end{definition}

\begin{remark}
	By Remark~\ref{rem:de Rham complete}, if $A$ is a Banach $\Qbb_p$-algebra, then $B_{\sdR,A}^+\cong B_{\dR,A}^+$ and $B_{\sdR,A}\cong B_{\dR,A}$.
	In this case, there is no difference between ``de Rham'' and ``strongly de Rham''.
\end{remark}

\begin{lemma}
	We have $B_{\sdR,A}^{G_K}=K\otimes A$.
\end{lemma}
\begin{proof}
	It can be proved in the same way as in Lemma~\ref{lem:HT fixed vector}.
\end{proof}

\begin{lemma}\label{lem:BsdR flat}
	The ring $B_{\sdR,A}$ is faithfully flat over $K\otimes A$.
\end{lemma}
\begin{proof}
	We may assume $A=\Qbb_p$.
	Then the claim easily follows from \cite[Lemma 3.21]{RJRC22}.
\end{proof}

\begin{definition}
	\begin{enumerate}
\item Let $V^+$ be a finite projective $B_{\sdR,A}^+$-module with a semilinear $G_K$-action.
We also define $D_{\sdR}^K(V^+)\coloneqq V^+[1/t]^{G_K}$.
When $K$ is clear from the context, we omit $K$ from the notation.
We say that $V$ is \textit{strongly de Rham} if the natural morphism
$$D_{\sdR}(V^+)\otimes_{K\otimes A} B_{\sdR,A}\to V^+[1/t]$$
is an isomorphism.
Let $\FP_{B_{\sdR,A}^+}(G_K)$ denote the category of finite projective $B_{\sdR,A}^+$-modules with a semilinear $G_K$-action, and $\FP_{B_{\sdR,A}^+}^{\sdR}(G_K)$ denote the subcategory of $\FP_{B_{\sdR,A}^+}(G_K)$ consisting of strongly de Rham representations.
\item Let $V$ be a $G_K$-equivariant vector bundle over $X_{\Cbb_p,\Acal}$.
We define $V^+_{\sdR},V_{\sdR}$ as
\begin{align*}
&V^+_{\sdR}\coloneqq V^{[1,1]}\otimes_{\tilde{B}^{[1,1]}_{\Cbb_p,A}}B_{\sdR,A}^+,\\
&V_{\sdR}\coloneqq V^{[1,1]}\otimes_{\tilde{B}^{[1,1]}_{\Cbb_p,A}}B_{\sdR,A}=V^+_{\dR}[1/t]
\end{align*}
which is a finite projective $B_{\sdR,A}^+$-module (resp. $B_{\sdR,A}$-module) with $G_K$-action.
We say that $V$ is \textit{strongly de Rham} if $V_{\sdR}$ is strongly de Rham.
In this case, we simply write $D_{\sdR}(V)$ for $D_{\sdR}(V_{\sdR}^+)$.
Let $\Vect^{\sdR}(X_{\Cbb_p,\Acal}/G_K)$ denote the category of strongly de Rham $G_K$-equivariant vector bundles over $X_{\Cbb_p,\Acal}$.
\end{enumerate}
\end{definition}

The following seven claims can be proved in the same way as in the Hodge-Tate case.

\begin{lemma}\label{lem:sdR fil exhaustive}
Let $V^+\in \FP_{B_{\sdR,A}^+}^{\sdR}(G_K)$.
Then $D_{\sdR}(V^+)$ is a finite projective $K\otimes A$-modules.
\end{lemma}

\begin{lemma}\label{lem:sdR dual direct sum tensor product}
The subcategory $\FP_{B_{\sdR,A}^+}^{\sdR}(G_K)\subset \FP_{B_{\sdR,A}^+}(G_K)$ is stable under taking duals, direct sums, and tensor products.
\end{lemma}

\begin{lemma}\label{lem:potentially sdR}
Let $L/K$ be a finite extension.
Then for $V^+\in \FP_{B_{\sdR,A}^+}(G_K)$, $V^+$ is strongly de Rham as a $G_K$-representation if and only if it is strongly de Rham as a $G_L$-representation.
\end{lemma}

\begin{lemma}\label{lem:sdR basechange}
For $V^+\in \FP_{B_{\sdR,A}^+}^{\sdR}(G_K)$ and a morphism $\Acal\to \Bcal=(B,B^+)_{\square}$, $V^+_B=V^+\otimes_{\Acal}\Bcal$ is strongly de Rham.
Moreover, in this case, there is an isomorphism of $K\otimes B$-modules
$$D_{\sdR}(V^+_B)\cong D_{\sdR}(V^+)\otimes_{(K\otimes A)} (K\otimes B).$$
\end{lemma}

\begin{proposition}\label{prop:sdR analytic local}
Let $\{\AnSpec\Acal_i \to \AnSpec\Acal\}_{i=1}^n$ 
be an affinoid covering of $\AnSpec \Acal$.
Then for $V^+\in \FP_{B_{\sdR,A}^+}(G_K)$, $V^+$ is strongly de Rham if and only if $V^+_i=V^+\otimes_{\Acal}\Acal_i$ is strongly de Rham for each $i$.
\end{proposition}

\begin{lemma}\label{lem:sdR inj}
	Assume $\Acal \in \AlgAff_{\Qbb_p}$.
	Then for $V^+\in \FP_{B_{\sdR,A}^+}(G_K)$, the natural morphism
	\begin{align*}
		&D_{\sdR}(V^+)\otimes_{K\otimes A} B_{\sdR,A}\to V^+[1/t]
	\end{align*}
	is injective.
\end{lemma}

\begin{corollary}\label{cor:sdR sub quot}
Assume $\Acal \in \AlgAff_{\Qbb_p}$.
Let $V^+\in \FP_{B_{\sdR,A}^+}^{\dR}(G_K)$, and let $W^+\subset V^+$ be a $G_K$-stable submodule such that $W^+$ and $V^+/W^+$ are also finite projective.
Then $W^+$ and $V^+/W^+$ are also strongly de Rham. 
\end{corollary}

We have also the following lemma. 

\begin{lemma}\label{lem:sdR implies dR}
	For $V^+\in \FP_{B_{\sdR,A}^+}^{\sdR}(G_K)$, $V^+\otimes_{B_{\sdR,A}^+} B_{\dR,A}^+$ is de Rham. 
\end{lemma}
\begin{proof}
	There is an isomorphism
	$$D_{\sdR}(V^+)\otimes_{K\otimes A} B_{\sdR,A}\cong V^+[1/t].$$
	From this isomorphism, we get an isomorphism
	$$D_{\sdR}(V^+)\otimes_{K\otimes A} B_{\dR,A} \cong V^+\otimes_{B_{\sdR,A}^+} B_{\dR,A}.$$
	Since $D_{\sdR}(V)$ is a finite projective $K\otimes A$-module, we get an isomorphism
	$$D_{\dR}(V^+)=(D_{\sdR}(V^+)\otimes_{K\otimes A} B_{\dR,A})^{G_K}\cong D_{\sdR}(V^+)\otimes_{K\otimes A} B_{\dR,A}^{G_K}\cong D_{\sdR}(V^+).$$
	Therefore, the natural morphism
	$$D_{\dR}(V^+\otimes_{B_{\dR,A}^+} B_{\dR,A}^+)\otimes_{K\otimes A} B_{\dR,A}\to V^+\otimes_{B_{\dR,A}^+} B_{\dR,A}$$
	is an isomorphism.
\end{proof}

Finally, we prove that de Rham $G_K$-equivariant vector bundles over $X_{\Cbb_p,\Acal}$ are strongly de Rham.
For the proof, we establish several lemmas.

\begin{lemma}\label{lem:dr extension}
	Assume $\Acal \in \AlgAff_{\Qbb_p}$.
	Then for $V^+\in \FP_{B_{\dR,A}^+}^{\dR}(G_K)$ and for any integers $l\leq m \leq n$, 
	$$0\to (t^mV^+/t^nV^+)^{G_K} \to (t^lV^+/t^nV^+)^{G_K} \to (t^lV^+/t^mV^+)^{G_K}\to 0$$
	is exact.
\end{lemma}
\begin{proof}
	First, let us prove 
	$$0\to (t^mV^+)^{G_K} \to (t^lV^+)^{G_K} \to (t^lV^+/t^mV^+)^{G_K}\to 0$$
	is exact for $l\leq m$.
	We proceed by induction on $m-l$.
	When $m-l=1$, then the claim follows from Proposition~\ref{prop:cdR to filtered}.
	When $m-l>1$, there is a diagram
	$$
	\xymatrix{
		& 0\ar[d] & 0\ar[d] & 0 \ar[d]& \\
	0\ar[r]   & (t^mV^+)^{G_K}\ar[r]\ar[d] &(t^{l+1}V^+)^{G_K}\ar[r]\ar[d] & (t^{l+1}V^+/t^mV^+)^{G_K}\ar[r]\ar[d] &0\\
	0 \ar[r]  & (t^mV^+)^{G_K}\ar[r]\ar[d] &(t^lV^+)^{G_K}\ar[r]\ar[d] & (t^lV^+/t^mV^+)^{G_K}\ar[r]\ar[d] &0\\
	0 \ar[r] & 0\ar[r]\ar[d] & (t^lV^+/t^{l+1}V^+)^{G_K}\ar[r]\ar[d] & (t^lV^+/t^{l+1}V^+)^{G_K}\ar[r]\ar[d] &0\\
	 & 0 & 0 & 0 & ,
	}
	$$
	where the left and middle vertical sequences and the upper and lower horizontal sequences are exact.
	Therefore, the middle horizontal sequence is also exact.

	Next, for any integers $l\leq m \leq n$, there is a diagram 
	$$
	\xymatrix{
		& 0\ar[d] & 0\ar[d] & 0 \ar[d]& \\
	0\ar[r]   & (t^nV^+)^{G_K}\ar[r]\ar[d] &(t^nV^+)^{G_K}\ar[r]\ar[d] & 0 \ar[r]\ar[d] &0\\
	0 \ar[r]  & (t^mV^+)^{G_K}\ar[r]\ar[d] &(t^lV^+)^{G_K}\ar[r]\ar[d] & (t^lV^+/t^mV^+)^{G_K}\ar[r]\ar[d] &0\\
	0 \ar[r] & (t^mV^+/t^nV^+)^{G_K} \ar[r]\ar[d] & (t^lV^+/t^nV^+)^{G_K} \ar[r]\ar[d] & (t^lV^+/t^mV^+)^{G_K} \ar[r]\ar[d] &0\\
	 & 0 & 0 & 0 & ,
	}
	$$
	where all three vertical sequences and the upper and middle horizontal sequences are exact.
	Therefore, the lower exact sequence is also exact.
\end{proof}

\begin{lemma}\label{lem:dr extension class}
Let $A$ be an affinoid $\Qbb_p$-algebra, and let $f\in A$.
Let $V^+$ be a finitely generated $t$-torsion free $B_{\dR,A}^+$-module with a $G_K$-action, and we write $V=V[1/t]$.
Assume that there exist integers $a\leq b$ and a finitely generated $K\otimes A$-module $N_i$ for $a\leq i\leq b$ such that there is a $G_K$-equivariant decomposition
$$V^+/tV^+\cong \bigoplus_{i=a}^b N_i\otimes_{K\otimes A}(\Cbb_p\otimes A)(i),$$
and for any integers $l\leq m \leq n$, 
$$0\to t^mV^+/t^nV^+[1/f]^{G_K} \to t^lV^+/t^nV^+[1/f]^{G_K} \to t^lV^+/t^mV^+[1/f]^{G_K}\to 0$$
is exact.
Then the natural morphism 
$$V[1/f]^{G_K} \otimes_{K\otimes A_f} B_{\sdR,A_f}\to V[1/f]$$
is an isomorphism.
\end{lemma}
\begin{proof}
	We proceed by induction on $b-a$.
	We may assume $a=0$.
	By the same argument as in the proof of Proposition~\ref{prop:cdR to filtered}, we get a finitely generated $t$-torsion free $B_{\dR,A}^+$-submodule $V_0^+\subset V^+$ with a $G_K$-action satisfying
	\begin{itemize}
		\item $V_0^+/t=N_0\otimes_{K\otimes A}(\Cbb_p\otimes A).$
		\item the natural morphism 
		$$V_0^{G_K} \otimes_{K\otimes A} B_{\dR,A}\to V_0,$$
		where $V_0=V_0^+[1/t]$, is an isomorphism
		\item $V^+/V_0^+$ is also $t$-torsion free and $(V^+/V_0^+)/t\cong \bigoplus_{i=1}^b N_i\otimes_{K\otimes A}(\Cbb_p\otimes A)(i).$
	\end{itemize}
	We write $W^+=V^+/V_0^+$ and $W=W^+[1/t]$.
	By the construction, for any $n$, the morphism
	$$(t^nV^+/t^{n+1}V^+)^{G_K} \to (t^nW^+/t^{n+1}W^+)^{G_K}$$ 
	is surjective.
	From a diagram
	$$
	\xymatrix@C=0.4em{
		0\ar[r] & t^{n+1}V^+/t^{n+2}V^+\left[\frac{1}{f}\right]^{G_K}\ar[r]\ar@{->>}[d] & t^nV^+/t^{n+2}V^+\left[\frac{1}{f}\right]^{G_K}\ar[r]\ar[d] & t^nV^+/t^{n+1}V^+\left[\frac{1}{f}\right]^{G_K}\ar[r]\ar@{->>}[d] &0\\
		0\ar[r] & t^{n+1}W^+/t^{n+2}W^+\left[\frac{1}{f}\right]^{G_K}\ar[r] & t^nW^+/t^{n+2}W^+\left[\frac{1}{f}\right]^{G_K}\ar[r] & t^nW^+/t^{n+1}W^+\left[\frac{1}{f}\right]^{G_K}, &
	}
	$$
	where the upper sequence is exact and the lower sequence is left exact, the morphism
	$$t^nV^+/t^{n+2}V^+[1/f]^{G_K} \to t^nW^+/t^{n+2}W^+[1/f]^{G_K}$$ 
	is also surjective and 
	\begin{align*}
		0\to &t^{n+1}W^+/t^{n+2}W^+[1/f]^{G_K} \to t^nW^+/t^{n+2}W^+[1/f]^{G_K} \\
		\to &t^nW^+/t^{n+1}W^+[1/f]^{G_K}\to 0
     \end{align*}
    is also exact.
	By repeating this argument, for any $m\leq n$, the morphism
	$$t^mV^+/t^nV^+[1/f]^{G_K} \to t^mW^+/t^nW^+[1/f]^{G_K}$$ 
	is surjective and for any integers $l\leq m \leq n$, 
	\begin{align*}
	0\to &t^mW^+/t^nW^+[1/f]^{G_K} \to t^lW^+/t^nW^+[1/f]^{G_K}\\ 
	\to &t^lW^+/t^mW^+[1/f]^{G_K}\to 0
	\end{align*}
	is exact.
	In particular, by Lemma~\ref{lem:fixed vector modulo t^n}, 
	$$0\to V_0[1/f]^{G_K}\to V[1/f]^{G_K}\to W[1/f]^{G_K}\to 0$$
	is exact, and therefore, 
	\begin{align*}
	0\to &V_0[1/f]^{G_K}\otimes_{K\otimes A_f} B_{\sdR,A_f}\to V[1/f]^{G_K}\otimes_{K\otimes A_f} B_{\sdR,A_f}\\
	\to &W[1/f]^{G_K}\otimes_{K\otimes A_f} B_{\sdR,A_f}\to 0
	\end{align*}
	is also exact, where we note that $B_{\dR,A}$ is flat over $K\otimes A$ by Lemma~\ref{lem:BsdR flat}.
	By the induction hypothesis, the natural morphism
	$$W[1/f]^{G_K} \otimes_{K\otimes A_f} B_{\sdR,A_f}\to W[1/f]$$
	is an isomorphism.
	Therefore, the natural morphism 
	$$V[1/f]^{G_K} \otimes_{K\otimes A_f} B_{\sdR,A_f}\to V[1/f]$$
	is also an isomorphism.
\end{proof}

\begin{theorem}\label{thm:cdR implies dR}
	A de Rham $G_K$-equivariant vector bundle $V$ over $X_{\Cbb_p,\Acal}$ is strongly de Rham. 
\end{theorem}
\begin{proof}
	If $A$ is a Banach $\Qbb_p$-algebra, then the claim is clear.
	Therefore, we may assume $\Acal \in \AlgAff_{\Qbb_p}$.
	By Proposition~\ref{prop:sdR analytic local} and \cite[Proposition 2.37]{Mikami24}, we can reduce to the case where $A=R_f$ for some affinoid $\Qbb_p$-algebra $R$ and $f\in R$.
	By replacing $R$ with the image of $R$ in $A=R_f$, we may assume that $f$ is a non-zero divisor of $R$.
	Let $I=[1,1]$.
	For $n$ sufficiently large, we get a finite projective $B_{K,n,A}^I$-module $V_{K,n}^I$ with the action of $\Gamma_K$ as in \cite{Mikami24}.
	We have $B_{K,n,A}^I=B_{K,n,R}^I[1/f]$.
	To simplify the notation, we write $B_n=B_{K,n,R}^I$ and $V_n=V_{K,n}^I$.
	We note that $t,f$ is a regular sequence in $B_n$.
	By \cite[Proposition 1.51]{Mikami24}, there is a finitely generated $B_n$-submodule $W \subset V_n$ stable under the action of $\Gamma_K$ such that $W[1/f]=V_n$.
	
	\noindent\textbf{Claim:} There exists an integer $k\geq 0$ such that for any $l\geq k$, $$t^lV_n\cap W=t^{l-k}(t^kV_n\cap W).$$
	\begin{proof}[Proof of the claim]
	Let us take a $B_n[1/f]$-module $U$ such that $U\oplus V_n\cong B_n[1/f]^{\oplus m}$ for some integer $m\geq 0$. Then for $i$ sufficiently large, we have $W \subset U\oplus V_n\cong B_n[1/f]^{\oplus m}$ is contained in $f^{-i}B_n^{\oplus m}$.
	By replacing the basis, we may assume $W\subset B_n^{\oplus m}\subset B_n[1/f]^{\oplus m}$.
	By construction, we have $t^lW=t^lB_n[1/f]^{\oplus m}\cap W$ and $t^lB_n^{\oplus m}=t^lB_n[1/f]^{\oplus m}\cap B_n^{\oplus m}$ for any $l\geq 0$. 
	Therefore, we have $t^lV_n\cap W=t^lB_n^{\oplus m} \cap W$ for any $l\geq 0$.
	Then the claim follows from the Artin-Rees Lemma.
	\end{proof}

	We may replace $W$ with $t^kV_n\cap W$ which is also stable under the action of $\Gamma_K$.
	Then for any $l\geq 0$, the natural morphism
	$$t^lW/t^{l+1}W\to t^lV_n/t^{l+1}V_n$$
	is injective.
	We write $K_m=B_{K,n}^I/t$.
	Then we have $D_{\Sen}^{K_m}(V_n)=V_n/tV_n$.
	Since $W/tW\subset V_n/tV_n=D_{\Sen}^{K_m}(V_n)$ is stable under the action of $\Gamma_K$, the Sen operator on $D_{\Sen}^{K_m}(V_n)$ defines an operator on $W/tW$.
	By Proposition~\ref{prop:dR implies HT} and Corollary~\ref{cor:Sen op HT}, we have $\prod_{i=a}^b(\Theta_{\Sen}-i)$ on $D_{\Sen}^{K_m}(V)$ is $0$ for some integers $a\leq b$.
	Therefore, $\prod_{i=a}^b(\Theta_{\Sen}-i)$ on $W/tW$ is also $0$.
	We write $\bar{W}_i=\Ker (\Theta_{\Sen}-i\colon W/tW\to W/tW)$.
	Then we have $W/tW=\bigoplus_{i=a}^b \bar{W}_i.$
	By the same argument as in the proof of Proposition~\ref{prop:Sen operator HT}, the action of $\Gamma_K$ on $\bar{W}_i(-i)$ is smooth for $a\leq i \leq b$.
	By increasing $n$ if necessary, we may assume that the action of $\Gamma_{K_m}$ on $\bar{W}_i(-i)$ is trivial for $a\leq i \leq b$.
	By the Galois descent, there is a finitely generated $K\otimes A$-module $N_i$ such that there is a $\Gamma_K$-equivariant isomorphism $\bar{W}_i(-i) \cong K_m\otimes_{K} N_i(-i)$.
	Therefore, we get a $\Gamma_K$-equivariant isomorphism
	$$W/tW\cong \bigoplus_{i=a}^b (K_m\otimes_{K} N_i)(-i).$$
	We write $W_{\dR}^+=W\otimes_{B_n} B_{\dR,R}^+$ which is a $t$-adically complete finitely generated $B_{\dR,R}^+$-module with the action of $G_K$, where we note $B_{\dR,R}^+$ is noetherian.
	Moreover, we have 
	\begin{align*}
	W_{\dR}^+/t\cong &W/tW\otimes_{K_m\otimes A} (\Cbb_p\otimes A)\\
	\cong &\bigoplus_{i=a}^b N_i \otimes_{K_m\otimes A} (\Cbb_p\otimes A)(-i).
    \end{align*} 
	We write $W_{\dR}=W_{\dR}^+[1/t].$
	Then we have $V_{\sdR}=W_{\dR}[1/f].$
	By Lemma~\ref{lem:dr extension} and Lemma~\ref{lem:dr extension class}, the morphism
	$$V_{\sdR}^{G_K}\otimes_{K\otimes A} B_{\sdR,A}\to V_{\sdR}$$
	is an isomorphism.
\end{proof}
\begin{remark}
	We expect that for $V^+\in \FP_{B_{\sdR,A}^+}(G_K)$, if $V^+\otimes_{B_{\sdR,A}^+} B_{\dR,A}^+$ is de Rham, then $V^+$ is strongly de Rham.
	We also believe that a similar argument could be carried out if the ``deperfection'' of $(V^+)^{H_K}$ were appropriately defined, but we shall not pursue this direction in the present paper.
\end{remark}

\section{Semistable representations and p-adic monodromy theorem}
We fix compatible embeddings $\overline{\Qbb_p}\hookrightarrow B_{\dR}^+$ and $\breve{\Qbb}_p\hookrightarrow \tilde{B}^{I}_{\Cbb_p}$ for closed intervals $I \subset (0,\infty)$.
\subsection{Definition and basic properties of semistable representations}

\begin{definition}
	Let $I \subset (0,\infty)$ be a closed interval.
	\begin{enumerate}
	\item Let us consider the group homomorphism
	$$\Lcal\colon 1+\mfrak_{\Cbb_p^{\flat}} \to \tilde{B}^{I}_{\Cbb_p};\; a \mapsto \log[a]=\sum_{n=0}^{\infty}(-1)^{n+1}\frac{([a]-1)^n}{n},$$
	where we note that $[a]-1$ is topologically nilpotent in $\tilde{B}^{I}_{\Cbb_p}$.
	Let $\mu$ be the group of roots of unity in $\Cbb_p^{\flat}$, then $\Ocal_{\Cbb_p^{\flat}}^{\times}=\mu \times(1+\mfrak_{\Cbb_p^{\flat}})$.
	We extend $\Lcal$ to $\Lcal\colon \Ocal_{\Cbb_p^{\flat}}^{\times} \to \tilde{B}^{I}_{\Cbb_p}$ so that $\Lcal|_{\mu}=0$.
	This induces a ring homomorphism 
	$$\Sym_{\Zbb}\Ocal_{\Cbb_p^{\flat}}^{\times}\to \tilde{B}^{I}_{\Cbb_p}.$$
	We define $\tilde{B}_{\log,\Cbb_p}^{I}=\tilde{B}^{I}_{\Cbb_p}\otimes_{\Sym_{\Zbb}\Ocal_{\Cbb_p^{\flat}}^{\times}}\Sym_{\Zbb}\Cbb_p^{\flat \times}.$
	For $x \in \Cbb_p^{\flat\times}$, let $\log [x] \in \tilde{B}_{\log,\Cbb_p}^I$ denote the image of $x\in \Sym_{\Zbb}\Cbb_p^{\flat \times}$.
	For $x\in \Cbb_p^{\flat\times}\setminus \Ocal_{\Cbb_p^{\flat}}^{\times}$, the natural morphism $\tilde{B}^{I}_{\Cbb_p}[X]\to \tilde{B}_{\log,\Cbb_p}^{I};\; X\mapsto \log[x]$ is an isomorphism.
	We also define $\tilde{B}_{\log,\Cbb_p,A}^{I}=\tilde{B}_{\log,\Cbb_p}^{I}\otimes \Acal$, which is independent of the choice of $A^+$ since $\tilde{B}_{\log,\Cbb_p}^{I}$ is a nuclear $\Qbb_{p,\square}$-module.
	\item The Frobenius on $\Cbb_p^{\flat\times}$ induces a Frobenius morphism 
	$$\varphi\colon \tilde{B}_{\log,\Cbb_p}^{I} \to \tilde{B}_{\log,\Cbb_p}^{p^{-1}I}.$$
	It is characterized by the condition that $\varphi|_{\tilde{B}_{\Cbb_p}^{I}}$ is the usual Frobenius and $\varphi(\log[x])=p\log[x]$ for every $x \in \Cbb_p^{\flat\times}$.
	By applying $-\otimes \Acal$, we also define
	$$\varphi\colon \tilde{B}_{\log,\Cbb_p,A}^{I} \to \tilde{B}_{\log,\Cbb_p,A}^{p^{-1}I}.$$
	\item We define $p^{\flat} \in \Cbb_p^{\flat\times}$ as $p^{\flat}=(p,p^{1/p},p^{1/p^2},\ldots)$.
	There is a $\tilde{B}_{\Cbb_p}^I$-derivation $N\colon \tilde{B}_{\log,\Cbb_p}^{I} \to \tilde{B}_{\log,\Cbb_p}^{I}$ defined via $N(\log[p^{\flat}])=-1$, which is called the \textit{monodromy operator}.
	The monodromy operator $N$ is $G_K$-equivariant, and it satisfies $p\varphi N=N\varphi\colon \tilde{B}_{\log,\Cbb_p}^{I} \to \tilde{B}_{\log,\Cbb_p}^{p^{-1}I}$.
	By applying $-\otimes \Acal$, we also define
	$$N\colon \tilde{B}_{\log,\Cbb_p,A}^{I} \to \tilde{B}_{\log,\Cbb_p,A}^{I}.$$
	\end{enumerate}
\end{definition}

\begin{lemma}
	Let $I\subset (0,\infty)$ be a closed interval such that $1\in I$.
	Then the natural morphism 
	$$K\otimes_{K_0}\tilde{B}_{\Cbb_p}^I\to B_{\dR}^+$$
	is injective.
	In particular, we have $(\tilde{B}_{\Cbb_p}^I)^{G_K}=K_0$.
\end{lemma}
\begin{proof}
	It follows from \cite[Proposition 10.2.7]{FF18}.
\end{proof}

Next, we prove that $(\tilde{B}_{\log,\Cbb_p}^I)^{G_K}=K_0$.
The following arguments are essentially the same as those in \cite[2.4]{Ber02}.

\begin{definition}
	We define $u\coloneqq \log ([p^{\flat}]/p)=\sum_{n=0}^{\infty}(-1)^{n+1}\frac{([p^{\flat}]/p-1)^n}{n}\in B_{\dR}^+$. 
	It is well-defined since $[p^{\flat}]/p-1\in \Fil^1 B_{\dR}$.
	We also define $\xi\coloneqq p-[p^{\flat}]\in A_{\inf}=W(\Ocal_{\Cbb_p^{\flat}})$, whose image in $B_{\dR}^+$ is a generator of the ideal $\Fil^1 B_{\dR}\subset B_{\dR}^+$.
\end{definition}

\begin{lemma}\label{lem:image S}
Let $I\subset (0,\infty)$ be a closed interval such that $1\in I$.
Then there is a (well-defined) morphism of $A_{\inf}$-algebras
$$\tilde{B}_{\Cbb_p}^{I}\to (A_{\inf}[[T]]/(pT-\xi))[1/p].$$
\end{lemma}
\begin{proof}
	Since there is a morphism $\tilde{B}_{\Cbb_p}^{I}\to \tilde{B}_{\Cbb_p}^{[1,1]}$, we may assume $I=[1,1]$.
	We note $\tilde{B}_{\Cbb_p}^{I}=A_{\inf}\left\langle\frac{[p^{\flat}]}{p},\frac{p}{[p^{\flat}]}\right\rangle[1/p]=(A_{\inf}\langle X,Y\rangle/(pX-[p^{\flat}],XY-1))[1/p]$.
	Let us consider the morphism 
	$$A_{\inf}[X,Y]\to A_{\inf}[[T]];\; X\mapsto 1-T, Y\mapsto (1-T)^{-1}.$$
	By taking the $p$-adic completions, we obtain a morphism 
	$$A_{\inf}\langle X,Y\rangle \to A_{\inf}[[T]].$$
	It induces a morphism 
	$$A_{\inf}\langle X,Y\rangle/(pX-[p^{\flat}],XY-1)\to A_{\inf}[[T]]/(pT-\xi).$$
	By adding $1/p$, we obtain the desired morphism.
\end{proof}

\begin{proposition}\label{prop:log fixed point}
Let $I\subset (0,\infty)$ be a closed interval such that $1\in I$.
Then the natural morphism
$$K\otimes_{K_0} \tilde{B}^I_{\log,\Cbb_p}\to B_{\dR}^+;\; \log[p^{\flat}]\mapsto u$$
is a $G_K$-equivariant injection.
In particular, there is an injection
$$K\otimes_{K_0} \tilde{B}^I_{\log,\Cbb_p}[1/t]\to B_{\dR},$$
and we have $(\tilde{B}^I_{\log,\Cbb_p}[1/t])^{G_K}=K_0$.
\end{proposition}
\begin{proof}
	The $G_K$-equivariance follows from the direct computation.
	Let us prove the injectivity.
	Since $K\otimes_{K_0} \tilde{B}^I_{\log,\Cbb_p}\cong (K\otimes_{K_0} \tilde{B}_{\Cbb_p}^I)[X]$ is an integral domain, it is enough to show that $\tilde{B}^I_{\log,\Cbb_p}\to B_{\dR}^+$ is injective.
	We note that this morphism is $G_{K_0}$-equivariant.
	Let $S\subset B_{\dR}^+$ be the image of $(A_{\inf}[[T]]/(pT-\xi))[1/p]\to B_{\dR}^+;\; T\mapsto \xi/p$.
	By Lemma~\ref{lem:image S}, the image of $\tilde{B}_{\Cbb_p}^I$ in $B_{\dR}^+$ is contained in $S$.
	By the proof \cite[4.3.2]{Fontaine94} (or \cite[Lemma 7.15]{Fontaine-Ouyang}), $u$ is not contained in $\Frac(S)\subset B_{\dR}.$
	In particular, $u$ is not contained in $\Frac(\tilde{B}_{\Cbb_p}^I)\subset B_{\dR}$.
	Let us prove that $u$ is transcendental over $\Frac(\tilde{B}_{\Cbb_p}^I)$.
	Assume that $u$ is algebraic over $\Frac(\tilde{B}_{\Cbb_p}^I)$ and let $f(X)=X^n+a_{n-1}X^{n-1}+a_0$ be the minimal polynomial of $u$.
	Let $c\colon G_K\to \Qbb_p$ be the cocycle such that for any $g\in G_{K_0}$, $gu=u+c(g)t$.
	Then for $g\in G_{K_0}$, we have 
	$$(u+c(g)t)^n+g(a_{n-1})(u+c(g)t)^{n-1}+\cdots+g(a_0)=0$$.
	Since $f$ is a minimal polynomial, we get 
	$$g(a_{n-1})+nc(g)t=a_{n-1}.$$
	Let $a=a_{n-1}+nu$, then $g(a)=a$ for $g\in G_{K_0}$.
	Therefore, we obtain $a\in (B_{\dR})^{G_{K_0}}=K_0\in \tilde{B}_{\Cbb_p}^I$.
	Thus, we get $u=(a-a_{n-1})/n \in \tilde{B}_{\Cbb_p}^I$, which is a contradiction.
\end{proof}

\begin{corollary}
	Let $I\subset (0,\infty)$ be a closed interval such that $p^n\in I$ for some $n\in \Zbb$.
	Then we have $(\tilde{B}^I_{\log,\Cbb_p,A}[1/t])^{G_K}=K_0\otimes A$.
\end{corollary}
\begin{proof}
	By the same argument as in Lemma~\ref{lem:HT fixed vector}, we can reduce to the case where $A=\Qbb_p$.
	Since there is an isomorphism
	$$\varphi^{n}\colon \tilde{B}^I_{\log,\Cbb_p}[1/t]\to \tilde{B}^{p^{-n}I}_{\log,\Cbb_p}[1/t],$$
	the claim follows from Proposition~\ref{prop:log fixed point}.
\end{proof}

By the usual argument, we get the following corollary.
\begin{corollary}\label{cor:log regular}
	Let $I\subset (0,\infty)$ be a closed interval such that $p^n\in I$ for some $n\in \Zbb$.
	Then $\tilde{B}^I_{\log,\Cbb_p}[1/t]$ is a $(\Qbb_p,G_K)$-regular ring in the sense of Fontaine.
\end{corollary}

\begin{definition}
	Let $V$ be a $G_K$-equivariant vector bundle over $X_{\Cbb_p,\Acal}$, and let $I\subset (0,\infty)$ be a closed interval.
	\begin{enumerate}
	\item 
	We define $V_{\log}^{I}=V^{I}\otimes_{\tilde{B}^{I}_{\Cbb_p,A}}\tilde{B}^{I}_{\log,\Cbb_p,A},$ which is a finite projective $\tilde{B}^{I}_{\log,\Cbb_p,A}$-module with a $G_K$-action.
	We also define $$D_{\st}^{K,I}(V)\coloneqq(V_{\log}^{I}[1/t])^{G_K},$$ which admits a canonical \textit{monodromy operator} $N\colon D_{\st}^{K,I}(V)\to D_{\st}^{K,I}(V)$.
	When $K$ is clear from the context, we simply write $D_{\st}^I(V)$.
	\item Assume that $p^n,p^{n+1}\in I$ for some $n\in \Zbb$.
	Then we say that $V$ is \textit{$I$-semistable}\footnote{By replacing $\tilde{B}_{\log,\Cbb_p,A}^{I}[1/t]$ with $\tilde{B}_{\Cbb_p,A}^{I}[1/t]$, we obtain the definition of crystalline $G_K$-equivariant vector bundles. The subsequent arguments work similarly if we replace ``semistable'' with ``crystalline'' and set the monodromy operator $N$ to $0$, but we omit the details here.} if the morphism
	$$D_{\st}^I(V)\otimes_{K_0\otimes A} \tilde{B}_{\log,\Cbb_p,A}^{I}[1/t]\to V_{\log}^{I}[1/t]$$
	is an isomorphism.
	Moreover, we say that $V$ is \textit{$I$-potentially semistable} if there is a finite extension $L/K$ such that $V$ is $I$-semistable as a $G_L$-equivariant vector bundle.
	\end{enumerate}
\end{definition}

\begin{remark}\label{rem:numerical criterion}
	By Corollary~\ref{cor:log regular}, a $G_K$-equivariant vector bundle $V$ over $X_{\Cbb_p}$ is $I$-semistable if and only if $\rk V=\dim_{K_0} D_{\st}^I(V)$.
\end{remark}
\begin{remark}
	Let $V$ be a $G_K$-equivariant \'{e}tale vector bundle $V$ over $X_{\Cbb_p}$, and $W$ be a corresponding $p$-adic representation of $G_K$.
	Then we can define $D_{\st}(W)$ as usual.
	Then $D_{\st}^I(V)$ is not necessarily isomorphic to $D_{\st}(W)$.
	If $W$ is semistable, then $D_{\st}^I(V)\cong D_{\st}(W)$, see the proof of Theorem~\ref{thm:comparison with the classical definition}.
\end{remark}

The following lemma can be proved in the same way as in the Hodge-Tate case.

\begin{lemma}\label{lem:Dst fin proj}
Let $I\subset (0,\infty)$ be a closed interval such that $p^n, p^{n+1}\in I$ for some $n\in \Zbb$, and let $V$ be a $I$-semistable $G_K$-equivariant vector bundle over $X_{\Cbb_p,\Acal}$.
Then $D_{\st}^I(V)$ is a finite projective $K_0\otimes A$-module.
\end{lemma}

\begin{proposition}
Let $V$ be a $G_K$-equivariant vector bundle over $X_{\Cbb_p,\Acal}$, and let $I,I'\subset (0,\infty)$ be closed intervals such that $p^n, p^{n+1}\in I$, $p^{n'},p^{n'+1}\in I'$ for some $n, n'\in \Zbb$.
Then $V$ is $I$-semistable (resp. $I$-potentially semistable) if and only if it is $I'$-semistable (resp. $I'$-potentially semistable).
\end{proposition}
\begin{proof}
	It suffices to show only the claim regarding semistability.
	We may assume $I\subset I'$.
	By the same argument as in the proof of Lemma~\ref{lem:sdR implies dR}, we find that $I'$-semistability implies $I$-semistability.
	We want to prove the converse. 
	We may assume $I=[p^n,p^{n+1}]$ and $I'=[p^{-l},p^m]$.
	By applying the Frobenius $\varphi$, we may assume $I=[p^{-1},1]$.
	Let $V$ be a $I$-semistable $G_K$-equivariant vector bundle over $X_{\Cbb_p,\Acal}$.
	Then for $r=p^{-1},1$, we have an isomorphism 
	$$V_{\log}^{[p^{-1},1]}[1/t]^{G_K}\otimes_{K_0\otimes A} \tilde{B}_{\log,\Cbb_p,A}^{[r,r]}[1/t]\cong V_{\log}^{[r,r]}[1/t].$$
	Therefore, we obtain 
	\begin{align*}
		V_{\log}^{[r,r]}[1/t]^{G_K}&\cong (V_{\log}^{[p^{-1},1]}[1/t]^{G_K}\otimes_{K_0\otimes A} \tilde{B}_{\log,\Cbb_p,A}^{[r,r]}[1/t])^{G_K}\\
		&\cong V_{\log}^{[p^{-1},1]}[1/t]^{G_K}\otimes_{K_0\otimes A} \tilde{B}_{\log,\Cbb_p,A}^{[r,r]}[1/t]^{G_K}\\
		&\cong V_{\log}^{[p^{-1},1]}[1/t]^{G_K}.
	\end{align*}
	In other words, the natural morphism 
	$$V_{\log}^{[p^{-1},1]}\to V_{\log}^{[r,r]}$$
	induces an isomorphism 
	$$V_{\log}^{[p^{-1},1]}[1/t]^{G_K}\cong V_{\log}^{[r,r]}[1/t]^{G_K}.$$
	The same holds for the interval of the form $[p^{-1-l},p^{-l}]$ by applying the Frobenius $\varphi^l$.
	By repeatedly applying a gluing argument, we find that for any $l\geq 1$, $m\geq 0$, the natural morphisms
	$$V_{\log}^{[p^{-l},p^m]}[1/t]^{G_K}\to V_{\log}^{[p^{-1},1]}[1/t]^{G_K}$$
	and
	$$V_{\log}^{[p^{-l},p^m]}[1/t]^{G_K} \otimes_{K_0\otimes A}\tilde{B}_{\log,\Cbb_p,A}^{[p^{-l},p^m]} \to V_{\log}^{[p^{-l},p^m]}[1/t]$$
	are isomorphisms.
\end{proof}

Since $I$-semistability (resp. $I$-potentially semistability) is independent of the choice of $I$, we henceforth omit $I$ from the terminology.
In practice, we usually take $I=[p^{-1},1]$.
By the proof, we get the following corollary.
\begin{corollary}
Let $V$ be a semistable $G_K$-equivariant vector bundle over $X_{\Cbb_p,\Acal}$, and
let $I, I'\subset (0,\infty)$ be a closed interval such that $p^n \in I$, $p^{n'}\in I'$ for some $n,n'\in \Zbb$.
Then for an integer $a\in \Zbb$ such that $p^a I\subset I'$, the natural morphism
$$D_{\st}^{I'}(V)\xrightarrow{\varphi^a} D_{\st}^{p^{-a}I'}(V)\to D_{\st}^I(V)$$
is a $\varphi^a$-semilinear isomorphism of $K_0\otimes A$-modules.
\end{corollary}

\begin{definition}\label{def:phi,N}
Let $V$ be a semistable $G_K$-equivariant vector bundle over $X_{\Cbb_p,\Acal}$.
Then we write $D_{\st}(V)=D_{\st}^{[p^{-1},1]}(V)$.
By the restriction, there are isomorphisms $D_{\st}(V)\simeq D_{\st}^{[p^{-1},p^{-1}]}(V)$ and $D_{\st}(V)\simeq D_{\st}^{[1,1]}(V)$.
Moreover, there is a $\varphi$-semilinear isomorphism $\varphi\colon D_{\st}^{[1,1]}(V)\xrightarrow{\sim} D_{\st}^{[p^{-1},p^{-1}]}(V)$.
Therefore, we obtain a $\varphi$-semilinear automorphism $\varphi$ on $D_{\st}(V)$, and we call it the \textit{Frobenius automorphism} on $D_{\st}(V)$
\end{definition}

By construction, we get the following lemma.
\begin{lemma}
	Let $V$ be a semistable $G_K$-equivariant vector bundle over $X_{\Cbb_p,\Acal}$.
	Then the Frobenius automorphism $\varphi$ and the monodromy operator $N$ on $D_{\st}(V)$ satisfy $p\varphi N=N\varphi.$
\end{lemma}

We will explore the correspondence between filtered $(\varphi,N)$-modules and semistable representations further in a later subsection.

The following five claims can be proved in the same way as in the Hodge-Tate case.

\begin{lemma}\label{lem:log-crys dual direct sum tensor product}
	The class of semistable (resp. potentially semistable) $G_K$-equivariant vector bundles is stable under taking duals, direct sums, and tensor products.
\end{lemma}

\begin{lemma}\label{lem:log-crys basechange}
Let $V$ be a semistable (resp. potentially semistable) $G_K$-equivariant vector bundle over $X_{\Cbb_p,\Acal}$.
Then for a morphism $\Acal\to \Bcal=(B,B^+)_{\square}$, $V_B=V\otimes_{\Acal}\Bcal$ is a semistable (resp. potentially semistable) $G_K$-equivariant vector bundle over $X_{\Cbb_p,\Bcal}$.
If $V$ is semistable as a $G_L$-equivariant vector bundle, then there is a natural isomorphism
$$D_{\st}^L(V)\otimes_{L_0\otimes A}(L_0\otimes B) \simeq D_{\st}^L(V_{B}),$$
which is compatible with the monodromy operators and Frobenius automorphisms.
\end{lemma}

\begin{proposition}\label{prop:log-crys analytic local}
Let $V$ be a $G_K$-equivariant vector bundle over $X_{\Cbb_p,\Acal}$, and 
$$\{\AnSpec\Acal_i \to \AnSpec\Acal\}_{i=1}^n$$ 
be an affinoid covering of $\AnSpec \Acal$.
Then $V$ is semistable (resp. potentially semistable) if and only if $V_i=V\otimes_{\Acal}\Acal_i$ is a semistable (resp. potentially semistable) $G_K$-equivariant vector bundle over $X_{\Cbb_p,\Acal_i}$ for each $i$.
\end{proposition}

\begin{lemma}\label{lem:log-crys inj}
	Assume $\Acal \in \AlgAff_{\Qbb_p}$.
	Let $I\subset (0,\infty)$ be a closed interval such that $p^n\in I$ for some $n\in \Zbb$, and let $V$ be a $G_K$-equivariant vector bundle over $X_{\Cbb_p,\Acal}$.
	The the natural morphism
	\begin{align*}
		&(V_{\log}^I[1/t])^{G_K}\otimes_{K_0\otimes A} \tilde{B}^I_{\log,\Cbb_p,A}\to V^I_{\log}[1/t]
	\end{align*}
	is injective.
\end{lemma}

\begin{corollary}\label{cor:log-crys sub quot}
Assume $\Acal \in \AlgAff_{\Qbb_p}$.
Let $V$ be a semistable (resp. potentially semistable) $G_K$-equivariant vector bundle over $X_{\Cbb_p,\Acal}$, and let $W\subset V$ be a $G_K$-stable subbundle.
Then $W$ and $V/W$ are also semistable (resp. potentially semistable). 
\end{corollary}

We can also prove the following standard lemma.

\begin{lemma}\label{lem:log-crys implies dR}
	A semistable $G_K$-equivariant vector bundle $V$ over $X_{\Cbb_p,\Acal}$ is de Rham. 
	In this case, there is a natural isomorphism
	$$D_{\st}(V)\otimes_{(K_0\otimes A)}(K\otimes A)\cong D_{\dR}(V).$$
\end{lemma}
\begin{proof}
	There is an isomorphism
	$$D_{\st}(V)\otimes_{K_0\otimes A} (K\otimes_{K_0}\tilde{B}_{\log,\Cbb_p,A}^{[p^{-1},1]}[1/t])\cong (K\otimes_{K_0}V_{\log}^{[p^{-1},1]}[1/t]).$$
	From this isomorphism, we obtain an isomorphism
	$$D_{\st}(V)\otimes_{K_0\otimes A} B_{\dR,A} \cong V_{\dR}.$$
	Since $D_{\st}(V)$ is a finite projective $K_0\otimes A$-module, we obtain an isomorphism
	$$D_{\dR}(V)=V_{\dR}^{G_K}\cong D_{\st}(V)\otimes_{K_0\otimes A} B_{\dR,A}^{G_K}\cong D_{\st}(V)\otimes_{K_0\otimes A} (K\otimes A).$$
	Therefore, the natural morphism
	$$D_{\dR}(V)\otimes_{K\otimes A} B_{\dR,A}\to V_{\dR}$$
	is an isomorphism.
\end{proof}

\begin{corollary}\label{cor:pot log-crys implies dR}
	A potentially semistable $G_K$-equivariant vector bundle over $X_{\Cbb_p,\Acal}$ is de Rham. 
\end{corollary}
\begin{proof}
	It follows from the above lemma and Lemma~\ref{lem:potentially dR}.
\end{proof}

Next, we give other definitions of semistable vector bundles.

\begin{construction}
Let $I=[r,s]\subset (0,\infty)$ be a closed interval.
To simplify the notation, we write $u=\log[p^{\flat}]\in \tilde{B}^I_{\log,\Cbb_p}.$
We define the $G_K$-stable filtration of $\tilde{B}_{\log,\Cbb_p}^{I}$ as 
$$\Fil^n\tilde{B}_{\log,\Cbb_p}^{I}=\bigoplus_{i=0}^n\tilde{B}_{\Cbb_p}^Iu^i.$$
By taking $H_K$-invariant vectors, we obtain a filtered ring
$(\tilde{B}_{\log,K_{\infty}}^{I}, \Fil^{\bullet}\tilde{B}_{\log,K_{\infty}}^{I})$.
By \cite[Theorem 3.12]{Mikami24}, each $\Fil^n\tilde{B}_{\log,K_{\infty}}^{I}$ is a finite projective $\tilde{B}_{K_{\infty}}^{I}$-module of rank $n+1$ with a semilinear $\Gamma_K$-action, and the natural morphism
$$\tilde{B}_{\log,K_{\infty}}^{I}\otimes_{\tilde{B}_{K_{\infty}}^{I}}\tilde{B}_{\Cbb_p}^{I}\to \tilde{B}_{\log,\Cbb_p}^{I}$$
is an isomorphism.
By taking $\Gamma_K$-locally analytic vectors of $(\tilde{B}_{\log,K_{\infty}}^{I}, \Fil^{\bullet}\tilde{B}_{\log,K_{\infty}}^{I})$, we obtain a filtered ring
$(B_{\log,K,\infty}^{I}, \Fil^{\bullet} B_{\log,K,\infty}^{I})$.
By \cite[Theorem 3.37]{Mikami24}, each $\Fil^n B_{\log,K,\infty}^{I}$ is a finite projective $B_{K,\infty}^{I}$-module of rank $n+1$ with a semilinear $\Gamma_K$-action, and the natural morphism
$$B_{\log,K,\infty}^{I}\otimes_{B_{K,\infty}^{I}}\tilde{B}_{K_{\infty}}^{I}\to \tilde{B}_{\log,K_{\infty}}^{I}$$
is an isomorphism.
By applying $-\otimes  \Acal$, we obtain filtered rings $(\tilde{B}_{\log,K_{\infty},A}^{I}, \Fil^{\bullet}\tilde{B}_{\log,K_{\infty},A}^{I})$ and $(B_{\log,K,\infty,A}^{I}, \Fil^{\bullet} B_{\log,K,\infty,A}^{I})$.
\end{construction}

We give more explicit descriptions of the above rings.
\begin{construction}
Let $I=[r,s]\subset (0,1)$ be a closed interval.
We write $=[\varepsilon]-1\in A_{\inf}$.
We set 
$$\varpi=[\varepsilon-1](1-\sum_{n=1}^{\infty}p^n[x_n])$$
in $W(\Cbb_p^{\flat})$.
Then $z=\sum_{n=1}^{\infty}p^n[x_n]$ is a topologically nilpotent element in $\tilde{B}_{\Cbb_p}^I$.
Therefore, we can define $\log(1-z)=-\sum_{n=1}^{\infty}z^n/n$ in $\tilde{B}_{\Cbb_p}^I$.
We define $\log \varpi=\log[\varepsilon-1] + \log(1-z)$ in $\tilde{B}_{\log,\Cbb_p}^I$.
Since $\tilde{B}_{\log,\Cbb_p}^I=\tilde{B}_{\Cbb_p}^I[\log[\varepsilon-1]]$, we also have $\tilde{B}_{\log,\Cbb_p}^I=\tilde{B}_{\Cbb_p}^I[\log\varpi]$.
Moreover, $[\varepsilon-1]$ and $z$ are fixed by the action of $H_K$, so $\log\varpi$ is also fixed by $H_K$.
Therefore, we get $\tilde{B}_{\log,K_{\infty}}^{I}=\tilde{B}_{K_{\infty}}^{I}[\log\varpi].$

For $\gamma \in \Gamma_K$, we take $a_{\gamma}\in \mu_{p-1}$ such that $\chi(\gamma)/a_{\gamma}\in 1+p\Zbb_p$.
Then we have 
\begin{equation}\label{eq:gamma action}
	\begin{split}
\gamma(\log \varpi)-\log \varpi&=\log \left[\frac{\gamma(\varepsilon-1)}{\varepsilon-1}\right]+\log \frac{\gamma(1-z)}{1-z}\\
&=\log \left[\frac{\gamma(\varepsilon-1)}{a_{\gamma}(\varepsilon-1)}\right]+\log \frac{\gamma(1-z)}{1-z}\\
&=\log \left(\left[\frac{\gamma(\varepsilon-1)}{a_{\gamma}(\varepsilon-1)}\right]\frac{\gamma(1-z)}{1-z}\right)\\
&=\log \frac{\gamma(\varpi)}{a_{\gamma}\varpi},
	\end{split}
\end{equation}
where we note that $\left[\frac{\gamma(\varepsilon-1)}{a_{\gamma}(\varepsilon-1)}\right]-1$ and $\frac{\gamma(1-z)}{1-z}-1$ are topologically nilpotent in $\tilde{B}_{\Cbb_p}^{I}$.
Since $\frac{\gamma(\varpi)}{a_{\gamma}\varpi}-1=\frac{(1+\varpi)^{\chi(\gamma)}-1}{a_{\gamma}\varpi}-1$ is a topologically nilpotent element in $B_{\Qbb_p,0}^I=\Qbb_p\left\langle\frac{\varpi}{p^{pr/(p-1)}},\frac{p^{ps/(p-1)}}{\varpi} \right\rangle$ (cf. \cite[Example 3.19]{Mikami24}), we get $\log \frac{\gamma(\varpi)}{a_{\gamma}\varpi}\in B_{\Qbb_p,0}^I\subset B_{K,\infty}^I$.
Therefore $B_{K,\infty}^I\oplus B_{K,\infty}^I \log\varpi\subset \tilde{B}_{K_{\infty}}^{I}[\log\varpi]$ is stable under the $\Gamma_K$-action. 
Since the $\Gamma_K$-representation $(B_{K,\infty}^I\oplus B_{K,\infty}^I \log\varpi)/B_{K,\infty}^I$ is isomorphic to $B_{K,\infty}^I$ by \eqref{eq:gamma action}, it is $\Gamma_K$-locally analytic.
Since locally analytic representations are stable under extensions, $B_{K,\infty}^I\oplus B_{K,\infty}^I \log\varpi$ is also $\Gamma_K$-locally analytic.
In particular, $\log\varpi$ is a $\Gamma_K$-locally analytic vector. 
Hence $B_{K,\infty}^I[\log\varpi]$ is also $\Gamma_K$-locally analytic.
Therefore, we obtain $B_{\log,K,\infty}^{I}=B_{K,\infty}^I[\log\varpi]$.

For $s$ sufficiently small so that $B_{K,0}^I$ is defined, we have $\log \frac{\gamma(\varpi)}{a_{\gamma}\varpi}\in B_{\Qbb_p,0}^I\subset B_{K,0}^I$.
Therefore, $B_{K,0}^I[\log\varpi]\subset B_{K,\infty}^I[\log\varpi]$ is stable under the $\Gamma_K$-action.
For $n\geq 0$, we define $B_{\log,K,n}^{I}=B_{K,n}^I[\log\varpi]$ and $B_{\log,K,n,A}^{I}=B_{\log,K,n}^{I}\otimes\Acal=B_{K,n,A}^I[\log\varpi]$.
\end{construction}

By using this period ring, we get other definitions of semistable vector bundles.
\begin{proposition}\label{prop:log-crys decompletion}
Let $V$ be a $G_K$-equivariant vector bundle over $X_{\Cbb_p,\Acal}$.
It descends to a $\Gamma_K$-equivariant vector bundle $V_{K_{\infty}}$ (resp. $V_{K,\infty}$) over $X_{K_{\infty},\Acal}$ (resp. $X_{K,\Acal}^{\la}$) by \cite{Mikami24}.
We write 
$$V_{\log,K_{\infty}}^{[p^{-1},1]}=V_{K_{\infty}}^{[p^{-1},1]}\otimes_{\tilde{B}^{[p^{-1},1]}_{K_{\infty},A}}\tilde{B}^{[p^{-1},1]}_{\log,K_{\infty},A},$$
and
$$V_{\log,K,\infty}^{[p^{-1},1]}=V_{K,\infty}^{[p^{-1},1]}\otimes_{B^{[p^{-1},1]}_{K,\infty,A}}B^{[p^{-1},1]}_{\log,K,\infty,A}.$$
\begin{enumerate}
	\item Both $(V_{\log,K_{\infty}}^{[p^{-1},1]}[1/t])^{\Gamma_K}$ and $(V_{\log,K,\infty}^{[p^{-1},1]}[1/t])^{\Gamma_K}$ are isomorphic to $D_{\st}(V)$.
	\item The natural morphisms
	$$(V_{\log,K_{\infty}}^{[p^{-1},1]}[1/t])^{\Gamma_K}\otimes_{K_0\otimes A} B_{\log,K_{\infty},A}^{[p^{-1},1]}[1/t]\to V_{\log,K_{\infty}}^{[p^{-1},1]}[1/t]$$
	and
	$$(V_{\log,K,\infty}^{[p^{-1},1]}[1/t])^{\Gamma_K}\otimes_{K_0\otimes A} B_{\log,K,\infty,A}^{[p^{-1},1]}[1/t]\to V_{\log,K,\infty}^{[p^{-1},1]}[1/t]$$
	are injective.
	\item 
	The following are equivalent:
	\begin{enumerate}
		\item $V$ is semistable.
		\item The natural morphism
		$$(V_{\log,K_{\infty}}^{[p^{-1},1]}[1/t])^{\Gamma_K}\otimes_{K_0\otimes A} B_{\log,K_{\infty},A}^{[p^{-1},1]}[1/t]\to V_{\log,K_{\infty}}^{[p^{-1},1]}[1/t]$$
		is an isomorphism.
		\item The natural morphism
		$$(V_{\log,K,\infty}^{[p^{-1},1]}[1/t])^{\Gamma_K}\otimes_{K_0\otimes A} B_{\log,K,\infty,A}^{[p^{-1},1]}[1/t]\to V_{\log,K,\infty}^{[p^{-1},1]}[1/t]$$
		is an isomorphism.
	\end{enumerate}
\end{enumerate}
\end{proposition}
\begin{proof}
	By construction, we have natural isomorphisms
	$$V_{\log,K_{\infty}}^{[p^{-1},1]}[1/t] \cong (V_{\log}^{[p^{-1},1]}[1/t])^{H_K}$$
	and
	$$V_{\log,K,\infty}^{[p^{-1},1]}[1/t] \cong ((V_{\log}^{[p^{-1},1]}[1/t])^{H_K})^{\Gamma_K\mathchar`-\la}.$$
	The claims immediately follow from this and Lemma~\ref{lem:log-crys inj}.
\end{proof}

We compare the above definition and Berger's definition in \cite{Ber02} in the case where $(A,A^+)$ is a strongly noetherian Tate affinoid pair.

\begin{definition}
Let $(A,A^+)$ be a strongly noetherian Tate affinoid pair over $(\Qbb_p,\Zbb_p)$. 
We define 
\begin{align*}
	&B^{\dagger}_{\rig,K,A}=\varinjlim_{0<s}\varprojlim_{0<r<s}B_{K,0,A}^{[r,s]},\quad \tilde{B}^{\dagger}_{\rig,K_{\infty},A}=\varinjlim_{0<s}\varprojlim_{0<r<s}\tilde{B}_{K_{\infty},A}^{[r,s]},\\
	&B^{\dagger}_{\log,K,A}=B^{\dagger}_{\rig,K,A}[\log \varpi],\quad \tilde{B}^{\dagger}_{\log,K,A}=\tilde{B}^{\dagger}_{\rig,K,A}[\log \varpi].
\end{align*}
If $A=\Qbb_p$, then we omit $A$ from the notation.
We note that the ring $\tilde{B}^{\dagger}_{\rig,K_{\infty}}$ is written as $\tilde{B}^{\dagger}_{\rig,K}$ in \cite{Ber02}.
\end{definition}

\begin{definition}\label{def:comparison with the Robba}
Let $(A,A^+)$ be a strongly noetherian Tate affinoid pair over $\Qbb_p$. 
We write $\Acal=(A,A^+)_{\square}$.
Then there is a natural categorical equivalence
\begin{align*}
	&\{\text{$G_K$-equivariant vector bundles over $X_{\Cbb_p,\Acal}$}\}\\
	\cong &\{\text{$(\varphi,\Gamma_K)$-modules over the Robba ring $B^{\dagger}_{\rig,K,A}$}\}
\end{align*}
by \cite[Theorem 0.6]{Mikami24} and \cite[Proposition 2.2.7]{KPX14}.
For a $G_K$-equivariant vector bundle $V$ over $X_{\Cbb_p,\Acal}$, Let $D_{\rig}(V)$ denote the $(\varphi,\Gamma_K)$-module over $B^{\dagger}_{\rig,K,A}$ corresponding to $V$ under the above equivalence, and we write 
$$D_{\log}(V)=D_{\rig}(V)\otimes_{B^{\dagger}_{\rig,K,A}}B^{\dagger}_{\log,K,A}=D_{\rig}(V)[\log \varpi].$$
A $(\varphi,\Gamma_K)$-module $D_{\rig}(V)$ over $B^{\dagger}_{\rig,K,A}$ is said to be \textit{semistable} if $D_{\log}(V)[1/t]^{\Gamma_K}$ is a finite projective $K_0\otimes A$-module and the natural morphism 
$$D_{\log}(V)[1/t]^{\Gamma_K}\otimes_{K_0\otimes A} B^{\dagger}_{\log,K,A}[1/t]\to D_{\log}(V)[1/t]$$
is an isomorphism.
\end{definition}

\begin{remark}
	We note that \cite[Proposition 2.2.7]{KPX14} only treats the case where $A$ is an affinoid $\Qbb_p$-algebra, but the same argument works for a general strongly noetherian Tate $\Qbb_p$-algebra $A$.
	We also use \cite[Corollary 2.1.5]{KPX14} in the proof of the next proposition, and it also holds for a general strongly noetherian Tate $\Qbb_p$-algebra $A$.
\end{remark}

\begin{theorem}\label{thm:comparison with the classical definition}
	Let $(A,A^+)$ be a strongly noetherian Tate affinoid pair over $\Qbb_p$. 
	We write $\Acal=(A,A^+)_{\square}$.
	Let $V$ be a $G_K$-equivariant vector bundle over $X_{\Cbb_p,\Acal}$.
	Then $V$ is semistable if and only if $D_{\rig}(V)$ is semistable.
\end{theorem}
\begin{proof}
	We write 
	$$\tilde{D}_{\rig}(V)=D_{\rig}(V)\otimes_{B^{\dagger}_{\rig,K,A}}\tilde{B}^{\dagger}_{\rig,K_{\infty},A}$$ and 
	$$\tilde{D}_{\log}(V)=D_{\rig}(V)\otimes_{B^{\dagger}_{\rig,K,A}}\tilde{B}^{\dagger}_{\log,K_{\infty},A}=\tilde{D}_{\rig}(V)[\log\varpi].$$
	Moreover, we write
	$$V^{(0,s]}_{K_{\infty}}=\varprojlim_{0<r\leq s} V^{[r,s]}_{K_{\infty}}, \quad V^{(0,s]}_{\log,K_{\infty}}=V^{(0,s]}_{K_{\infty}}[\log\varpi].$$
	Since the transition morphisms $V^{[r,s]}_{K_{\infty}} \to V^{[r',s]}_{K_{\infty}}$ are injective, we have an isomorphism\footnote{Note that the natural morphism $B^{(0,s]}_{\log,K_{\infty}}[1/t]\to \varprojlim_{0<r\leq s} (B^{[r,s]}_{\log,K_{\infty}}[1/t])$ is not an isomorphism. To avoid this problem, the argument in the proof becomes somewhat roundabout.}
	$$V^{(0,s]}_{\log,K_{\infty}}=\varprojlim_{0<r\leq s} V^{[r,s]}_{\log,K_{\infty}}.$$
	We note that it is not true in the derived sense.

	By \cite[Theorem 4.8]{Mikami24}, we have
	$$\left(\frac{1}{t^m}\bigoplus_{i=0}^n D_{\rig}(V)(\log\varpi)^i\right)^{\Gamma_K}=\left(\frac{1}{t^m}\bigoplus_{i=0}^n \tilde{D}_{\rig}(V)(\log\varpi)^i\right)^{\Gamma_K}$$
	for any integers $m,n\geq 0$.
	Therefore, we get $D_{\log}(V)[1/t]^{\Gamma_K}=\tilde{D}_{\log}(V)[1/t]^{\Gamma_K}.$

	First, we assume that $D_{\rig}(V)$ is semistable.
	By \cite[Corollary 2.1.5]{KPX14}, $K_0\otimes A \to B^{\dagger}_{\rig,K,A}$ is faithfully flat. 
	Therefore, $D_{\log}(V)[1/t]^{\Gamma_K}$ is a finite projective $K_0\otimes A$-module.
	From the isomorphism 
	$$D_{\log}(V)[1/t]^{\Gamma_K}\otimes_{K_0\otimes A} B^{\dagger}_{\log,K,A}[1/t]\xrightarrow{\sim} D_{\log}(V)[1/t],$$
	and $D_{\log}(V)[1/t]^{\Gamma_K}=\tilde{D}_{\log}(V)[1/t]^{\Gamma_K},$ we obtain the isomorphism
	$$\tilde{D}_{\log}(V)[1/t]^{\Gamma_K}\otimes_{K_0\otimes A} \tilde{B}^{\dagger}_{\log,K,A}[1/t]\xrightarrow{\sim} \tilde{D}_{\log}(V)[1/t].$$
	Since $\tilde{D}_{\log}(V)[1/t]=\varinjlim_{l}V_{\log,K_{\infty}}^{(0,p^{-l}]}[1/t]$, for $l$ sufficiently large, the natural morphisms
	$$V_{\log,K_{\infty}}^{(0,p^{-l}]}[1/t]^{\Gamma_K} \otimes_{K_0\otimes A} \tilde{B}_{\log,K_{\infty},A}^{(0,p^{-l}]}[1/t]\to V_{\log,K_{\infty}}^{(0,p^{-l}]}[1/t]$$
	and 
	$$V_{\log,K_{\infty}}^{(0,p^{-l}]}[1/t]^{\Gamma_K} \to \tilde{D}_{\log}(V)[1/t]^{\Gamma_K}\cong D_{\log}(V)[1/t]^{\Gamma_K}$$
	are isomorphisms.
	In particular, $(V_{\log,K_{\infty}}^{(0,p^{-l}]}[1/t])^{\Gamma_K}$ is a finite projective $K_0\otimes A$-module.
	By applying $\varphi^{-l}$, we may assume $l=0$.
	Since we have
	\begin{align*}
	V^{[p^{-1},1]}_{\log,K_{\infty}}[1/t]&\cong V^{(0,1]}_{\log,K_{\infty}}[1/t]\otimes_{\tilde{B}_{\log,K_{\infty},A}^{(0,1]}[1/t]}\tilde{B}_{\log,K_{\infty},A}^{[p^{-1},1]}[1/t]\\
	&\cong V_{\log,K_{\infty}}^{(0,1]}[1/t]^{\Gamma_K}\otimes_{K_0\otimes A} \tilde{B}_{\log,K_{\infty},A}^{[p^{-1},1]}[1/t],
    \end{align*}
	the natural morphisms
	$$V_{\log,K_{\infty}}^{(0,1]}[1/t]^{\Gamma_K}\to V_{\log,K_{\infty}}^{[p^{-1},1]}[1/t]^{\Gamma_K}$$
	and 
	$$V_{\log,K_{\infty}}^{[p^{-1},1]}[1/t]^{\Gamma_K}\otimes_{K_0\otimes A} \tilde{B}_{\log,K_{\infty},A}^{[p^{-1},1]}[1/t]\to V_{\log,K_{\infty}}^{[p^{-1},1]}[1/t]$$
	are isomorphisms.
	Therefore, $V$ is semistable by Proposition~\ref{prop:log-crys decompletion}.
	
	Next, we assume $V$ is semistable.
	Then for $r=p^{-1},1$, we have an isomorphism 
	$$V_{\log,K_{\infty}}^{[p^{-1},1]}[1/t]^{\Gamma_K}\otimes_{K_0\otimes A} \tilde{B}_{\log,K_{\infty},A}^{[r,r]}[1/t]\cong V_{\log,K_{\infty}}^{[r,r]}[1/t].$$
	Therefore, the natural morphism 
	$$V_{\log,K_{\infty}}^{[p^{-1},1]}\to V_{\log,K_{\infty}}^{[r,r]}$$
	induces an isomorphism 
	$$V_{\log,K_{\infty}}^{[p^{-1},1]}[1/t]^{\Gamma_K}\cong V_{\log,K_{\infty}}^{[r,r]}[1/t]^{\Gamma_K}.$$
	Since $V_{\log,K_{\infty}}^{[p^{-1},1]}[1/t]^{\Gamma_K}$ is a finite projective $K_0\otimes A$-module by Lemma~\ref{lem:Dst fin proj}, there exists an integer $M\geq 0$ such that for any $m\geq M$,
	$$V_{\log,K_{\infty}}^{[p^{-1},1]}[1/t]^{\Gamma_K}=\left(\frac{1}{t^m}V_{\log,K_{\infty}}^{[p^{-1},1]}\right)^{\Gamma_K}.$$
	Moreover, the cokernel $W$ of the natural morphism
	$$\left(\frac{1}{t^M}V_{\log,K_{\infty}}^{[p^{-1},1]}\right)^{\Gamma_K} \otimes_{K_0\otimes A}\tilde{B}_{\log,K_{\infty},A}^{[p^{-1},1]} \to \frac{1}{t^M}V_{\log,K_{\infty}}^{[p^{-1},1]}$$
	is $t$-power torsion since the above morphism becomes an isomorphism after inverting $t$.
	Since $W$ is a finitely generated $\tilde{B}_{\log,K_{\infty},A}^{[p^{-1},1]}$-module, there exists an integer $N\geq 0$ such that $t^NW=0$.
	For $r=p^{-1},1$, we have also 
	$$V_{\log,K_{\infty}}^{[r,r]}[1/t]^{\Gamma_K}=\left(\frac{1}{t^m} V_{\log,K_{\infty}}^{[r,r]}\right)^{\Gamma_K},$$
	and the cokernel of the natural morphism 
	$$\left(\frac{1}{t^M}V_{\log,K_{\infty}}^{[r,r]}\right)^{\Gamma_K} \otimes_{K_0\otimes A}\tilde{B}_{\log,K_{\infty},A}^{[r,r]} \to \frac{1}{t^M}V_{\log,K_{\infty}}^{[r,r]}$$
	is killed by $t^N$.
	The same holds for the interval of the form $[p^{-1-l},p^{-l}]$ by applying the Frobenius $\varphi^l$.
	By repeatedly applying a gluing argument, we obtain an isomorphism  
	$$\left(\frac{1}{t^m} V_{\log,K_{\infty}}^{[p^{-l},1]}\right)^{\Gamma_K}\cong \left(\frac{1}{t^m} V_{\log,K_{\infty}}^{[p^{-1},1]}\right)^{\Gamma_K}$$
	for any $l\geq 1$, $m\geq M$, and the cokernel of the natural morphism 
	$$\left(\frac{1}{t^M}V_{\log,K_{\infty}}^{[p^{-l},1]}\right)^{\Gamma_K} \otimes_{K_0\otimes A}\tilde{B}_{\log,K_{\infty},A}^{[p^{-l},1]} \to \frac{1}{t^M}V_{\log,K_{\infty}}^{[p^{-l},1]}$$
	is killed by $t^N$.
	By taking the projective limit, we get 
	$$\left(\frac{1}{t^m} V_{\log,K_{\infty}}^{(0,1]}\right)^{\Gamma_K}\cong \left(\frac{1}{t^m} V_{\log,K_{\infty}}^{[p^{-1},1]}\right)^{\Gamma_K}$$
	for any $m\geq M$, and therefore, we obtain an isomorphism 
	$$V_{\log,K_{\infty}}^{(0,1]}[1/t]^{\Gamma_K} \cong V_{\log,K_{\infty}}^{[p^{-1},1]}[1/t]^{\Gamma_K}.$$
	Moreover, the cokernel of the natural morphism 
	$$\left(\frac{1}{t^M}V_{\log,K_{\infty}}^{(0,1]}\right)^{\Gamma_K} \otimes_{K_0\otimes A}\tilde{B}_{\log,K_{\infty},A}^{(0,1]} \to \frac{1}{t^M}V_{\log,K_{\infty}}^{(0,1]}$$
	is killed by $t^N$.
	Therefore we obtain an isomorphism
	$$V_{\log,K_{\infty}}^{(0,1]}[1/t]^{\Gamma_K} \otimes_{K_0\otimes A}\tilde{B}_{\log,K_{\infty},A}^{(0,1]} \xrightarrow{\sim} V_{\log,K_{\infty}}^{(0,1]}[1/t].$$
	The same holds true if we replace the intervals $(0,1]$ with $(0,p^{-l}]$.
	Therefore, we get isomorphisms
	$$\tilde{D}^{\dagger}_{\log}(V)[1/t]^{\Gamma_K}\cong V_{\log,K_{\infty}}^{[p^{-1},1]}[1/t]^{\Gamma_K}$$
	and 
	$$\tilde{D}^{\dagger}_{\log}(V)[1/t]^{\Gamma_K}\otimes_{K_0 \otimes A}\tilde{B}^{\dagger}_{\log,K,A}[1/t]\xrightarrow{\sim} \tilde{D}_{\log}(V)[1/t].$$
	Since $D_{\log}(V)[1/t]^{\Gamma_K}=\tilde{D}_{\log}(V)[1/t]^{\Gamma_K}$, the natural morphism
	\begin{align}\label{eq:comparison}
	D_{\log}(V)[1/t]^{\Gamma_K}\otimes_{K_0 \otimes A} B^{\dagger}_{\log,K,A}[1/t]\to D_{\log}(V)[1/t]
	\end{align}
	becomes an isomorphism after applying $-\otimes_{B^{\dagger}_{\log,K,A}[1/t]}\tilde{B}^{\dagger}_{\log,K,A}[1/t]$.
	Since the morphism $B^{\dagger}_{\log,K,A}[1/t]\to \tilde{B}^{\dagger}_{\log,K,A}[1/t]$ has a section as a morphism of $B^{\dagger}_{\log,K,A}[1/t]$-modules, we find that the morphism \eqref{eq:comparison} is already an isomorphism.
	In other words, $D_{\rig}(V)$ is semistable.
\end{proof}

Next, we construct a pointwise criterion of semistable representations as Corollary~\ref{cor:pointwise criterion of HT}.

\begin{lemma}\label{lem:ideal pointwise}
	Let $R$ be a noetherian ring, and $I$ be an ideal of $R$.
	Then we have
	$$I=\bigcap_{\mfrak \in \Spm R, r\geq 1} (I+\mfrak^r).$$
\end{lemma}
\begin{proof}
	By replacing $R$ with $R/I$, we may assume $I=0$.
	We write $J=\bigcap_{\mfrak, r\geq 1} \mfrak^r.$
	By Krull's intersection theorem, for every maximal ideal $\mfrak$, we have $J\cdot R_{\mfrak}=0$.
	Therefore, we get $J=0$.
\end{proof}

\begin{proposition}\label{prop:ideal ext intersection}
	Assume that $A$ is an affinoid $\Qbb_p$-algebra.
	Let $\{I_{\lambda}\}_{\lambda \in \Lambda}$ be a family of ideals in $K_0\otimes A$, and we set $I=\bigcap_{\lambda}I_{\lambda}$.
	Then we have
	$$I\cdot B_{\log,K,n,A}^{[p^{-1},1]}[1/t]=\bigcap_{\lambda}(I_{\lambda}\cdot B_{\log,K,n,A}^{[p^{-1},1]}[1/t]).$$
\end{proposition}
\begin{proof}
	We have an isomorphism $B_{\log,K,n}^{[p^{-1},1]}[1/t]=\bigoplus_{\Nbb}B_{K,n}^{[p^{-1},1]}[1/t].$
	As $K_0$-modules, there is a decomposition 
	$$B_{K,n}^{[p^{-1},1]}[1/t]\cong B_{K,n}^{[p^{-1},1]}\oplus \bigoplus_{\Nbb} B_{K,n}^{[p^{-1},1]}/t.$$
	Therefore, $B_{\log,K,n}^{[p^{-1},1]}[1/t]$ can be written as a direct sum 
	$$B_{\log,K,n}^{[p^{-1},1]}[1/t]\cong\bigoplus_{i} M_i$$ 
	of Banach $K_0$-submodules.
	We also take an isomorphism 
	$$M_i \cong \hat{\bigoplus_{\Sigma_i}}K_0=(\bigoplus_{\Sigma_i} \Ocal_{K_0})^{\wedge}[1/p],$$ 
	where $(-)^{\wedge}$ denotes the $p$-adic completion.
	Then we obtain an isomorphism 
	\begin{align}\label{eq:decomposition}
	B_{\log,K,n,A}^{[p^{-1},1]}[1/t]\cong\bigoplus_{i} \left(\hat{\bigoplus_{\Sigma_i}}K_0\otimes A\right).	
	\end{align}
	For an ideal $J\subset K_0\otimes A$, we set $\hat{\bigoplus_{\Sigma_i}}J=J\otimes_{K_0\otimes A} \left(\hat{\bigoplus_{\Sigma_i}}K_0\otimes A\right)$.
	By \cite[Lemma 3.13]{RJRC22}, $\hat{\bigoplus_{\Sigma_i}}J$ is isomorphic to the usual completed tensor product $J\hotimes_{K_0\otimes A} \left(\hat{\bigoplus_{\Sigma_i}}K_0\otimes A\right)$.
	In particular, we have 
	\begin{align}\label{eq:intersection}
		\hat{\bigoplus_{\Sigma_i}}J= \hat{\bigoplus_{\Sigma_i}}(K_0\otimes A)\cap \prod_{\Sigma_i}J.
	\end{align}
	Under the isomorphism \eqref{eq:decomposition}, $J\cdot B_{\log,K,n,A}^{[p^{-1},1]}[1/t] \subset B_{\log,K,n,A}^{[p^{-1},1]}[1/t]$ corresponds to 
	$$\bigoplus_{i} \left(\hat{\bigoplus_{\Sigma_i}}J\right)\subset \bigoplus_{i} \left(\hat{\bigoplus_{\Sigma_i}}K_0\otimes A\right).$$
	Therefore, it is enough to show 
	$$\bigoplus_{i} \left(\hat{\bigoplus_{\Sigma_i}}I\right)=\bigcap_{\lambda}\bigoplus_{i} \left(\hat{\bigoplus_{\Sigma_i}}I_{\lambda}\right).$$
	This reduces to showing
	$$\hat{\bigoplus_{\Sigma_i}}I=\bigcap_{\lambda}\left(\hat{\bigoplus_{\Sigma_i}}I_{\lambda}\right).$$
	This follows immediately from \eqref{eq:intersection}.
\end{proof}

\begin{theorem}\label{thm:Gamma-stable ideal}
	Assume that $A$ is an affinoid $\Qbb_p$-algebra.
	Let $I\subset B_{\log,K,\infty,A}^{[p^{-1},1]}[1/t]$ be a $\Gamma_K$-stable ideal.
	Then there exists an ideal $J\subset K_0\otimes_{\Qbb_p}A$ such that $I=J\cdot B_{\log,K,\infty,A}^{[p^{-1},1]}[1/t]$.
\end{theorem}
\begin{proof}
	First, we consider the case where $A=\Qbb_p$.
	It is enough to show that a non-zero $G_K$-stable ideal $I$ of $\tilde{B}_{\log,\Cbb_p}^{[p^{-1},1]}[1/t]$ is the unit ideal.
	We take the smallest integer $n$ such that $I_n=I\cap \Fil^n\tilde{B}_{\log,\Cbb_p}^{[p^{-1},1]}[1/t]\neq 0$.
	Since $\tilde{B}_{\Cbb_p}^{[p^{-1},1]}[1/t]$ is a principal ideal domain, $I\cap \Fil^n\tilde{B}_{\log,\Cbb_p}^{[p^{-1},1]}[1/t]$ is a free $\tilde{B}_{\Cbb_p}^{[p^{-1},1]}[1/t]$-module of rank $1$.
	Moreover, by \cite[Proposition 10.1.1]{FF18}, the image of $I_n$ in $\tilde{B}_{\Cbb_p}^{[p^{-1},1]}[1/t]\cong \Fil^{n}/\Fil^{n-1}$ is a $G_K$-stable $\tilde{B}_{\Cbb_p}^{[p^{-1},1]}[1/t]$-submodule of $\tilde{B}_{\Cbb_p}^{[p^{-1},1]}[1/t]$, and hence is equal to $\tilde{B}_{\Cbb_p}^{[p^{-1},1]}[1/t]$.
	Therefore, we can take a generator of $I_n$ of the form $a=a_0+\cdots+a_{n-1}X^{n-1}+X^n.$
	Since $I_n$ is stable under the $G_K$-action, for every $g\in G_K$, there exists $b_g \in \tilde{B}_{\Cbb_p}^{[p^{-1},1]}$ such that $ga=b_ga$.
	By comparing the coefficient of $X^n$, we get $b_g=1$.
	Therefore, we get $0\neq a\in (\tilde{B}_{\log,\Cbb_p}^{[p^{-1},1]}[1/t])^{G_K}=K_0$, which implies the claim.

	Next, we assume that $A$ is finite over $\Qbb_p$. 
	It is enough to show that any $\Gamma_K$-stable submodule $N \subset \tilde{B}_{\log,\Cbb_p}^{[p^{-1},1]}[1/t]^{\oplus r}$ is induced from a submodule of $K_0^{\oplus r}$.
	Let $d$ denote the generic rank of $N$.
	By changing the basis of $\tilde{B}_{\log,\Cbb_p}^{[p^{-1},1]}[1/t]^{\oplus r}$, we may assume that the morphism 
	$$f\colon N \subset \tilde{B}_{\log,\Cbb_p}^{[p^{-1},1]}[1/t]^{\oplus r} \to \tilde{B}_{\log,\Cbb_p}^{[p^{-1},1]}[1/t]^{\oplus d}$$
	is injective, where the second morphism is the projection from the first component to the $d$th component.
	Then $\Coker(f)$ is a torsion $\tilde{B}_{\log,\Cbb_p}^{[p^{-1},1]}[1/t]$-module whose annihilator is a $G_K$-stable ideal.
	From the case where $A=\Qbb_p$, this annihilator is equal to $\tilde{B}_{\log,\Cbb_p}^{[p^{-1},1]}[1/t]$.
	In other words, $f$ is an isomorphism.
	Therefore, we can take a basis $x_1,\ldots,x_d$ of $N$ of the form 
	\begin{align*}
		x_1&=e_1+a_{d+1,1}e_{d+1}+\cdots+a_{r,1}e_r\\
		x_2&=e_2+a_{d+1,2}e_{d+1}+\cdots+a_{r,2}e_r\\
		\vdots & \\
		x_d&=e_d+a_{d+1,d}e_{d+1}+\cdots+a_{r,d}e_r,
	\end{align*}
	where $e_1,\ldots,e_r$ denote the standard basis of $\tilde{B}_{\log,\Cbb_p}^{[p^{-1},1]}[1/t]^{\oplus r}$.
	For $g\in G_K$, we can uniquely write $gx_1=b_1x_1+\cdots+b_dx_d$ for $b_1,\ldots,b_d\in \tilde{B}_{\log,\Cbb_p}^{[p^{-1},1]}[1/t]$.
	By considering the coefficients of $e_1,\ldots,e_d$ on both sides, we find $b_1=1$ and $b_2=\cdots=b_d=0$.
	In other words, $x_1$ is fixed by $G_K$, and therefore, $a_{d+1,1},\ldots,a_{r,1}\in K_0$.
	The same holds true for $x_2,\ldots,x_d$, so we get the claim.
	
	Finally, we consider the general case.
	Since $I$ is generated by the $\Gamma_K$-stable ideals $I\cap B_{\log,K,n,A}^{[p^{-1},1]}$ of $B_{\log,K,n,A}^{[p^{-1},1]}$ for all $n\geq 0$, it is enough to show that any $\Gamma_K$-stable ideal of $B_{\log,K,n,A}^{[p^{-1},1]}$ is induced from an ideal of $K_0\otimes A$.
	To simplify the notation, we write $B_{\log,K,n,A}^{[p^{-1},1]}[1/t]=B_{A,n}$.
	Let $I$ be a $\Gamma_K$-stable ideal of $B_{A,n}$.
	For a maximal ideal $\mfrak \subset A$ and $r\geq 1$, there is a (unique) ideal $\mfrak^r \subset J_{\mfrak,r}\subset K_0\otimes A$ such that $\bar{I}=J_{\mfrak,r}\cdot B_{A/\mfrak^r,n}$, where $\bar{I}$ denotes the image of $I$ in $B_{A/\mfrak^r,n}$.
	We set $J=\bigcap_{\mfrak,r} J_{\mfrak,r}$, and we prove $J\cdot B_{A,n}=I$.
	By Proposition~\ref{prop:ideal ext intersection}, it is enough to show 
	$$\bigcap_{\mfrak,r}(J_{\mfrak,r}\cdot B_{A,n})=I.$$
	By construction, we have 
	$$\bigcap_{\mfrak,r}(J_{\mfrak,r}\cdot B_{A,n})=\bigcap_{\mfrak,r}(I+\mfrak^r \cdot B_{A,n}).$$
	Let $\nfrak$ be a maximal ideal of $B_{A,n}$.
	Since $B_{A,n}$ is an algebraic-affinoid $\Qbb_{p,\square}$-algebra, $B_{A,n}/\nfrak$ is a finite extension of $\Qbb_p$.
	In particular, $\mfrak=\nfrak\cap A$ is a maximal ideal of $A$.
	Therefore, we have $I+\mfrak^r\cdot B_{A,n} \subset I+\nfrak^r.$
	By Lemma~\ref{lem:ideal pointwise}, we obtain
	$$I\subset \bigcap_{\mfrak,r}(I+\mfrak^r\cdot B_{A,n})\subset \bigcap_{\nfrak\in \Spm B_{A,n},r\geq 1}(I+\nfrak^r)=I,$$
	which proves the claim.
\end{proof}

\begin{corollary}\label{cor:Gamma-stable ideal}
	Let $A$ be an algebraic-affinoid $\Qbb_{p,\square}$-algebra of the form $B_f$ where $B$ is an affinoid $\Qbb_p$-algebra and $f\in B$.
	Let $I\subset B_{\log,K,\infty,A}^{[p^{-1},1]}[1/t]$ be a $\Gamma_K$-stable ideal.
	Then there exists an ideal $J\subset K_0\otimes A$ such that $I=J\cdot B_{\log,K,\infty,A}^{[p^{-1},1]}[1/t]$.
\end{corollary}
\begin{proof}
	Let $I_0$ be a preimage of $I$ under the morphism 
	$$B_{\log,K,\infty,B}^{[p^{-1},1]}[1/t]\to B_{\log,K,\infty,B}^{[p^{-1},1]}[1/t][1/f]= B_{\log,K,\infty,A}^{[p^{-1},1]}[1/t],$$
	which is a $\Gamma_K$-stable ideal.
	By Theorem~\ref{thm:Gamma-stable ideal}, there exists an ideal $J_0\subset K_0\otimes B$ such that $I_0=J_0\cdot B_{\log,K,\infty,B}^{[p^{-1},1]}[1/t]$.
	Then the ideal $J=J_0\cdot (K_0\otimes A)$ of $K_0\otimes A$ satisfies $I=J\cdot B_{\log,K,\infty,A}^{[p^{-1},1]}[1/t]$.
\end{proof}

\begin{theorem}\label{thm:log-crys criterion}
	Assume $\Acal \in \AlgAff_{\Qbb_p}$.
	Let $V$ be a $G_K$-equivariant vector bundle over $X_{\Cbb_p,\Acal}$.
	Then $V$ is semistable if and only if for any $\mfrak\in \Spm A$, $V/\mfrak$ is a semistable vector bundle over $X_{\Cbb_p,\Acal/\mfrak}$ and the natural morphism
	$$D_{\st}(V)\otimes_{A}A/\mfrak \to D_{\st}(V/\mfrak)$$
	is surjective.
\end{theorem}
\begin{proof}
	The only if part is clear. Let us prove the if part.
	The above condition is stable under scalar extensions.
	Therefore, by Proposition~\ref{prop:log-crys analytic local} and \cite[Proposition 2.37]{Mikami24}, we may assume that $A$ is of the form $B_f$ where $B$ is an affinoid $\Qbb_p$-algebra and $f\in B$.
	We want to show that 
	$$D_{\st}(V)\otimes_{K_0\otimes A} B_{\log,K,\infty,A}^{[p^{-1},1]}[1/t]\to V_{\log,K,\infty}^{[p^{-1},1]}[1/t]$$
	is an isomorphism.
	To simplify the notation, we denote the above morphism by $\iota\colon M_A\to N_A$.
	By Proposition~\ref{prop:log-crys decompletion} (2), it is injective.
	Note that $M_A$ is a relatively discrete $B_{\log,K,\infty,A}^{[p^{-1},1]}[1/t]$-module by \cite[Lemma 1.24]{Mikami24}.
	It suffices to show 
	$$\Fit_0(\Coker \iota)=B_{\log,K,\infty,A}^{[p^{-1},1]}[1/t],$$
	where $\Fit_0$ denotes the $0$th Fitting ideal.
	It is a $\Gamma_K$-stable ideal of $B_{\log,K,\infty,A}^{[p^{-1},1]}[1/t]$, so we can take an ideal $J\subset K_0\otimes A$ such that $J\cdot B_{\log,K,\infty,A}^{[p^{-1},1]}[1/t]=\Fit_0(\Coker \iota)$ by Corollary~\ref{cor:Gamma-stable ideal}.
	By the assumption, $\iota \colon M_{A/\mfrak}\to N_{A/\mfrak}$ is surjective where $M_{A/\mfrak}=M\otimes_A A/\mfrak$ and $N_{A/\mfrak}=N\otimes_A A/\mfrak$.
	Therefore, the image of $\Fit_0(\Coker \iota)$ in $B_{\log,K,\infty,A/\mfrak}^{[p^{-1},1]}[1/t]$ is the unit ideal.
	In other words, the image of $J$ in $K_0\otimes A/\mfrak$ is also the unit ideal for any $\mfrak \in \Spm A$. 
	Therefore, $J$ is also the unit ideal.
\end{proof}

\subsection{The $p$-adic monodromy theorem}
In this subsection, we prove that ``de Rham'' implies ``potentially semistable''.

First, we consider the reduced case.
Let $(A,A^+)$ be a reduced affinoid pair over $(\Qbb_p,\Zbb_p)$ of weakly finite type and we set $\Acal=(A,A^+)_{\square}$.
We consider $\Acal_f=(A_f,A^+)_{\square}$ for $f\in A$.
We assume that $A\to A_f$ is injective.
In this case, we can take a Banach $\Qbb_p$-algebra $E$ which is a finite product of complete discretely valued fields with algebraically closed residue field and an injection $A_f \to E$ such that $A\to E$ is a closed embedding (\cite[Corollaire 2.1.4]{BC08}).

\begin{lemma}
	The morphism $A\to E$ has a retraction as a morphism of Banach $\Qbb_p$-vector spaces.
\end{lemma}
\begin{proof}
	It is enough to construct a section of $E\to C=E/A$.
	By the open mapping theorem, the image $C^{\circ}$ of $E^{\circ}$ in $C$ is a module of definition.
	Since $C^{\circ}$ is $p$-torsion free and $p$-adically complete, it is of the form $\hat{\oplus}\Zbb_p$.
	Therefore, we can take a section of $E^{\circ}\to C^{\circ}$.
\end{proof}

\begin{theorem}\label{thm:intersetion}
	The intersection of $B_{\dR,A}$ and $K\otimes_{K_0}\tilde{B}^{[p^{-1},1]}_{\log,\Cbb_p,E}[1/t]$ in $B_{\dR,E}$ is equal to $K\otimes_{K_0}\tilde{B}_{\log,\Cbb_p,R}^{[p^{-1},1]}[1/t]$
\end{theorem}
\begin{proof}
	We take a splitting $E=A\oplus C$ as Banach $\Qbb_p$-vector spaces.
	Then we have the following isomorphisms:
	\begin{align*}
		&B_{\dR,A}\cong B_{\dR}\otimes A,\\
		&B_{\dR,E}\cong (B_{\dR}\otimes A) \oplus (B_{\dR}\otimes C),\\
		&K\otimes_{K_0}\tilde{B}_{\log,\Cbb_p,E}^{[p^{-1},1]}[1/t]\cong ((K\otimes_{K_0}\tilde{B}_{\log,\Cbb_p}^{[p^{-1},1]}[1/t])\otimes A)\oplus((K\otimes_{K_0}\tilde{B}_{\log,\Cbb_p}^{[p^{-1},1]}[1/t])\otimes C).
	\end{align*}
	The claim easily follows from these isomorphisms.
\end{proof}

\begin{corollary}\label{cor:intersection}
 	The intersection of $B_{\sdR,A_f}$ and $K\otimes_{K_0}\tilde{B}^{[p^{-1},1]}_{\log,\Cbb_p,E}[1/t]$ in $B_{\dR,E}$ is equal to $K\otimes_{K_0}\tilde{B}^{[p^{-1},1]}_{\log,\Cbb_p,A_f}[1/t]$
\end{corollary}
\begin{proof}
	It easily follows from Theorem~\ref{thm:intersetion}, 
	$$B_{\sdR,A_f}=B_{\dR,A}[1/f],$$ 
	and 
	$$K\otimes_{K_0}\tilde{B}^{[p^{-1},1]}_{\log,\Cbb_p,A_f}[1/t]=(K\otimes_{K_0}\tilde{B}^{[p^{-1},1]}_{\log,\Cbb_p,A}[1/t])[1/f].$$
\end{proof}

\begin{theorem}\label{thm:p-adic monodromy reduced}
	Let $V$ be a de Rham $G_K$-equivariant vector bundle over $X_{\Cbb_p,\Acal_f}$.
	Then $V$ is potentially semistable.
\end{theorem}
\begin{proof}
	We write $V_E=V\otimes_{A_f} E$, which is a $G_K$-equivariant vector bundle over $X_{\Cbb_p,(E,\Ocal_E)_{\square}}$.
	By \cite[6.2]{BC08}, \cite{Bergererrata}, and Theorem~\ref{thm:comparison with the classical definition}, there is a finite extension $L/K$ such that the natural morphism
	\begin{align*}
		(V_{\log,\Cbb_p,E}^{[p^{-1},1]}[1/t])^{G_L}\otimes_{L_0\otimes E} \tilde{B}_{\log,\Cbb_p,E}^{[p^{-1},1]}[1/t]\to V_{\log,\Cbb_p,E}^{[p^{-1},1]}[1/t]
	\end{align*}
	is an isomorphism. 
	By Corollary~\ref{thm:intersetion}, there is a Cartesian diagram
	\begin{align*}
		\xymatrix{
			V_{\dR} \ar[r] & V_{\dR,E}\\
			L\otimes_{L_0} V_{\log}^{[p^{-1},1]}[1/t]\ar[r]\ar[u] & L\otimes_{L_0} V_{\log,E}^{[p^{-1},1]}[1/t].\ar[u]
	}
	\end{align*}
	By taking the $G_L$-invariant vectors, we obtain a Cartesian diagram
	\begin{align*}
		\xymatrix{
			V_{\dR}^{G_L} \ar[r] & V_{\dR,E}^{G_L}\\
			L\otimes_{L_0} (V_{\log}^{[p^{-1},1]}[1/t])^{G_L}\ar[r]\ar[u] & L\otimes_{L_0} (V_{\log,E}^{[p^{-1},1]}[1/t])^{G_L},\ar[u]
	}
	\end{align*}
	and the right vertical morphism is an isomorphism.
	Therefore, the morphism
	$$L\otimes_{L_0} D_{\st}^L(V)=L\otimes_{L_0} (V_{\log}^{[p^{-1},1]}[1/t])^{G_L} \to V_{\dR}^{G_L}=D_{\dR}^L(V)$$
	is an isomorphism.
	For $\mfrak \in \Spm A_f$, there is a natural isomorphism
	$$L\otimes_{L_0} D_{\st}^L(V)\otimes_{A_f}A_f/\mfrak \cong D_{\dR}^L(V) \otimes_{A_f} A_f/\mfrak\cong D_{dR}^L(V/\mfrak),$$
	where the second isomorphism follows since $V$ is de Rham.
	On the other hand, there is a natural injection
	$$L\otimes_{L_0} D_{\st}^L(V/\mfrak)\to D_{dR}^L(V/\mfrak).$$
	Since the composition
	$$L\otimes_{L_0} D_{\st}^L(V)\otimes_{A_f}A_f/\mfrak\to L\otimes_{L_0} D_{\st}^L(V/\mfrak)\to D_{dR}^L(V/\mfrak)$$
	is an isomorphism, we find that the natural morphism   
	$$D_{\st}^L(V)\otimes_{A_f}A_f/\mfrak\to D_{\st}^L(V/\mfrak)$$
	is an isomorphism, and that $V/\mfrak$ is semistable as a $G_L$-equivariant vector bundle over $X_{\Cbb_p,\Acal_f/\mfrak}$.
	Therefore, $V$ is a semistable $G_L$-equivariant vector bundle over $X_{\Cbb_p,\Acal_f}$ by Theorem~\ref{thm:log-crys criterion}. 
\end{proof}

Next, we aim to extend the $p$-adic monodromy theorem by removing the reducedness assumption.

\begin{lemma}\label{lem:log-crys zariski dense criterion}
	Let $\Acal=(A,A^+)_{\square}$ be an algebraic-affinoid analytic $\Qbb_{p,\square}$-algebra, and let $f\in A$.
	We take an integer $n\geq 1$ such that $A[f^{\infty}]=A[f^n]$, where $A[f^n]=\Ker(f^n\colon A\to A)$ and $A[f^{\infty}]=\bigcup_m A[f^m]$.
	(Since $A$ is noetherian, such $n$ exists.)
	Then a $G_K$-equivariant vector bundle $V$ over $X_{\Cbb_p,\Acal}$ is semistable if and only if $V_f=V\otimes_{\Acal}\Acal_f$ over $X_{\Cbb_p,\Acal_f}$ and $V/f^n=V\otimes_{\Acal}\Acal/f^n$ over $X_{\Cbb_p,\Acal/f^n}$ are semistable.
\end{lemma}
\begin{proof}
	The only if part is clear. Let us prove the if part.
	Let $A'$ be the image of $A$ in $A_f$.
	There is an exact sequence of $A$-modules
	$$0\to A\to A_f\times A/f^nA\to A_f/f^nA'\to 0.$$
	Therefore, we obtain an exact sequence
	\begin{align}\label{eq1}
	0\to V_{\log}^{[p^{-1},1]}[1/t] &\to (V_f)_{\log}^{[p^{-1},1]}[1/t]\times (V/f^n)_{\log}^{[p^{-1},1]}[1/t] \\
	&\to V_{\log}^{[p^{-1},1]}[1/t]\otimes_{A}A_f/f^nA'\to 0\notag.
	\end{align}
	We note that there is a natural isomorphism
	$$V_{\log}^{[p^{-1},1]}[1/t]\otimes_{A}A_f/f^nA'\cong \varinjlim_m V_{\log}^{[p^{-1},1]}[1/t]\otimes_{A}f^{-m}A'/f^nA'.$$
	To simplify the notation, we write $V_{\log}^{[p^{-1},1]}[1/t]\otimes_{A}f^{-m}A'/f^nA'=W_m$, which is a finite projective $\tilde{B}_{\log,\Cbb_p,A'/f^{m+n}}^{[p^{-1},1]}[1/t]$-module.
	By taking $G_K$-invariant vectors, we obtain an left exact sequence
	$$0\to D_{\st}(V) \to D_{\st}(V_f)\times D_{\st}(V/f^n) \to (\varinjlim_{m}W_m)^{G_K}.$$
	By applying $-\otimes_{K_0\otimes A}\tilde{B}_{\log,\Cbb_p,A}^{[p^{-1},1]}[1/t]$, we get a left exact sequence
	\begin{align}\label{eq2}
		0\to D_{\st}(V)\otimes\tilde{B}_{\log,\Cbb_p,A}^{[p^{-1},1]}[1/t] &\to (D_{\st}(V_f)\times D_{\st}(V/f^n)) \otimes\tilde{B}_{\log,\Cbb_p,A}^{[p^{-1},1]}[1/t] \\ 
		&\to(\varinjlim_{m}W_m)^{G_K}\otimes\tilde{B}_{\log,\Cbb_p,A}^{[p^{-1},1]}[1/t] \notag.
	\end{align}
	There is a natural morphism from the left exact sequence \eqref{eq2} to the left exact sequence \eqref{eq1}.
	By the assumption, the morphism in the middle term
	$$(D_{\st}(V_f)\times D_{\st}(V/f^n)) \otimes\tilde{B}_{\log,\Cbb_p,A}^{[p^{-1},1]}[1/t] \to 
	(V_f)_{\log}^{[p^{-1},1]}[1/t]\times (V/f^n)_{\log}^{[p^{-1},1]}[1/t]$$
	is an isomorphism.
	By Lemma~\ref{lem:log-crys inj}, the morphisms in the left and right terms
	\begin{align*}
		D_{\st}(V)\otimes\tilde{B}_{\log,\Cbb_p,A}^{[p^{-1},1]}[1/t]\to V_{\log}^{[p^{-1},1]}[1/t],\\
		(\varinjlim_{m}W_m)^{G_K}\otimes\tilde{B}_{\log,\Cbb_p,A}^{[p^{-1},1]}[1/t] \to \varinjlim_{m}W_m
	\end{align*}
	are injective.
	Therefore, the morphism in the left terms
	$$D_{\st}(V)\otimes\tilde{B}_{\log,\Cbb_p,A}^{[p^{-1},1]}[1/t]\to V_{\log}^{[p^{-1},1]}[1/t]$$
	is also an isomorphism.
	In other words, $V$ is semistable.
\end{proof}

\begin{theorem}[$p$-adic monodromy theorem]\label{thm:p-adic monodromy}
	Let $\Acal=(A,A^+)_{\square}$ be an algebraic-affinoid analytic $\Qbb_{p,\square}$-algebra.
	Then de Rham $G_K$-equivariant vector bundles over $X_{\Cbb_p,\Acal}$ are potentially semistable.
\end{theorem}
\begin{proof}
	By \cite[Proposition 2.37]{Mikami24}, we may assume that $A$ is of the form $R_f$ where $R$ is an affinoid $\Qbb_p$-algebra and $f\in R$.
	We proceed by induction on $\dim A$. 
	If $\dim A=0$, then $A$ is a finite $\Qbb_p$-algebra.
	Therefore, the claim follows from the case where $A=\Qbb_p$.
	We assume $d=\dim A>0$.
	By Noether normalization theorem, there is a finite injection 
	$$\Qbb_p\langle T_1,\ldots, T_d\rangle \to R.$$
	Then by \cite[02ML,051T]{stacks-project}, there is an element $g\in \Qbb_p\langle T_1,\ldots, T_d\rangle$ such that the morphism $$\Qbb_p\langle T_1,\ldots, T_d\rangle_g \to R_{fg}$$ 
	is finite flat, that is, $R_{fg}$ is finite projective as a $\Qbb_p\langle T_1,\ldots, T_d\rangle_g$-module.
	The $p$-adic monodromy theorem for $R_{fg}$ follows from that for $\Qbb_p\langle T_1,\ldots, T_d\rangle_g$, which is proved in Theorem~\ref{thm:p-adic monodromy reduced}.
	Let $m$ be a positive integer such that $R_f[g^{\infty}]=R_f[g^m]$.
	Since $R_f/g^m$ is quasi-finite over $\Qbb_p\langle T_1,\ldots, T_d\rangle/g^m$ whose dimension is smaller than $d$, we get $\dim R_f/g^m<d$.
	Therefore, the $p$-adic monodromy theorem for $R_f/g^m=(R/g^m)_f$ follows from the induction hypothesis.
	Finally, the $p$-adic monodromy theorem for $A=R_f$ follows from Lemma~\ref{lem:log-crys zariski dense criterion}.
\end{proof}

\subsection{Semistable representations and filtered $(\varphi,N)$-modules}

In this subsection, we consider the relationship between semistable representations and filtered $(\varphi,N)$-modules.
Let $\Acal=(A,A^+)_{\square}$ be an algebraic-affinoid analytic $\Qbb_{p,\square}$-algebra.

\begin{definition}
	\begin{enumerate}
	\item A \textit{filtered $(\varphi,N)$-module over $K_0\otimes A$} is a pair $(D_0, \Fil^{\bullet}D)$, where
	\begin{itemize}
		\item $D_0$ is a finite projective $K_0\otimes A$-module with 
		\begin{itemize}
			\item a $K_0\otimes A$-linear endomorphism $N$,
			\item a $\varphi$-semilinear automorphism $\varphi$
		\end{itemize}
		satisfying $p\varphi N=N\varphi$.
		\item $\Fil^{\bullet}D$ is a filtration of $D=K\otimes_{K_0}D_0$ satisfying the condition in Proposition~\ref{prop:filtered module BdR+}.
	\end{itemize}
	Let $\Fil_{K,A}^{\varphi,N}$ denote the category of filtered $(\varphi,N)$-modules over $K_0\otimes A$.
	For $(D_0, \Fil^{\bullet}D), (D'_0, \Fil^{\bullet}D') \in \Fil_{K,A}^{\varphi,N}$, we define their tensor product as 
	$$(D_0, \Fil^{\bullet}D)\otimes(D'_0, \Fil^{\bullet}D')=(D_0\otimes D'_0, \Fil^{\bullet}(D\otimes D')),$$ 
	where $\Fil^{\bullet}(D\otimes D')$ is the tensor product defined in Definition~\ref{def:Fil}.
	This defines a symmetric monoidal structure on $\Fil_{K,A}^{\varphi,N}$.
	\item Let $L/K$ be a finite Galois extension.
	A \textit{filtered $(\varphi,N,\Gal(L/K))$-module over $L_0\otimes A$} is a pair $(D_{L_0}, \Fil^{\bullet}D)$, where
	\begin{itemize}
		\item $D_{L_0}$ is a finite projective $L_0\otimes A$-module with 
		\begin{itemize}
			\item a semilinear action of $\Gal(L/K)$,
			\item a $\Gal(L/K)$-equivariant $L_0\otimes A$-linear endomorphism $N$,
			\item a $\Gal(L/K)$-equivariant $\varphi$-semilinear automorphism $\varphi$
		\end{itemize}
		satisfying $p\varphi N=N\varphi$.
		\item $\Fil^{\bullet}D$ is a filtration of the finite projective $K\otimes A$-module $D=(L\otimes_{L_0}D_{L_0})^{\Gal(L/K)}$ satisfying the condition in Proposition~\ref{prop:filtered module BdR+}.
	\end{itemize}
	Let $\Fil_{L/K,A}^{\varphi,N,\Gal(L/K)}$ denote the category of filtered $(\varphi,N,\Gal(L/K))$-modules over $L_0\otimes A$.
	As in (1), we define a symmetric monoidal structure on $\Fil_{L/K,A}^{\varphi,N,\Gal(L/K)}$.
	\item A \textit{filtered $(\varphi,N,G_K)$-module over $K_0^{\unr}\otimes A$} is a pair $(D^{\unr}, \Fil^{\bullet}D)$, where
	\begin{itemize}
		\item  $D^{\unr}$ is a finite projective $K_0^{\unr}\otimes A$-module with 
		\begin{itemize}
			\item a semilinear smooth action of $G_K$ (i.e., $D^{\unr}=\bigcup_{L/K}(D^{\unr})^{\Gal(L/K)}$ where $L/K$ runs over all finite extensions),
			\item a $G_K$-equivariant $K_0^{\unr}\otimes A$-linear endomorphism $N$,
			\item a $G_K$-equivariant $\varphi$-semilinear automorphism $\varphi$
		\end{itemize}
		satisfying $p\varphi N=N\varphi$.
		\item $\Fil^{\bullet}D$ is a filtration of the finite projective $K\otimes A$-module $D=(\bar{K}\otimes_{K_0^{\unr}}D^{\unr})^{G_K}$ satisfying the condition in Proposition~\ref{prop:filtered module BdR+}.
	\end{itemize}
	Let $\Fil_{K,A}^{\varphi,N,\Gal_K}$ denote the category of filtered $(\varphi,N,G_K)$-modules over $K_0^{\unr}\otimes A$.
	As in (1), we define a symmetric monoidal structure on $\Fil_{K,A}^{\varphi,N,\Gal_K}$.
	\end{enumerate}
\end{definition}

\begin{lemma}\label{lemma:MF transition}
	\begin{enumerate}
	\item For finite Galois extensions $L$ and $L'$ of $K$ such that $L\subset L'$, the functor 
	$$\iota_{L,L'}\colon \Fil_{L/K,A}^{\varphi,N,\Gal(L/K)}\to \Fil_{L'/K,A}^{\varphi,N,\Gal(L'/K)};\; (D_{L_0},\Fil^{\bullet}D)\mapsto (L'_0\otimes_{L_0}D_{L_0},\Fil^{\bullet}D),$$
	where the action of $\Gal(L'/K)$ on $L'_0\otimes_{L_0}D_{L_0}$ is given by the diagonal action, is well-defined, symmetric monoidal, and fully faithful.
	\item For a finite Galois extension $L/K$, the functor
	$$\iota_{L}\colon \Fil_{L/K,A}^{\varphi,N,\Gal(L/K)}\to \Fil_{K,A}^{\varphi,N,\Gal_K};\; (D_{L_0},\Fil^{\bullet}D)\mapsto (K_0^{\unr}\otimes_{L_0} D_{L_0},\Fil^{\bullet}D),$$
	where the action of $G_K$ on $L_0^{\unr}\otimes_{L_0} D_{L_0}$ is given by the diagonal action, is well-defined, symmetric monoidal, and fully faithful.
	Moreover, for a finite Galois extension $L'/K$ such that $L\subset L'$, we have $\iota_{L'}\circ\iota_{L,L'}=\iota_L$.
	\item Any object of $\Fil_{K,A}^{\varphi,N,\Gal_K}$ lies in the essential image of $\iota_L$ for some finite extension $L/K$.
	In other words, there is a categorical equivalence $$\varinjlim_{L/K} \Fil_{L/K,A}^{\varphi,N,\Gal(L/K)} \simeq \Fil_{K,A}^{\varphi,N,\Gal_K},$$
	where $L/K$ runs over all finite Galois extensions.
	\end{enumerate}
\end{lemma}
\begin{proof}
	First, we prove (1) and (2).
	It is clear that the action of $G_K$ on $L_0^{\unr}\otimes_{L_0} D_{L_0}$ is smooth.
	Moreover, we have $\bar{K}\otimes_{L_0^{\unr}}(L_0^{\unr}\otimes_{L_0} D_{L_0})=\bar{K}\otimes_L (L\otimes_{L_0} D_{L_0})$.
	Since $(\bar{K}\otimes_L (L\otimes_{L_0} D_{L_0}))^{G_L}=L\otimes_{L_0} D_{L_0}$, we find that $(\bar{K}\otimes_{L_0^{\unr}}(L_0^{\unr}\otimes_{L_0} D_{L_0}))^{G_K}=(L\otimes_{L_0} D_{L_0})^{\Gal(L/K)}.$
	Therefore, $\iota_L$ is well-defined.
	The symmetric monoidality of $\iota_L$ directly follows from the definition.
	Moreover, since $(L_0^{\unr}\otimes_{L_0} D_{L_0})^{G_L}=D_{L_0}$, $(D_{L_0},\Fil^{\bullet}D)$ can be canonically recovered from $(L_0^{\unr}\otimes_{L_0} D_{L_0},\Fil^{\bullet}D)$.
	Therefore, $\iota_L$ is fully faithful.
	Similarly, we can prove that $\iota_{L,L'}$ is well-defined and symmetric monoidal, and that $\iota_{L'}\circ\iota_{L,L'}=\iota_L$, which proves (1) and (2).

	Next, we prove (3).
	Let $(D^{\unr}, \Fil^{\bullet}D)\in \Fil_{K,A}^{\varphi,N,\Gal_K}$.
	Since the action of $G_K$ on the finite projective $K_0^{\unr}\otimes A$-module $D^{\unr}$ is smooth, there exists a finite Galois extension $L'/K^{\unr}$ such that the action of $G_{L'}$ on $D^{\unr}$ is trivial.
	We take a finite Galois extension $L/K$ such that $L^{\unr}=L'$.
	Since we have $L^{\unr}=L\otimes_{L_0}K_0^{\unr}$, we get $\Gal(L^{\unr}/L)\cong \Gal(K_0^{\unr}/L_0)$.
	Via this isomorphism, we get the semilinear action of $\Gal(K_0^{\unr}/L_0)$ on the finite projective $K_0^{\unr}\otimes A$-module $D^{\unr}$.
	We set $D_{L_0}=(D^{\unr})^{\Gal(K_0^{\unr}/L_0)}=(D^{\unr})^{\Gal(L^{\unr}/L)}$.
	It admits a natural semilinear $\Gal(L/K)$-action and is stable under the action of $\varphi$ as the actions of $\varphi$ and $G_K$ commute. 
	By the Galois descent, the natural morphism $K_0^{\unr}\otimes_{L_0} D_{L_0}\to D^{\unr}$ is an isomorphism, and by construction, it is $\Gal(L^{\unr}/K)$-equivariant.
	Since 
	$$(L\otimes_{L_0} D_{L_0})^{\Gal(L/K)}=(\bar{K}\otimes_{L_0} D_{L_0})^{G_K}=(\bar{K}\otimes_{K_0^{\unr}} D^{\unr})^{G_K}=D,$$ 
	we obtain $(D_{L_0},\Fil^{\bullet}D)\in \Fil_{K,A}^{\varphi,N,\Gal_K}$ which satisfies $\iota_L(D_{L_0},\Fil^{\bullet}D)=(D^{\unr}, \Fil^{\bullet}D).$
\end{proof}

\begin{definition}
Let $\Vect^{\st}(X_{\Cbb_p,\Acal}/G_K)$ denote the category of semistable $G_K$-equivariant vector bundles over $X_{\Cbb_p,\Acal}$.
For a finite Galois extension $L/K$, let $\Vect^{L/K,\st}(X_{\Cbb_p,\Acal}/G_K)$ denote the category of $G_K$-equivariant vector bundles over $X_{\Cbb_p,\Acal}$ that are semistable as $G_L$-equivariant vector bundles.
\end{definition}

By the $p$-adic monodromy theorem (Theorem~\ref{thm:p-adic monodromy}), we have $\Vect^{\dR}(X_{\Cbb_p,\Acal}/G_K)=\bigcup_{L/K}\Vect^{L/K,\st}(X_{\Cbb_p,\Acal}/G_K),$ where $L/K$ runs over all finite Galois extensions.
The goal of this subsection is to construct the natural categorical equivalences
\begin{align*}
		\Vect^{\st}(X_{\Cbb_p,\Acal}/G_K) &\simeq \Fil_{K,A}^{\varphi,N},\\
		\Vect^{L/K,\st}(X_{\Cbb_p,\Acal}/G_K)&\simeq \Fil_{L/K,A}^{\varphi,N,\Gal(L/K)},\\
		\Vect^{\dR}(X_{\Cbb_p,\Acal}/G_K) &\simeq \Fil_{K,A}^{\varphi,N,\Gal_K}.
\end{align*}

First, we have already constructed the following functors.
\begin{proposition}\label{prop:semist to MF}
	Let $L/K$ be a finite Galois extension.
	Then the following functors are well-defined and symmetric monoidal
	\begin{align*}
		\Vect^{\st}(X_{\Cbb_p,\Acal}/G_K) \to \Fil_{K,A}^{\varphi,N}&;\; V \mapsto (D_{\st}(V), \Fil^{\bullet}D_{\dR}(V)),\\
		\Vect^{L/K,\st}(X_{\Cbb_p,\Acal}/G_K)\to \Fil_{L/K,A}^{\varphi,N,\Gal(L/K)}&;\; V \mapsto (D_{\st}^L(V), \Fil^{\bullet}D_{\dR}(V)),\\
		\Vect^{\dR}(X_{\Cbb_p,\Acal}/G_K) \to \Fil_{K,A}^{\varphi,N,\Gal_K}&;\; V\mapsto (D_{\pst}(V), \Fil^{\bullet}D_{\dR}(V)),
	\end{align*}
	where $D_{\pst}(V)$ denotes the finite projective $K_0^{\unr}\otimes A$-module $\varinjlim_{L/K} D_{\st}^{L}(V).$
	Moreover, for $V\in \Vect^{L/K,\st}(X_{\Cbb_p,\Acal}/G_K)$, there is a natural isomorphism
	$$\iota_L(D_{\st}^L(V), \Fil^{\bullet}D_{\dR}(V))\cong (D_{\pst}(V), \Fil^{\bullet}D_{\dR}(V)).$$
\end{proposition}
\begin{proof}
	For $V\in \Vect^{\st}(X_{\Cbb_p,\Acal}/G_K)$, by Definition~\ref{def:phi,N}, $D_{\st}(V)$ has the monodromy operator $N$ and the Frobenius automorphism $\varphi$.
	By Lemma~\ref{lem:log-crys implies dR}, there is a natural isomorphism $K\otimes_{K_0} D_{\st}(V)\cong D_{\dR}(V)$, and it admits a natural filtration satisfying the condition in Proposition~\ref{prop:filtered module BdR+} by Proposition~\ref{prop:cdR to filtered}.
	Therefore, the functor $$\Vect^{\st}(X_{\Cbb_p,\Acal}/G_K) \to \Fil_{K,A}^{\varphi,N};\; V \mapsto (D_{\st}(V), \Fil^{\bullet}D_{\dR}(V))$$
	is well-defined.
	The symmetric monoidality of this functor directly follows from the definition and Theorem~\ref{thm:cdR filtered equivalence}.

	Next, let $V\in\Vect^{L/K,\st}(X_{\Cbb_p,\Acal}/G_K)$.
	Then by definition, $D_{\st}^L(V)=(V_{\log}^{[p^{-1},1]}[1/t])^{G_L}$ has a natural semilinear action of $\Gal(L/K)=G_K/G_L$ and it is compatible with the monodromy operator $N$ and the Frobenius automorphism $\varphi$.
	Moreover, by definition, we have $(L\otimes_{L_0} D_{\st}(V))^{\Gal(L/K)}\cong D_{\dR}^L(V)^{\Gal(L/K)}=D_{\dR}(V)$, so the functor $$\Vect^{L/K,\st}(X_{\Cbb_p,\Acal}/G_K)\to \Fil_{L/K,A}^{\varphi,N,\Gal(L/K)};\; V \mapsto (D_{\st}^L(V), \Fil^{\bullet}D_{\dR}(V))$$
	is also well-defined.
	It is also symmetric monoidal by construction and Theorem~\ref{thm:cdR filtered equivalence}.
	By using the $p$-adic monodromy theorem (Theorem~\ref{thm:p-adic monodromy}), we can easily prove the well-definedness and symmetric monoidality of $$\Vect^{\dR}(X_{\Cbb_p,\Acal}/G_K) \to \Fil_{K,A}^{\varphi,N,\Gal_K};\; V\mapsto (D_{\pst}(V), \Fil^{\bullet}D_{\dR}(V))$$
	and $\iota_L(D_{\st}^L(V), \Fil^{\bullet}D_{\dR}(V))\cong (D_{\pst}(V), \Fil^{\bullet}D_{\dR}(V))$ for $V\in\Vect^{L/K,\st}(X_{\Cbb_p,\Acal}/G_K)$.
\end{proof}

We want to construct quasi-inverse functors of them.
\begin{lemma}\label{lem:nonzero slope}
	Let $L/K$ be a finite Galois extension, and let $n$ be a non-zero integer.
	If $x\in L_0\otimes A$ satisfies $p^nx=\varphi(x)$, then $x=0$.
\end{lemma}
\begin{proof}
	We set $f=[L_0:\Qbb_p]$.
	Then we have $\varphi^f=\id$.
	Therefore, we get $p^{nf}x=x$.
	Since $1-p^{nf}$ is a unit in $L_0\otimes A$, we get $x=0$.
\end{proof}

\begin{lemma}\label{lem:exp nilpotent}
	Let $L/K$ be a finite Galois extension.
	For $(D_{L_0}, \Fil^{\bullet}D)\in \Fil_{L/K,A}^{\varphi,N,\Gal(L/K)}$, the monodromy operator $N$ on $D_{L_0}$ is nilpotent.
\end{lemma}
\begin{proof}
	Since $A$ is noetherian, the number of connected components of $\Spec A$ is finite.
	Therefore, we may assume that $\Spec A$ is connected.
	Then $\varphi$ acts on $\Spec L_0\otimes A$ transitively, so the rank of the finite projective $L_0\otimes A$-module $D_{L_0}$ is constant.
	We set $n=\rk D_{L_0}$.
	Let $F(T)=T^n+a_{n-1}T^{n-1}+\cdots+a_0\in (L_0\otimes A)[T]$ be the characteristic polynomial of $N\colon D_{L_0}\to D_{L_0}$.
	The equation $p\varphi N=N\varphi$ implies $p^n\varphi(F)(T)=F(pT)$.
	By Lemma~\ref{lem:nonzero slope}, we get $F(T)=T^n$.
	Therefore, by the Cayley-Hamilton theorem, we get $N^n=0$.
\end{proof}

\begin{lemma}\label{lem:monodromy N=0}
	Let $R$ be a $\Qbb$-algebra.
	We define an $R$-derivation $N\colon R[T]\to R[T]$ by $N(T)=-1$.
	Let $M$ be a finite projective $R$-module with a nilpotent endomorphism $N\colon M\to M$.
	We extend $N$ to an $R$-linear endomorphism of $M[T]$ by setting 
	$$N(mT^r)=N(m)T^r - rmT^{r-1}.$$
	Then the kernel $M[T]^{N=0}=\Ker(N\colon M[T]\to M[T])$ is a finite projective $R$-module, and the natural morphism
	$$M[T]^{N=0}\otimes_R R[T]\to M[T]$$
	is an isomorphism.
\end{lemma}
\begin{proof}
    We define an $R[T]$-linear automorphism $\tau \colon M[T]\to M[T]$ by setting
    $$\tau(m) = \sum_{k=0}^{\infty} \frac{1}{k!} N^k(m) T^k$$
    for $m\in M$, and extending it $R[T]$-linearly to $M[T]$.
    Since $N$ is nilpotent, the sum is finite and $\tau$ is well-defined.
    The inverse of $\tau$ is given by the same formula with $T$ replaced by $-T$, so $\tau$ is an isomorphism.
    For any $m\in M$, we have
    \begin{align*}
        N(\tau(m)) &= N\left(\sum_{k=0}^{\infty} \frac{1}{k!} N^k(m) T^k\right) \\
        &= \sum_{k=0}^{\infty} \frac{1}{k!} N^{k+1}(m) T^k - \sum_{k=1}^{\infty} \frac{1}{k!} N^k(m) k T^{k-1} \\
        &= \sum_{k=0}^{\infty} \frac{1}{k!} N^{k+1}(m) T^k - \sum_{k=1}^{\infty} \frac{1}{(k-1)!} N^k(m) T^{k-1} \\
        &= 0.
    \end{align*}
    Therefore, we get $\tau(m) \in M[T]^{N=0}$.  
    By the Leibniz rule, for any $m\in M$ and $P \in R[T]$, we have
    $$N(\tau(m) P) = N(\tau(m)) P + \tau(m) N(P) = \tau(m) N(P).$$
    This implies that $N(\tau(v)) = \tau(N'(v))$ for all $v \in M[T]$, where $N'$ denotes the endomorphism on $M[T]$ defined by $N'(mP)=mN(P)$.
    Since the kernel of $N'$ on $M[T]$ is exactly $M$ (regarded as degree zero polynomials), $\tau$ induces an $R$-module isomorphism
    $$M \xrightarrow{\sim} M[T]^{N=0}, \quad m \mapsto \tau(m).$$
    Since $M$ is a finite projective $R$-module, $M[T]^{N=0}$ is also a finite projective $R$-module.
    Finally, the natural morphism $M[T]^{N=0}\otimes_R R[T]\to M[T]$ corresponds to
    $$\tau(m)\otimes P \mapsto \tau(m)P=\tau(mP),$$
    which coincides with the isomorphism $\tau$.
\end{proof}

\begin{proposition}\label{prop:outside t construction}
	Let $L/K$ be a finite Galois extension, and let $I\subset (0,\infty)$ be a closed interval such that $p^n\in I$ for some $n\in \Zbb$.
	Then the following categories are equivalent:
	\begin{enumerate}
		\item Finite projective $\tilde{B}^{I}_{\Cbb_p,A}[1/t]$-modules $V^I$ with a semilinear $G_K$-action such that the natural morphism 
		$$(V_{\log}^I)^{G_L}\otimes_{L_0\otimes A} \tilde{B}_{\log,\Cbb_p,A}^{I}[1/t]\to V^{I},$$
		where $V_{\log}^I=V^I\otimes_{\tilde{B}^{I}_{\Cbb_p,A}[1/t]}\tilde{B}^{I}_{\log,\Cbb_p,A}[1/t]$, is an isomorphism.
		\item Finite projective $L_0\otimes A$-modules $D_0$ with a semilinear $\Gal(L/K)$-action and a $\Gal(L/K)$-equivariant nilpotent endomorphism $N$.
	\end{enumerate}
	The equivalence is given by functors
	$$V^I\mapsto (V_{\log}^I)^{G_L}$$ and 
	$$D_0\mapsto (D_0\otimes_{L_0\otimes A}\tilde{B}_{\log,\Cbb_p,A}^{I}[1/t])^{N=0}.$$
\end{proposition}
\begin{proof}
	First, we prove the well-definedness of both functors. 
	By the same argument as in Lemma~\ref{lem:Dst fin proj}, we get $(V_{\log}^I)^{G_L}$ is a finite projective $L_0\otimes A$-module (with a $\Gal(L/K)$-action and a $\Gal(L/K)$-equivariant monodromy operator $N$).
	The monodromy operator $N$ on $\tilde{B}_{\log,\Cbb_p,A}^{I}[1/t]$ is locally nilpotent, so the monodromy operator $N$ on $(V_{\log}^I)^{G_L}$ is also locally nilpotent.
	Since $(V_{\log}^I)^{G_L}$ is a finite projective $L_0\otimes A$-module, the monodromy operator $N$ on $(V_{\log}^I)^{G_L}$ is nilpotent.
	Therefore, $V^I\mapsto (V_{\log}^I)^{G_L}$ is well-defined.
	The well-definedness of $D_0\mapsto (D_0\otimes_{L_0\otimes A}\tilde{B}_{\log,\Cbb_p,A}^{I}[1/t])^{N=0}$ easily follows from Lemma~\ref{lem:exp nilpotent}.

	Next, we prove that the natural morphism
	$$((V_{\log}^I)^{G_L}\otimes_{L_0\otimes A}\tilde{B}_{\log,\Cbb_p,A}^{I}[1/t])^{N=0}\to V^I$$
	is an isomorphism.
	It follows from the condition that the natural morphism 
	$$(V_{\log}^I)^{G_L}\otimes_{L_0\otimes A}\tilde{B}_{\log,\Cbb_p,A}^{I}[1/t]\to V_{\log}^I$$
	is an isomorphism and the fact that $(V^I_{\log})^{N=0}=V^I$.

	Finally, we prove that the natural morphism
	$$((D_0\otimes_{L_0\otimes A}\tilde{B}_{\log,\Cbb_p,A}^{I}[1/t])^{N=0}\otimes_{\tilde{B}^{I}_{\Cbb_p,A}[1/t]}\tilde{B}^{I}_{\log,\Cbb_p,A}[1/t])^{G_L}\to D_0$$
	is an isomorphism.
	By Lemma~\ref{lem:exp nilpotent}, the natural morphism
	$$(D_0\otimes_{L_0\otimes A}\tilde{B}_{\log,\Cbb_p,A}^{I}[1/t])^{N=0}\otimes_{\tilde{B}^{I}_{\Cbb_p,A}[1/t]}\tilde{B}^{I}_{\log,\Cbb_p,A}[1/t]\to D_0\otimes_{L_0\otimes A}\tilde{B}_{\log,\Cbb_p,A}^{I}[1/t]$$
	is an isomorphism.
	Since $(D_0\otimes_{L_0\otimes A}\tilde{B}_{\log,\Cbb_p,A}^{I}[1/t])^{G_L}\cong D_0\otimes_{L_0\otimes A}(\tilde{B}_{\log,\Cbb_p,A}^{I}[1/t])^{G_L}\cong D_0$ by Proposition~\ref{prop:log fixed point}, we get the claim.
\end{proof}

\begin{notation}
For an integer $k$, we write $\varphi^k_*B_{\dR,A}^+=\varinjlim_{n} \tilde{B}_{\Cbb_p,A}^{[p^{-k},p^{-k}]}/t^n.$
The Frobenius morphism $\varphi^k\colon \tilde{B}_{\Cbb_p,A}^{[1,1]}\to \tilde{B}_{\Cbb_p,A}^{[p^{-k},p^{-k}]}$ induces an isomorphism $\varphi^k\colon B_{\dR,A}^+\xrightarrow{\sim}\varphi^k_*B_{\dR,A}^+$.
For a closed interval $I\subset (0,\infty)$ of the form $I=[p^{-k},p^{-l}]$ for some integers $l\leq k$, we have
$$\varinjlim_{n}\tilde{B}_{\Cbb_p,A}^{I}/t^n \simeq \prod_{l\leq m\leq k}\varphi^m_*B_{\dR,A}^+.$$
\end{notation}

\begin{lemma}\label{lem:Beauville-Laszlo}
Let $L/K$ be a finite Galois extension, and let $I\subset (0,\infty)$ be a closed interval of the form $I=[p^{-k},p^{-l}]$ for some integers $l\leq k$.
Then the following categories are equivalent:
\begin{enumerate}
	\item Finite projective $\tilde{B}^I_{\Cbb_p,A}$-modules $V$ with a semilinear $G_K$-action.
	\item Triples $(V', V_1,\tau)$, where 
	\begin{itemize}
		\item $V'$ is a finite projective $\tilde{B}^I_{\Cbb_p,A}[1/t]$-module with a semilinear $G_K$-action,
		\item $V_1$ is a finite projective $\prod_{l\leq m\leq k}\varphi^m_*B_{\dR,A}^+$-module with a semilinear $G_K$-action,
		\item $\tau \colon V'\otimes_{\tilde{B}^I_{\Cbb_p,A}[1/t]} (\prod_{k\leq m\leq l}\varphi^m_*B_{\dR,A})\simeq V_1[1/t]$ is a $G_K$-equivariant isomorphism.
	\end{itemize}
\end{enumerate}
The equivalence is given by the functors
$V\mapsto (V[1/t], V\otimes_{\tilde{B}^I_{\Cbb_p,A}}(\prod_{l\leq m\leq k}\varphi^m_*B_{\dR,A}^+),\can)$ and $(V', V_1,\tau)\mapsto \Ker(V'\times V_1 \to V_1[1/t])$.
\end{lemma}
\begin{proof}
	This follows from the Beauville-Laszlo theorem (\cite{BL95}).
\end{proof}

\begin{construction}
	Let $L/K$ be a finite Galois extension, and let $(D_{L_0},\Fil^{\bullet}D)\in\Fil_{L/K,A}^{\varphi,N,\Gal(L/K)}$. 
	We construct an object $V\in \Vect^{L/K,\st}(X_{\Cbb_p,\Acal}/G_K)$ from $(D_{L_0},\Fil^{\bullet}D)$.
	First, for a closed interval $I\subset (0,\infty)$ of the form $I=[p^{-k},p^{-l}]$ for some integers $l\leq k$, we obtain a finite projective $\tilde{B}^{I}_{\Cbb_p,A}[1/t]$-module $V'^I=(D_0\otimes_{L_0\otimes A}\tilde{B}_{\log,\Cbb_p,A}^{I}[1/t])^{N=0}$ with a semilinear $G_K$-action by Proposition~\ref{prop:outside t construction}.
	The Frobenius automorphism $\varphi$ on $D_0$ induces the $\varphi$-semilinear isomorphism 
	$$\varphi\otimes \varphi \colon D_0\otimes_{L_0\otimes A}\tilde{B}_{\log,\Cbb_p,A}^{I}[1/t] \xrightarrow{\sim} D_0\otimes_{L_0\otimes A}\tilde{B}_{\log,\Cbb_p,A}^{p^{-1}I}[1/t],$$
	and by applying $(-)^{N=0}$, we obtain the $\varphi$-semilinear isomorphism
	$\varphi \colon V'^I\xrightarrow{\sim}V'^{p^{-1}I}.$
	By construction and the proof of Proposition~\ref{prop:outside t construction}, we have
	\begin{align*}
	V'^{[1,1]}\otimes_{\tilde{B}^{[1,1]}_{\Cbb_p,A}[1/t]} B_{\dR,A} &\cong (D_0\otimes_{L_0\otimes A} \tilde{B}^{[1,1]}_{\log,\Cbb_p,A}[1/t])\otimes_{\tilde{B}^{[1,1]}_{\log,\Cbb_p,A}[1/t]} B_{\dR,A} \\
	&\cong D_0\otimes_{L_0\otimes A}B_{\dR,A}\\
	&\cong D\otimes_{K\otimes A} B_{\dR,A}.
	\end{align*}
	By Theorem~\ref{thm:cdR filtered equivalence}, this yields a $G_K$-stable $B_{\dR,A}^+$-lattice $V_{\dR}\subset V'^{[1,1]}\otimes_{\tilde{B}^{[1,1]}_{\Cbb_p,A}[1/t]} B_{\dR,A}$.
	By using the Frobenius isomorphism
	$$\varphi^m\colon V'^{[1,1]}\to V'^{[p^{-m},p^{-m}]},$$
	we obtain a $G_K$-stable $\varphi^k_*B_{\dR,A}^+$-lattice $\varphi^k_*V_{\dR}\subset V'^{[p^{-m},p^{-m}]}\otimes_{\tilde{B}^{[p^{-m},p^{-m}]}_{\Cbb_p,A}[1/t]} \varphi^m_*B_{\dR,A}$.
	From this, for a closed interval $I=[p^{-k},p^{-l}]$, we define the $G_K$-stable $\prod_{l\leq m\leq k} \varphi^m_*B_{\dR,A}^+$-lattice $$\prod_{l\leq m\leq k}\varphi^m_*V_{\dR}\subset V'^{I}\otimes_{\tilde{B}^{I}_{\Cbb_p,A}[1/t]} (\prod_{l\leq m\leq k}\varphi^m_*B_{\dR,A}).$$
	Therefore, by Lemma~\ref{lem:Beauville-Laszlo}, we obtain a finite projective $\tilde{B}_{\Cbb_p,A}^I$-module $V^I$ with a semilinear $G_K$-action.
	By construction, we have a $\varphi$-semilinear isomorphism
	$\varphi \colon V^I\xrightarrow{\sim}V^{p^{-1}I}.$
	Moreover, for closed intervals $I\subset I'$, there is a restriction morphism
	$V^{I'}\to V^I$ such that $V^{I'}\otimes_{\tilde{B}_{\Cbb_p,A}^{I'}}\tilde{B}_{\Cbb_p,A}^I\to V^I$ is an isomorphism.
	Consequently, $\{V^I\}_I$ defines an object in $\Vect(X_{\Cbb_p,\Acal}/G_K)$, which we denote by $V(D_{L_0},\Fil^{\bullet}D)$.
	By construction, we find that $V(D_{L_0},\Fil^{\bullet}D) \in \Vect^{L/K,\st}(X_{\Cbb_p,\Acal}/G_K)$.
	Thus, this assignment defines a functor 
	$$\Fil_{L/K,A}^{\varphi,N,\Gal(L/K)} \to \Vect^{L/K,\st}(X_{\Cbb_p,\Acal}/G_K);\;(D_{L_0},\Fil^{\bullet}D)\to V(D_{L_0},\Fil^{\bullet}D).$$
	By construction, for a finite Galois extension $L'/K$ such that $L\subset L'$, the following diagram is commutative:
	\begin{align}\label{compatibility}
		\xymatrix{
		\Fil_{L/K,A}^{\varphi,N,\Gal(L/K)} \ar[r]\ar[d]^{\iota_{L,L'}} & \Vect^{L/K,\st}(X_{\Cbb_p,\Acal}/G_K) \ar@{^{(}->}[d]\\
		\Fil_{L'/K,A}^{\varphi,N,\Gal(L'/K)} \ar[r] & \Vect^{L'/K,\st}(X_{\Cbb_p,\Acal}/G_K),
	}
	\end{align}
	where the right vertical morphism is the natural inclusion.

	Let $(D^{\unr}, \Fil^{\bullet}D)\in \Fil_{K,A}^{\varphi,N,\Gal_K}$.
	Then we can take a finite Galois extension $L/K$ and $(D_{L_0},\Fil^{\bullet}D)$ such that $\iota_L(D_{L_0},\Fil^{\bullet}D)\cong(D^{\unr}, \Fil^{\bullet}D)$.
	Using this, we define $V(D^{\unr}, \Fil^{\bullet}D)=V(D_{L_0},\Fil^{\bullet}D)\in \Vect^{\dR}(X_{\Cbb_p,\Acal}/G_K)$.
	By the commutativity of the diagram \eqref{compatibility}, this definition does not depend on the choice of $L$.
	Therefore, we get a functor 
	$$\Fil_{K,A}^{\varphi,N,\Gal_K}\to \Vect^{\dR}(X_{\Cbb_p,\Acal}/G_K) ;\; (D^{\unr}, \Fil^{\bullet}D) \to V(D^{\unr}, \Fil^{\bullet}D).$$
\end{construction}

\begin{theorem}\label{thm:dR filtered phi N Gal-module}
	\begin{enumerate}
	\item The functors 
	$$\Vect^{\st}(X_{\Cbb_p,\Acal}/G_K)\xrightarrow{\sim} \Fil_{K,A}^{\varphi,N};\; V \mapsto (D_{\st}(V), \Fil^{\bullet}D_{\dR}(V))$$
	and 
	$$\Fil_{K,A}^{\varphi,N} \xrightarrow{\sim} \Vect^{\st}(X_{\Cbb_p,\Acal}/G_K);\;(D_0,\Fil^{\bullet}D)\mapsto V(D_0,\Fil^{\bullet}D)$$
	are symmetric monoidal and quasi-inverse to each other.
	\item Let $L/K$ be a finite Galois extension.
	Then the functors 
	$$\Vect^{L/K,\st}(X_{\Cbb_p,\Acal}/G_K)\to \Fil_{L/K,A}^{\varphi,N,\Gal(L/K)};\; V \mapsto (D_{\st}^L(V), \Fil^{\bullet}D_{\dR}(V))$$
	and 
	$$\Fil_{L/K,A}^{\varphi,N,\Gal(L/K)} \to \Vect^{L/K,\st}(X_{\Cbb_p,\Acal}/G_K);\;(D_{L_0},\Fil^{\bullet}D)\mapsto V(D_{L_0},\Fil^{\bullet}D).$$
	are symmetric monoidal and quasi-inverse to each other.
	\item The functors 
	$$\Vect^{\dR}(X_{\Cbb_p,\Acal}/G_K) \to \Fil_{K,A}^{\varphi,N,\Gal_K};\; V\mapsto (D_{\pst}(V), \Fil^{\bullet}D_{\dR}(V))$$
	and 
	$$\Fil_{K,A}^{\varphi,N,\Gal_K}\to \Vect^{\dR}(X_{\Cbb_p,\Acal}/G_K) ;\; (D^{\unr}, \Fil^{\bullet}D) \to V(D^{\unr}, \Fil^{\bullet}D)$$
	are symmetric monoidal and quasi-inverse to each other.
	\end{enumerate}
\end{theorem}
\begin{proof}
    We first verify that the functors are quasi-inverse to each other. 
    For (1) and (2), this follows from the construction, taking Lemma~\ref{lem:monodromy N=0}, Proposition~\ref{prop:outside t construction}, and Lemma~\ref{lem:Beauville-Laszlo} into account. 
    For (3), it follows from Lemma~\ref{lemma:MF transition} and Proposition~\ref{prop:semist to MF}. 
	Finally, since the functors $\Vect^{-}(-) \to \Fil_{-}^{-}$ are symmetric monoidal by Proposition~\ref{prop:semist to MF}, the same holds for their quasi-inverses.
\end{proof}

\begin{remark}
	When $A=\Qbb_p$, by composing the functor obtained above with 
	$$\Vect^{\dR}(X_{\Cbb_p}/G_K)\subset \Vect(X_{\Cbb_p}/G_K)\simeq \{\text{$(\varphi,\Gamma_K)$-modules over the Robba ring $B^{\dagger}_{\rig,K}$}\},$$
	we obtain a functor
	$$\Fil_{K,\Qbb_p}^{\varphi,N,\Gal_K}\to \{\text{$(\varphi,\Gamma_K)$-modules over the Robba ring $B^{\dagger}_{\rig,K}$}\}.$$
	By construction, this functor coincides with the one defined in \cite{Ber02}.
\end{remark}

Let $\Acal=(A,A^+)_{\square}\to \Bcal=(B,B^+)_{\square}$ be a morphism in $\AlgAff_{\Qbb_p}$.
Then there are natural base change functors
\begin{align*}
	-\otimes_{A}B \colon &\Fil_{K,A}^{\varphi,N}\to \Fil_{K,B}^{\varphi,N};\; \\
	&(D_0, \Fil^{\bullet}D) \mapsto (D_0\otimes_{K_0\otimes A} (K_0\otimes B), \Fil^{\bullet}D\otimes_{K\otimes A}(K\otimes B)),\\
	-\otimes_{A}B \colon &\Fil_{L/K,A}^{\varphi,N,\Gal(L/K)}\to \Fil_{L/K,B}^{\varphi,N,\Gal(L/K)};\; \\
	&(D_{L_0}, \Fil^{\bullet}D) \mapsto (D_{L_0}\otimes_{L_0\otimes A} (L_0\otimes B), \Fil^{\bullet}D\otimes_{K\otimes A}(K\otimes B)),\\
	-\otimes_{A}B \colon &\Fil_{K,A}^{\varphi,N,\Gal_K}\to \Fil_{K,B}^{\varphi,N,\Gal_K};\; \\
	&(D^{\unr}, \Fil^{\bullet}D) \mapsto (D^{\unr}\otimes_{K_0^{\unr}\otimes A} (K_0^{\unr}\otimes B), \Fil^{\bullet}D\otimes_{K\otimes A}(K\otimes B)).
\end{align*}

By construction, we get the following.
\begin{lemma}
	The following diagram commutes:
	$$
	\xymatrix{
		\Vect^{\st}(X_{\Cbb_p,\Acal}/G_K) \ar[r]^{-\otimes_{\Acal}\Bcal} \ar[d]^{\simeq}& \Vect^{\st}(X_{\Cbb_p,\Bcal}/G_K)\ar[d]^{\simeq} \\
		\Fil_{K,A}^{\varphi,N}\ar[r]^{-\otimes_{A}B} &\Fil_{K,B}^{\varphi,N}.
	}
	$$
	Moreover, similar statements hold true for $\Vect^{L/K,\st}$ and $\Vect^{\dR}$.
\end{lemma}
\begin{proof}
	It follows from Lemma~\ref{lem:dR basechange filtered} and Lemma~\ref{lem:log-crys basechange}.
\end{proof}


\section{Classification of $G_K$-equivariant line bundles}
In this section, we provide a classification of line bundles over $X_{\Cbb_p,\Acal}/G_K$ as an application of the $p$-adic monodromy theorem.
As in the previous section, we fix compatible embeddings $\overline{\Qbb_p}\hookrightarrow B_{\dR}^+$ and $\breve{\Qbb}_p\hookrightarrow \tilde{B}_{\Cbb_p}^{I}$ for closed intervals $I \subset (0,\infty)$.

First, we recall the construction of line bundles over $X_{\Cbb_p,\Acal}/G_K$ from continuous characters of $K^{\times}$ introduced in \cite{Nak13,KPX14,Mikami24}.
Let $\Acal=(A,A^+)_{\square}$ be an algebraic-affinoid analytic $K_{0,\square}$-algebra.
We write $\varphi\in \Gal(K_0/\Qbb_p)$ for the arithmetic Frobenius automorphism.
We set $f=[K_0:\Qbb_p].$
Then for each closed interval $I\subset (0,\infty)$, we have 
\begin{align}\label{K_0 decomposition}
	\tilde{B}_{\Cbb_p,A}^{I}\cong \prod_{i=0}^{f-1} (\tilde{B}_{\Cbb_p}^I\otimes_{K_0} \varphi^{i}_*A),
\end{align}
where $\varphi^{i}_*A$ is the $K_0$-algebra whose structure morphism is given by $K_0\xrightarrow{\varphi^i}K_0\to A$.
Under this isomorphism, the Frobenius morphism $\varphi\colon \tilde{B}_{\Cbb_p,A}^{I} \to \tilde{B}_{\Cbb_p,A}^{p^{-1}I}$ corresponds to
$$
\begin{array}{ccl}
    \prod_{i=0}^{f-1} (\tilde{B}_{\Cbb_p}^I\otimes_{K_0} \varphi^{i}_*A)  &\lra &\prod_{i=0}^{f-1} (\tilde{B}_{\Cbb_p}^{p^{-1}I}\otimes_{K_0} \varphi^{i}_*A) \\
    \rotatebox{90}{$\in$}       &                &  \quad\quad\rotatebox{90}{$\in$}\\
    (x_0,\ldots,x_{f-2},x_{f-1})&\longmapsto &(\varphi(x_{f-1}),\varphi(x_0),\ldots,\varphi(x_{f-2})).
\end{array}
$$          

\begin{definition}[{\cite[Definition 5.7]{Mikami24}}]\label{def:character type}
	Let $\pi$ be a uniformizer of $K$.
	For a continuous character $\delta\colon K^{\times}\to A^{\times}$, we define $\Ocal_{X_{\Cbb_p,\Acal}}(\delta)\in \Vect(X_{\Cbb_p,\Acal}/G_K)$ as follows:
	The choice of $\pi$ gives the decomposition $K^{\times}=\pi^{\Zbb}\times \Ocal_K^{\times}$, and local class field theory provides a continuous character
	$$\delta_{\pi}\colon G_K\to G_K^{\ab}\to K^{\times}/\pi^{\Zbb}=\Ocal_K^{\times}\xrightarrow{\delta}A^{\times}.$$
	Using this character, we define $\Ocal_{X_{\Cbb_p,\Acal}}(\delta)$ as the family $\{\tilde{B}_{\Cbb_p,A}^I\cdot e\}_I$, where $G_K$ acts on $e$ via $\delta_{\pi}$ and the Frobenius $\varphi$ acts on $e$ by setting $\varphi(e)=(\delta(\pi),1,\ldots,1)e\in \left(\prod_{i=0}^{f-1} (\tilde{B}_{\Cbb_p}^I\otimes_{K_0} \varphi^{i}_*A)\right)e$, using the isomorphism \eqref{K_0 decomposition}.
	This definition does not depend on the choice of $\pi$ by \cite[Theorem 5.9]{Mikami24}.
\end{definition}

\begin{definition}[{\cite[Definition 5.10]{Mikami24}}]
A $G_K$-equivariant line bundle $V$ over $X_{\Cbb_p,\Acal}$ is said to be \textit{of character type} if there exists a continuous character $\delta\colon K^{\times}\to A^{\times}$ and a finite projective $\Acal$-module $M$ of rank $1$ such that $V\cong \Ocal_{X_{\Cbb_p,\Acal}}(\delta)\otimes_{\Acal}M.$
\end{definition}

The goal of this section is to prove the following theorem.
\begin{theorem}\label{thm:classification of line bundles}
	Let $V$ be a $G_K$-equivariant line bundle over $X_{\Cbb_p,\Acal}$.
	Then there exist a unique continuous character $\delta\colon K^{\times}\to A^{\times}$ and a unique finite projective $\Acal$-module $M$ of rank $1$ such that $V\cong \Ocal_{X_{\Cbb_p,\Acal}}(\delta)\otimes_{\Acal}M.$
\end{theorem}
\begin{remark}
	In \cite[Theorem 5.18]{Mikami24}, the author proved the above classification under a certain freeness condition on $V$.
	The present result improves upon this by dropping the assumption.
\end{remark}

Let $E/K$ be a finite extension such that $[K:\Qbb_p]=\#\Hom_{\Alg_{\Qbb_p}}(K,E)$.
\begin{lemma}
	If Theorem~\ref{thm:classification of line bundles} holds true for any algebraic-affinoid analytic $E_{\square}$-algebra, then it holds true for any algebraic-affinoid analytic $K_{0,\square}$-algebra.
\end{lemma}
\begin{proof}
	Let $\Acal=(A,A^+)_{\square}$ be an algebraic-affinoid analytic $K_{0,\square}$-algebra.
	First, we prove that for a continuous character $\delta\colon K^{\times}\to A^{\times}$ and a finite projective $\Acal$-module $M$ of rank $1$, if $\Ocal_{X_{\Cbb_p,\Acal}}\cong \Ocal_{X_{\Cbb_p,\Acal}}(\delta)\otimes_{\Acal}M$, then $\delta$ is trivial and $M\cong A$.
	Applying Theorem~\ref{thm:classification of line bundles} for the algebraic-affinoid analytic $E_{\square}$-algebra $E_{\square}\otimes_{K_{0,\square}}\Acal$, the character $K^{\times}\xrightarrow{\delta} A^{\times}\to (E\otimes_{K_0}A)^{\times}$ is trivial.
	Since $A^{\times}\to (E\otimes_{K_0}A)^{\times}$ is injective, $\delta$ is also trivial.
	Moreover, there is an isomorphism
	$$A\cong H^0(X_{\Cbb_p,\Acal}/G_K, \Ocal_{X_{\Cbb_p,\Acal}})\cong H^0(X_{\Cbb_p,\Acal}/G_K, \Ocal_{X_{\Cbb_p,\Acal}}\otimes_{\Acal} M)\cong M.$$
	Thus, we get the uniqueness.
	The existence follows from \cite[Lemma 5.14]{Mikami24}.
\end{proof}

In the following, let $\Acal=(A,A^+)_{\square}$ be an algebraic-affinoid analytic $E_{\square}$-algebra.
We set $\Sigma=\Hom_{\Alg_{\Qbb_p}}(K,E)$.
Then there is a natural isomorphism
$$K\otimes A\cong \prod_{\sigma \in \Sigma}A_{\sigma},$$
where $A_{\sigma}=A$ for each $\sigma \in \Sigma$, and the subscript $\sigma$ is used merely to distinguish the components.

\begin{definition}
Let $\delta \colon K^{\times}\to A^{\times}$ be a continuous character.
Then it is locally analytic\footnote{When $A$ is an affinoid $L$-algebra, then $\delta$ is locally analytic by \cite[Proposition 8.3]{Buzz07}. In general, it follows from \cite[Proposition 5.2]{Mikami24}.}.
Therefore, we obtain a $\Qbb_p$-linear morphism $d\delta\colon \Lie K^{\times}\cong K\to A$, where the isomorphism $\Lie K^{\times}\cong K$ is induced by the $p$-adic logarithmic morphism.
We define $\wt(\sigma)=(\wt_{\sigma}(\delta))_{\sigma}\in \prod_{\sigma\in \Sigma}A\cong K\otimes A$ so that 
$d\delta=\sum_{\sigma} \wt_{\sigma}(\delta)\sigma$ in $\Hom_{\Qbb_p}(K,A)$.
Since $\Sigma$ is a basis of the $A$-module $\Hom_{\Qbb_p}(K,A)$, this definition is well-defined.
\end{definition}

\begin{proposition}\label{prop:Sen polynomial character}
	Let $\delta \colon K^{\times}\to A^{\times}$ be a continuous character.
	Then the Sen polynomial of $\Ocal_{X_{\Cbb_p,\Acal}}(\delta)$ is equal to $T-\wt(\delta)\in (K\otimes A)[T]$.
\end{proposition}
\begin{proof}
	Using \cite[Proposition 5.2]{Mikami24}, we may assume that $A$ is an affinoid $L$-algebra.
	Then the claim follows from \cite[Lemma 6.2.12]{KPX14}.
\end{proof}

First, we prove the uniqueness\footnote{It was already proved in \cite[Lemma 5.12]{Mikami24} based on \cite[Theorem 6.2.14]{KPX14}. Here, we give an alternative proof.}.
\begin{proof}[Proof of the uniqueness in Theorem~\ref{thm:classification of line bundles}]
	We prove that for a continuous character $\delta\colon K^{\times}\to A^{\times}$ and a finite projective $\Acal$-module $M$ of rank $1$, if $\Ocal_{X_{\Cbb_p,\Acal}}\cong \Ocal_{X_{\Cbb_p,\Acal}}(\delta)\otimes_{\Acal}M$, then $\delta$ is trivial and $M\cong A$.
	Since the Sen polynomial of $\Ocal_{X_{\Cbb_p,\Acal}}$ is given by $T\in (K\otimes A)[T]$, we get $\wt(\delta)=0$ by Proposition~\ref{prop:Sen polynomial character}.
	Therefore, the character $\delta$ is locally constant.
	We fix a uniformizer $\pi$ of $K$.
	Let $L/K$ be the finite totally ramified abelian extension corresponding to the kernel $\Ker(K^{\times}\to K^{\times}/\pi^{\Zbb}\cong \Ocal_{K}^{\times}\xrightarrow{\delta}A^{\times})$ via local class field theory.
	Note that we can regard $\delta|_{\Ocal_K^{\times}}$ as a character of $\Gal(L/K)$.
	Then $\Ocal_{X_{\Cbb_p,\Acal}}(\delta)\otimes_{\Acal}M$ is semistable as a $G_L$-equivariant line bundle.
	Moreover, under the identification $K_0\otimes A\cong \prod_{i=0}^{f-1} \varphi_*^iA$, the finite projective $K_0\otimes A$-module\footnote{Since $L/K$ is totally ramified, we have $L_0=K_0$.} $D^L_{\st}(\Ocal_{X_{\Cbb_p,\Acal}}(\delta)\otimes_{\Acal}M)$ is isomorphic to $\prod_{i=0}^{f-1}\varphi_*^iM.$
	The action of $\Gal(L/K)$ on $D^L_{\st}(\Ocal_{X_{\Cbb_p,\Acal}}(\delta)\otimes_{\Acal}M)$ is given by $\Gal(L/K)\xrightarrow{\delta} A^{\times}\to (K_0\otimes A)^{\times}$, and the (linear) action of $\varphi^f$ on $D^L_{\st}(\Ocal_{X_{\Cbb_p,\Acal}}(\delta)\otimes_{\Acal}M)$ is given by the multiplication by $\delta(\pi)$.
	On the other hand, $D_{\st}^L(\Ocal_{X_{\Cbb_p,\Acal}})$ is trivial.
	Therefore, we obtain that $M\cong A$, $\delta(\pi)=1$, and $\delta\colon \Gal(L/K)\to A^{\times}$ is trivial.
	The latter two conditions imply that $\delta\colon K^{\times}\to A^{\times}$ is trivial, so we get the claim.
\end{proof}

Next, we prove the existence in Theorem~\ref{thm:classification of line bundles}.
By the proof of Lemma~\ref{lem:exp nilpotent}, the monodromy operator $N$ is automatically zero in the case of rank $1$, so we will not explicitly mention $N$ in what follows.

\begin{lemma}\label{lem:unramified character}
	Let $\delta\colon K^{\times}\to A^{\times}$ be a continuous character such that $\delta\vert_{\Ocal_{K}^{\times}}$ is trivial.
	Then $\Ocal_{X_{\Cbb_p,\Acal}}(\delta)$ is semistable.
	Moreover, $(D_{\st}(\Ocal_{X_{\Cbb_p,\Acal}}(\delta)),\Fil^{\bullet}D_{\dR}(\Ocal_{X_{\Cbb_p,\Acal}}(\delta)))\in \Fil_{K,A}^{\varphi,N}$ is described as follows:
	Let $\pi\in K$ be a uniformizer, and set $a=\delta(\pi)$ (which is independent of the choice of $\pi$).
	Then $D_{\st}(\Ocal_{X_{\Cbb_p,\Acal}}(\delta))\cong (K_0\otimes A)e$, where $\varphi(e)=(a,1,\ldots,1)e \in (K_0\otimes A)e\cong \left(\prod_{i=0}^{f-1}\varphi_*^iA\right)e$, and the filtration on $D_{\dR}(\Ocal_{X_{\Cbb_p,\Acal}}(\delta))$ is trivial (i.e., $\Fil^0 D_{\dR}=D_{\dR}$ and $\Fil^1D_{\dR}=0$).
\end{lemma}
\begin{proof}
	It directly follows from the construction of $\Ocal_{X_{\Cbb_p,\Acal}}(\delta)$.
\end{proof}

\begin{lemma}\label{lem:LT character}
	Each $\sigma \in \Sigma$ induces the continuous character $\sigma\colon K^{\times}\to E^{\times}$.
	Then  $\Ocal_{X_{\Cbb_p,\Acal}}(\sigma)$ is semistable.
	Moreover, we have $D_{\st}(\Ocal_{X_{\Cbb_p,\Acal}}(\sigma))\cong (K_0\otimes A)e$, where $\varphi(e)=e$. 
	Under the decomposition $D_{\dR}(\Ocal_{X_{\Cbb_p,\Acal}}(\sigma))\cong (K\otimes A)e\cong \left(\prod_{\sigma'\in \Sigma}A_{\sigma'}\right)e$, the filtration is described as follows: for $\sigma'\neq \sigma$, the filtration on $A_{\sigma'}e$ is trivial, and for $\sigma$, it is given by $\Fil^{-1}A_{\sigma}e=A_{\sigma}e$ and $\Fil^0A_{\sigma}e=0$.
\end{lemma}
\begin{proof}
	This follows from Lemma~\ref{lem:unramified character} and \cite[Example 6.2.6 (1)]{KPX14}.
\end{proof}

\begin{lemma}\label{lem:descent to ab ext}
For $(D^{\unr}, \Fil^{\bullet}D)\in \Fil_{K,A}^{\varphi,N,\Gal_K}$ of constant rank $1$, there exists a finite totally ramified abelian extension $L/K$ such that $(D^{\unr}, \Fil^{\bullet}D)$ lies in the essential image of $\iota_L\colon \Fil_{L/K,A}^{\varphi,N,\Gal(L/K)}\to \Fil_{K,A}^{\varphi,N,\Gal_K}.$
\end{lemma}
\begin{proof}
Let $K^{\ab}$ denote the maximal abelian extension of $K$, and set $H=\Gal(\bar{K}/K^{\ab})$.
First, we prove that the action of $H$ on $D^{\unr}$ is trivial.
Let $W_K$ denote the Weil group of $K$.
Note that $H\subset W_K$ and $H=\Ker(W_K\to W_K^{\ab})$.
For $\sigma \in W_K$, we define $d(\sigma)\in \Zbb$ so that $\varphi^{-d(\sigma)}=\sigma\vert_{K_0^{\unr}},$ where $\varphi$ is the $p$th power Frobenius automorphism of $K_0^{\unr}$.
Then we obtain the $K_0^{\unr}\otimes A$-linear automorphism
$$\varphi^{d(\sigma)}\sigma \colon D^{\unr}\to D^{\unr}.$$
Since the actions of $\varphi$ and $W_K\subset G_K$ on $D^{\unr}$ commute, and $D^{\unr}$ is a finite projective $K_0^{\unr}\otimes A$-module of constant rank $1$, it defines a continuous character
$$W_K\to (K_0^{\unr}\otimes A)^{\times}.$$
Therefore, for $\sigma \in H$, the automorphism $\varphi^{d(\sigma)}\sigma$ on $D^{\unr}$ is trivial.
Since $d(\sigma)=0$ for $\sigma \in H$, we find that the action of $H$ on $D^{\unr}$ is trivial.
Since the action of $G_K^{\ab}=\Gal(K^{\ab}/K)$ on the finite projective $K_0^{\unr}\otimes A$-module $D^{\unr}$ is smooth, there exists a finite extension $L'/K^{\unr}$ such that $L'\subset K^{\ab}$ and the action of $G_{L'}$ on $D^{\unr}$ is trivial.
By local class field theory, there is a finite totally ramified abelian extension $L/K$ such that $L'\subset L^{\unr}$.
Then by the same argument as in the proof of Lemma~\ref{lemma:MF transition}, we find that $(D^{\unr}, \Fil^{\bullet}D)$ lies in the essential image of $\iota_L$.
\end{proof}

\begin{proposition}\label{prop:dR character type}
Let $V$ be a de Rham $G_K$-equivariant line bundle over $X_{\Cbb_p,\Acal}$.
Then $V$ is of character type.
\end{proposition}
\begin{proof}
	It is enough to show that for any $(D^{\unr}, \Fil^{\bullet}D)\in \Fil_{K,A}^{\varphi,N,\Gal_K}$ of constant rank $1$, $V(D^{\unr}, \Fil^{\bullet}D)$ is of character type.
	By Lemma~\ref{lem:descent to ab ext}, there exist a finite totally ramified abelian extension $L/K$ and $(D_{0},\Fil^{\bullet}D)\in \Fil_{L/K,A}^{\varphi,N,\Gal(L/K)}$ such that $\iota_L(D_{0},\Fil^{\bullet}D)=(D^{\unr}, \Fil^{\bullet}D)$.
	Let us prove that $V(D_{0},\Fil^{\bullet}D)$ is of character type.
	By Lemma~\ref{lem:LT character}, we may assume that the filtration $\Fil^{\bullet}D$ is trivial.
	We decompose the finite projective $K_0\otimes A$-module $D_0$ of rank $1$ into a product $D_0=\prod_{i=0}^{f-1} D_0^i$ of finite projective $\varphi_*^iA$-modules of rank $1$.
	We set $M=D_0^0$.
	Since $\varphi$ acts transitively on the factors of the product, we get 
	an isomorphism $D_0\cong \prod_{i=0}^{f-1} \varphi_*^iM,$ where the action of $\varphi$ on $D_0$ is given by
	$$(x_0,\ldots,x_{f-2},x_{f-1})\mapsto (ax_{f-1}, x_0,\ldots, x_{f-2})$$
	for some $a\in A^{\times}$.
	The linear action of $\Gal(L/K)$ on $M$ induces a character $\delta_0\colon \Gal(L/K)\to A^{\times}$.
	Since the action of $\Gal(L/K)$ on $D_0$ commutes with the action of $\varphi$, $\Gal(L/K)$ acts on $D_0$ diagonally via $\delta_0$.
	In other words, $\Gal(L/K)$ acts on the finite projective $K_0\otimes A$-module $D_0$ of rank $1$ via the character $\Gal(L/K)\xrightarrow{\delta_0}A^{\times}\to (K_0\otimes A)^{\times}.$
	By local class field theory, we get the locally constant character $\Ocal_K^{\times}\to \Gal(L/K)\xrightarrow{\delta_0} A^{\times},$ 
	which, by abuse of notation, we also denote by $\delta_0$.
	We take a uniformizer $\pi$ of $K$ from the kernel $\Ker(K^{\times}\cong W_K\to \Gal(L/K))$, and define the locally constant character $\delta\colon K^{\times}\to A^{\times}$ by setting $\delta(\pi)=a$ and $\delta\vert_{\Ocal_K^{\times}}=\delta_0$.
	Then by constructions of $V(D_{0},\Fil^{\bullet}D)$ and $\Ocal_{X_{\Cbb_p,\Acal}}(\delta)$, we obtain $V(D_{0},\Fil^{\bullet}D)\cong \Ocal_{X_{\Cbb_p,\Acal}}(\delta)\otimes_{\Acal}M.$
\end{proof}

\begin{lemma}\label{lem:wt character}
	For $a=(a_{\sigma})_{\sigma}\in \prod_{\sigma\in \Sigma}A^b_{\sigma}\cong K\otimes A^b$, there exist a finite \'{e}tale faithfully flat $A$-algebra $B$ and a continuous character $\delta\colon K^{\times}\to B^{\times}$ such that $\wt(\delta)=a$.
\end{lemma}
\begin{proof}
	By \cite[Lemma 2.17 (4)]{Mikami24}, there exists an affinoid $E$-algebra $A'$ of definition of $A$ such that $a_{\sigma}\in A'_{\sigma}$ for any $\sigma \in \Sigma$.
	Therefore, by replacing $A$ with $A'$, we may assume that $A$ is an affinoid $E$-algebra.

    We fix a uniformizer $\pi$ of $K$.
    Recall that the multiplicative group $K^{\times}$ decomposes as $\pi^{\Zbb} \times \mu \times \Gamma$, where $\mu$ is the finite group of roots of unity in $K$ and $\Gamma\cong \Zbb_p^d$ is a free $\Zbb_p$-module of rank $d = [K:\Qbb_p]$.
    We view $a$ as the $\Qbb_p$-linear map $\sum_{\sigma\in \Sigma} a_{\sigma}\sigma \colon K\to A$.
    For an integer $m \geq 0$, let $U_m$ be the open subgroup of $\Gamma$ corresponding to $(p^m \Zbb_p)^d$ under the identification $\Gamma\cong \Zbb_p^d$.
    For $m$ sufficiently large, the $p$-adic logarithm $\log \colon U_m \to K$ is well-defined, and for any $x \in U_m$, the series $\exp(a(\log x)) = \sum_{k=0}^{\infty} \frac{a(\log x)^k}{k!}$ converges in $A^{\times}$.
    We fix such an integer $m$.
    Then the map $\delta \colon U_m \to A^{\times}$ given by $x \mapsto \exp(a(\log x))$ is a continuous character of weight $a$.

    Let $\gamma_1, \ldots, \gamma_d$ be topological generators of $\Gamma$.
    Then $U_m$ is topologically generated by $\gamma_1^{p^m}, \ldots, \gamma_d^{p^m}$.
    We set $u_i = \delta(\gamma_i^{p^m}) \in A^{\times}$ for each $1\le i \le d$.
    We define the $A$-algebra $B$ as
    $$B = A[X_1, \ldots, X_d] / (X_1^{p^m} - u_1, \ldots, X_d^{p^m} - u_d),$$
    which is a finite \'{e}tale faithfully flat $A$-algebra.
    In $B$, each $X_i$ is a $p^m$th root of $u_i$.
    Therefore, we can extend $\delta$ to a continuous character on $\Gamma$ by setting $\gamma_i \mapsto X_i$.
    Finally, we extend this to a continuous character $\delta \colon K^{\times} \to B^{\times}$ by setting $\delta\vert_{\mu} = 1$ and $\delta(\pi) = 1$.
    By construction, we have $\wt(\delta)=a$, which proves the claim.
\end{proof}

\begin{proof}[Proof of Theorem~\ref{thm:classification of line bundles}]
	It remains to prove that every $G_K$-equivariant line bundle $V$ over $X_{\Cbb_p,\Acal}$ is of character type. 
	Let $T-a\in (K\otimes A)[T]$ be the Sen polynomial of $V$.
	By Theorem~\ref{thm:HTS wts are analytic}, we have $a\in K\otimes A^b$.
	Using Lemma~\ref{lem:wt character}, we obtain a finite faithfully flat $A$-algebra $B$ and a continuous character $\delta_0\colon K^{\times}\to B^{\times}$ such that $\wt(\delta_0)=a$.
	By \cite[Lemma 5.14]{Mikami24}, we may replace $A$ with $B$.
	Then the Sen polynomial of $V\otimes \Ocal_{X_{\Cbb_p,\Acal}}(\delta_0^{-1})$ is given by $T$.
	Therefore, $V\otimes \Ocal_{X_{\Cbb_p,\Acal}}(\delta_0^{-1})$ is Hodge-Tate of Hodge-Tate weight $0$.
	By Corollary~\ref{cor:HT of HT wt 0 implies dR}, it is also de Rham.
	Therefore, by Proposition~\ref{prop:dR character type}, $V\otimes \Ocal_{X_{\Cbb_p,\Acal}}(\delta_0^{-1})$ is of character type.
	Thus, $V$ is also of character type.
\end{proof}

\bibliography{p-adic_monodromy}
\end{document}